\documentclass[numbers=enddot,12pt,final,onecolumn,notitlepage]{scrartcl}%
\usepackage[headsepline,footsepline,manualmark]{scrlayer-scrpage}
\usepackage{amsfonts}
\usepackage{amssymb}
\usepackage{amsmath}
\usepackage{amsthm}
\usepackage{framed}
\usepackage{comment}
\usepackage{color}
\usepackage[breaklinks=true]{hyperref}
\usepackage[sc]{mathpazo}
\usepackage[T1]{fontenc}
\usepackage{needspace}
\usepackage{ytableau}
\usepackage{hyperxmp}
\providecommand{\U}[1]{\protect\rule{.1in}{.1in}}
\theoremstyle{definition}
\newtheorem{theo}{Theorem}[section]
\newenvironment{theorem}[1][]
{\begin{theo}[#1]\begin{leftbar}}
{\end{leftbar}\end{theo}}
\newtheorem{lem}[theo]{Lemma}
\newenvironment{lemma}[1][]
{\begin{lem}[#1]\begin{leftbar}}
{\end{leftbar}\end{lem}}
\newtheorem{prop}[theo]{Proposition}
\newenvironment{proposition}[1][]
{\begin{prop}[#1]\begin{leftbar}}
{\end{leftbar}\end{prop}}
\newtheorem{defi}[theo]{Definition}
\newenvironment{definition}[1][]
{\begin{defi}[#1]\begin{leftbar}}
{\end{leftbar}\end{defi}}
\newtheorem{remk}[theo]{Remark}

\newtheorem{coro}[theo]{Corollary}
\newenvironment{corollary}[1][]
{\begin{coro}[#1]\begin{leftbar}}
{\end{leftbar}\end{coro}}
\newtheorem{conv}[theo]{Convention}

\newtheorem{quest}[theo]{Question}
\newenvironment{question}[1][]
{\begin{quest}[#1]\begin{leftbar}}
{\end{leftbar}\end{quest}}
\newtheorem{warn}[theo]{Warning}

\newtheorem{conj}[theo]{Conjecture}

\newtheorem{exam}[theo]{Example}
\newenvironment{example}[1][]
{\begin{exam}[#1]\begin{leftbar}}
{\end{leftbar}\end{exam}}
\newcommand{\silentsection}{\section}

\newenvironment{noncompile}{}{}
\excludecomment{verlong}
\includecomment{vershort}
\excludecomment{noncompile}

\let\sumnonlimits\sum
\let\prodnonlimits\prod
\renewcommand{\sum}{\sumnonlimits\limits}
\renewcommand{\prod}{\prodnonlimits\limits}

\ihead{Representations of somewhere-to-below shuffles, version 2026-04-17}
\ohead{page \thepage}
\begin{document}

\title{The representation theory of somewhere-to-below shuffles}
\author{Darij Grinberg}
\date{version 3, April 17, 2026}
\maketitle

\begin{abstract}
\textbf{Abstract.} The \emph{somewhere-to-below shuffles} are the elements%
\[
t_{\ell}:=\operatorname*{cyc}\nolimits_{\ell}+\operatorname*{cyc}%
\nolimits_{\ell,\ell+1}+\operatorname*{cyc}\nolimits_{\ell,\ell+1,\ell
+2}+\cdots+\operatorname*{cyc}\nolimits_{\ell,\ell+1,\ldots,n}%
\]
(for $\ell\in\left\{  1,2,\ldots,n\right\}  $) in the group algebra
$\mathbf{k}\left[  S_{n}\right]  $ of the $n$-th symmetric group $S_{n}$.
Their linear combinations are called the \emph{one-sided cycle shuffles}. We
determine the eigenvalues of the action of any one-sided cycle shuffle on any
Specht module $\mathcal{S}^{\lambda}$ of $S_{n}$. \medskip

 \medskip


\end{abstract}
\tableofcontents

\section{Introduction}

This paper is a continuation of \cite{s2b1} with other means. Specifically,
our goal here is to answer some natural representation-theoretical questions
around the somewhere-to-below shuffles in the symmetric group algebra
(including \cite[Question 16.12]{s2b1}).

We recall that the \emph{somewhere-to-below shuffles} are $n$ special elements
$t_{1},t_{2},\ldots,t_{n}$ of the group algebra $\mathbf{k}\left[
S_{n}\right]  $ of a symmetric group $S_{n}$ over an arbitrary commutative
ring $\mathbf{k}$; they are defined by
\[
t_{\ell}:=\operatorname*{cyc}\nolimits_{\ell}+\operatorname*{cyc}%
\nolimits_{\ell,\ell+1}+\operatorname*{cyc}\nolimits_{\ell,\ell+1,\ell
+2}+\cdots+\operatorname*{cyc}\nolimits_{\ell,\ell+1,\ldots,n}\in
\mathbf{k}\left[  S_{n}\right]  ,
\]
where $\operatorname*{cyc}\nolimits_{\ell,\ell+1,\ldots,k}$ denotes the cycle
that sends $\ell\mapsto\ell+1\mapsto\ell+2\mapsto\cdots\mapsto k\mapsto\ell$
(and leaves all remaining elements of $\left[  n\right]  =\left\{
1,2,\ldots,n\right\}  $ unchanged). Together with their linear combinations
(called the \emph{one-sided cycle shuffles}), they have been introduced and
studied in the paper \cite{s2b1} (published with abridgements as
\cite{s2b1-pub}\footnote{The numbering of results in \cite{s2b1} and in
\cite{s2b1-pub} is identical except for Section 9, so the reader can consult
either version.}) by Lafreni\`{e}re and the present author. One of the main
results is \cite[Theorem 11.1]{s2b1}, which constructs a basis $\left(
a_{w}\right)  _{w\in S_{n}}$ of $\mathbf{k}\left[  S_{n}\right]  $ on which
each of the shuffles $t_{1},t_{2},\ldots,t_{n}$ acts (by right multiplication)
triangularly -- i.e., which satisfies%
\[
a_{w}t_{\ell}\in\operatorname*{span}\left\{  a_{v}\ \mid\ v\leq w\right\}
\ \ \ \ \ \ \ \ \ \ \text{for all }w\in S_{n}\text{ and }\ell\in\left\{
1,2,\ldots,n\right\}
\]
(for an appropriate total order $<$ on $S_{n}$). This entails that the
shuffles $t_{1},t_{2},\ldots,t_{n}$ and their linear combinations have integer
eigenvalues; these eigenvalues have indeed been found (\cite[\S 12]{s2b1})
along with their multiplicities (\cite[\S 13]{s2b1}). As a further
consequence, the $\mathbf{k}$-subalgebra of $\mathbf{k}\left[  S_{n}\right]  $
generated by $t_{1},t_{2},\ldots,t_{n}$ is isomorphic to an algebra of
upper-triangular matrices, and the commutators $\left[  t_{i},t_{j}\right]
:=t_{i}t_{j}-t_{j}t_{i}$ are nilpotent; a followup work \cite{s2b2} proves
even stronger claims.

However, like any elements of the group algebra $\mathbf{k}\left[
S_{n}\right]  $, the shuffles $t_{1},t_{2},\ldots,t_{n}$ act not just on the
whole algebra $\mathbf{k}\left[  S_{n}\right]  $, but on any of its modules,
i.e., on any representation of $S_{n}$. Thus, the question about eigenvalues
can be asked for each representation of $S_{n}$, in particular for the
\emph{Specht modules} (which are the irreducible representations of $S_{n}$,
at least in characteristic $0$).

The main goal of this paper is to answer this latter question. Let us give a
quick outline of the answer (which was announced in \cite[\S 11]%
{fps2024sn})\footnote{The formulation in \cite[\S 11]{fps2024sn} uses the
Frobenius characteristic map, but this has turned out to be a red herring.}:

We shall use some basic notions from the representation theory of $S_{n}$ and
from symmetric functions; the reader can find all prerequisites in
\cite[Chapters 6 and 7]{Fulton97}. For any partition $\lambda$ of $n$, a
Specht module $\mathcal{S}^{\lambda}$ is defined, which is a representation of
$S_{n}$ with a basis indexed by standard tableaux of shape $\lambda$. (In
\cite{Fulton97}, it is called $S^{\lambda}$.) This $S_{n}$-module
$\mathcal{S}^{\lambda}$ is irreducible when $\mathbf{k}$ has characteristic
$0$. Each $u\in\mathbf{k}\left[  S_{n}\right]  $ acts (on the left) on this
Specht module $\mathcal{S}^{\lambda}$; we let $L_{\lambda}\left(  u\right)  $
denote this action (viewed as a $\mathbf{k}$-module endomorphism of
$\mathcal{S}^{\lambda}$).

\begin{noncompile}
We let $\mathcal{R}$ denote the (inductive) representation ring of the
symmetric groups (called $R$ in \cite[\S 7.3]{Fulton97}), and $\Lambda$ denote
the ring of symmetric functions over $\mathbb{Z}$ (defined in \cite[\S 6.2]%
{Fulton97}). The \emph{Frobenius characteristic map} is the ring isomorphism
$\varphi:\Lambda\rightarrow\mathcal{R}$ defined in \cite[\S 7.3]{Fulton97},
and the famous \emph{Schur function} $s_{\lambda}\in\Lambda$ corresponding to
a partition $\lambda$ is the preimage of the Specht module $\mathcal{S}%
^{\lambda}$ (defined over a field of characteristic $0$) under this
isomorphism $\varphi$.
\end{noncompile}

We let $\Lambda$ denote the ring of symmetric functions over $\mathbb{Z}$
(defined in \cite[\S 6.2]{Fulton97}). We recall that it has a basis $\left(
s_{\lambda}\right)  _{\lambda\text{ is a partition}}$ of \emph{Schur
functions} $s_{\lambda}$.

For each $m\in\mathbb{N}$, we let $h_{m}\in\Lambda$ denote the $m$-th complete
homogeneous symmetric function. For each $m>1$, we let $z_{m}\in\Lambda$
denote the Schur function
\[
z_{m}:=s_{\left(  m-1,1\right)  }=h_{m-1}h_{1}-h_{m}\in\Lambda.
\]

A set of integers is called \emph{lacunar} if it contains no two consecutive
integers. For each lacunar subset $I$ of $\left[  n-1\right]  $, we define a
symmetric function%
\[
z_{I}:=h_{i_{1}-1}\prod_{j=2}^{m}z_{i_{j}-i_{j-1}}\in\Lambda,
\]
where $i_{1},i_{2},\ldots,i_{m}$ are the elements of $I\cup\left\{
n+1\right\}  $ in increasing order (so that $i_{m}=n+1$ and $I=\left\{
i_{1}<i_{2}<\cdots<i_{m-1}\right\}  $). When this symmetric function $z_{I}$
is expanded in the basis $\left(  s_{\lambda}\right)  _{\lambda\text{ is a
partition}}$ of $\Lambda$, the coefficient of a given Schur function
$s_{\lambda}$ shall be called $c_{I}^{\lambda}$. This coefficient
$c_{I}^{\lambda}$ is actually a Littlewood--Richardson coefficient (since
$z_{I}$ is a skew Schur function), hence a nonnegative integer.

We now claim the following:

\begin{theorem}
[part of Theorem \ref{thm.eigs}]\label{thm.intro.eigs}Let $\lambda$ be a
partition. Let $\omega_{1},\omega_{2},\ldots,\omega_{n}\in\mathbf{k}$. Then,
the eigenvalues of the operator $L_{\lambda}\left(  \omega_{1}t_{1}+\omega
_{2}t_{2}+\cdots+\omega_{n}t_{n}\right)  $ on the Specht module $\mathcal{S}%
^{\lambda}$ are the linear combinations%
\[
\omega_{1}m_{I,1}+\omega_{2}m_{I,2}+\cdots+\omega_{n}m_{I,n}%
\ \ \ \ \ \ \ \ \ \ \text{for }I\subseteq\left[  n-1\right]  \text{ lacunar
satisfying }c_{I}^{\lambda}\neq0,
\]
where the $m_{I,k}$ are certain nonnegative integers defined combinatorially
(namely, $m_{I,k}$ is the distance between $k$ and the smallest element of
$I\cup\left\{  n+1\right\}  $ that is $\geq k$). The algebraic multiplicities
of these eigenvalues are the $c_{I}^{\lambda}$ in the generic case (i.e., if
no two $I$'s produce the same linear combination; otherwise the multiplicities
of colliding eigenvalues should be added together). Moreover, if all these
linear combinations are distinct, then $L_{\lambda}\left(  \omega_{1}%
t_{1}+\omega_{2}t_{2}+\cdots+\omega_{n}t_{n}\right)  $ is diagonalizable.
\end{theorem}

The proof of this theorem will rely on the filtration $0=F_{0}\subseteq
F_{1}\subseteq F_{2}\subseteq\cdots\subseteq F_{f_{n+1}}=\mathbf{k}\left[
S_{n}\right]  $ of $\mathbf{k}\left[  S_{n}\right]  $ introduced in
\cite[\S 8.1]{s2b1}. We call this the \emph{Fibonacci filtration} of
$\mathbf{k}\left[  S_{n}\right]  $, as its length $f_{n+1}$ is the $\left(
n+1\right)  $-st Fibonacci number. We note that this filtration is not
completely canonical, as it depends on the choice of a listing $Q_{1}%
,Q_{2},\ldots,Q_{f_{n+1}}$ of all lacunar subsets of $\left[  n-1\right]  $ in
the order of increasing sum of elements (the ties can be broken arbitrarily,
whence the non-canonicity). Much about this filtration was already understood
in \cite{s2b1}, but we will need some additional information about the action
of $S_{n}$ on its subquotients $F_{i}/F_{i-1}$:

Let $\mathcal{A}$ be the $\mathbf{k}$-algebra $\mathbf{k}\left[  S_{n}\right]
$, and let $\mathcal{T}$ be its $\mathbf{k}$-subalgebra generated by
$t_{1},t_{2},\ldots,t_{n}$. Then, each $F_{i}$ is a left ideal of
$\mathcal{A}$ but is also fixed under right multiplication by each $t_{\ell}$;
therefore, each $F_{i}$ is an $\left(  \mathcal{A},\mathcal{T}\right)
$-subbimodule of $\mathcal{A}$. Thus, each subquotient $F_{i}/F_{i-1}$ of the
Fibonacci filtration is an $\left(  \mathcal{A},\mathcal{T}\right)
$-bimodule. As a right $\mathcal{T}$-module, it is \emph{scalar} (meaning that
each $t_{\ell}$ acts on it by a scalar, which is in fact the integer
$m_{Q_{i},\ell}$ from \cite[Theorem 8.1 \textbf{(c)}]{s2b1}). As a left
$\mathcal{A}$-module (i.e., as a representation of $S_{n}$), we describe it
explicitly here:

\begin{theorem}
[part of Theorem \ref{thm.Fi/Fi-1.as-ind}]\label{thm.intro.Fi/Fi-1}Let
$i\in\left[  f_{n+1}\right]  $.

\begin{noncompile}
Then, the isomorphism class of the $S_{n}$-representation $F_{i}/F_{i-1}$ is
$\varphi\left(  z_{Q_{i}}\right)  $. More explicitly, and true in arbitrary
characteristic, we can describe it as follows:
\end{noncompile}

Consider the lacunar subset $Q_{i}$ of $\left[  n-1\right]  $ (from the above
listing $Q_{1},Q_{2},\ldots,Q_{f_{n+1}}$). Write the set $Q_{i}\cup\left\{
n+1\right\}  $ as $\left\{  i_{1}<i_{2}<\cdots<i_{m}\right\}  $, so that
$i_{m}=n+1$. Furthermore, set $i_{0}:=1$. Set $j_{k}:=i_{k}-i_{k-1}$ for each
$k\in\left[  m\right]  $. Note that $j_{1}\geq0$ and $j_{2},j_{3},\ldots
,j_{m}>1$ and $j_{1}+j_{2}+\cdots+j_{m}=i_{m}-i_{0}=n$.

For each $p\in\mathbb{N}$, we let $\mathcal{H}_{p}$ denote the trivial
$1$-dimensional representation of $S_{p}$ (that is, the $\mathbf{k}$-module
$\mathbf{k}$ on which $S_{p}$ acts trivially), and we let $\mathcal{Z}_{p}$
denote the reflection quotient representation of $S_{p}$ (that is, the free
$\mathbf{k}$-module $\mathbf{k}^{p}$ on which $S_{p}$ acts by permuting the
coordinates, divided by the submodule consisting of all vectors of the form
$\left(  a,a,\ldots,a\right)  \in\mathbf{k}^{p}$). Then,%
\[
F_{i}/F_{i-1}\cong\operatorname*{Ind}\nolimits_{S_{j_{1}}\times S_{j_{2}%
}\times\cdots\times S_{j_{m}}}^{S_{n}}\underbrace{\left(  \mathcal{H}_{j_{1}%
}\otimes\mathcal{Z}_{j_{2}}\otimes\mathcal{Z}_{j_{3}}\otimes\cdots
\otimes\mathcal{Z}_{j_{m}}\right)  }_{\substack{\text{the first tensorand is
an }\mathcal{H}\text{,}\\\text{while all others are }\mathcal{Z}\text{'s}}}
\]
as $S_{n}$-representations. Here, we embed $S_{j_{1}}\times S_{j_{2}}%
\times\cdots\times S_{j_{m}}$ into $S_{n}$ by the usual parabolic embedding
(since $j_{1}+j_{2}+\cdots+j_{m}=n$).
\end{theorem}

This theorem will be a crucial stepping stone on our way to Theorem
\ref{thm.intro.eigs}.

We note that neither of our two main results requires any assumption about the
characteristic of $\mathbf{k}$. However, in positive characteristic, care must
be taken to distinguish between the reflection quotient representation
$\mathcal{Z}_{p}$ in Theorem \ref{thm.intro.Fi/Fi-1} and the reflection
subrepresentation $\mathcal{R}_{p}$ (which consists of the zero-sum vectors in
$\mathbf{k}^{p}$); the two representations have the same dimension $p-1$ (for
$p\geq1$), but are not isomorphic unless $\operatorname*{char}\mathbf{k}\neq
p$ or $p\leq2$.

\begin{noncompile}
It is worth noticing that the diagonalizability claim in Theorem
\ref{thm.intro.eigs} is an improvement upon the diagonalizability claim in
\cite[Theorem 12.3]{s2b1}. Indeed, if $\mathbf{k}$ is a field of
characteristic $0$, and if each partition $\lambda$ of $n$ has the property
that the scalars%
\[
\omega_{1}m_{I,1}+\omega_{2}m_{I,2}+\cdots+\omega_{n}m_{I,n}%
\ \ \ \ \ \ \ \ \ \ \text{for }I\subseteq\left[  n-1\right]  \text{ lacunar
satisfying }c_{I}^{\lambda}\neq0
\]
are distinct, then it follows that the action of $\omega_{1}t_{1}+\omega
_{2}t_{2}+\cdots+\omega_{n}t_{n}$ on each Specht module $\mathcal{S}^{\lambda
}$ is diagonalizable, and therefore so is the action of $\omega_{1}%
t_{1}+\omega_{2}t_{2}+\cdots+\omega_{n}t_{n}$ on $\mathcal{A}$ itself (since
$\mathcal{A}$ can be decomposed into a direct sum of Specht modules). In
\cite[Theorem 12.3]{s2b1}, this is proved only under the stronger condition
that the scalars%
\[
\omega_{1}m_{I,1}+\omega_{2}m_{I,2}+\cdots+\omega_{n}m_{I,n}%
\ \ \ \ \ \ \ \ \ \ \text{for }I\subseteq\left[  n-1\right]  \text{ lacunar}%
\]
are all distinct. However, I don't know of any specific cases in which this
helps (e.g., it does not help with the one-sided cycle shuffle $\sum_{k=1}%
^{n}\dfrac{1}{n+1-k}t_{k}\in\mathbb{Q}\left[  S_{n}\right]  $ for $n=12$).
\end{noncompile}

We suspect that our results can be generalized (\textquotedblleft%
$q$-deformed\textquotedblright) from the symmetric group algebra to the Hecke
algebra $\mathcal{H}_{q}\left(  S_{n}\right)  $. Most results from \cite{s2b1}
can definitely be generalized this way, as will be detailed in forthcoming work.

\paragraph{Acknowledgements}

The author would like to thank Sarah Brauner, Nadia Lafreni\`{e}re, Martin
Lorenz, Victor Reiner, Sheila Sundaram and Mark Wildon for inspiring discussions,
as well as two referees for helpful commentary
(in particular the suggestion of Example~\ref{exa.Fi/Fi-1.L3.e1}).

\begin{noncompile}
The SageMath CAS \cite{sagemath} was indispensable at every stage of the
research presented here.
\end{noncompile}

\section{Definitions and notations}

\subsection{Basics}

We recall some notations from \cite{s2b1}.

Let $\mathbf{k}$ be any commutative ring. (We don't require that $\mathbf{k}$
is a field or a $\mathbb{Q}$-algebra, but the reader can think of
$\mathbf{k}=\mathbb{Q}$ as a standing example.)

Let $\mathbb{N}:=\left\{  0,1,2,\ldots\right\}  $ be the set of all
nonnegative integers.

For any integers $a$ and $b$, we set
\[
\left[  a,b\right]  :=\left\{  k\in\mathbb{Z}\ \mid\ a\leq k\leq b\right\}
=\left\{  a,a+1,\ldots,b\right\}  .
\]
This is an empty set if $a>b$. In general, $\left[  a,b\right]  $ is called an
\emph{integer interval}.

For each $n\in\mathbb{Z}$, let $\left[  n\right]  :=\left[  1,n\right]
=\left\{  1,2,\ldots,n\right\}  $.

Fix an integer $n\in\mathbb{N}$. Let $S_{n}$ be the $n$-th symmetric group,
i.e., the group of all permutations of $\left[  n\right]  $. We multiply
permutations in the \textquotedblleft continental\textquotedblright\ way: that
is, $\left(  \pi\sigma\right)  \left(  i\right)  =\pi\left(  \sigma\left(
i\right)  \right)  $ for all $\pi,\sigma\in S_{n}$ and $i\in\left[  n\right]
$.

For any $k$ distinct elements $i_{1},i_{2},\ldots,i_{k}$ of $\left[  n\right]
$, we let $\operatorname*{cyc}\nolimits_{i_{1},i_{2},\ldots,i_{k}}$ be the
permutation in $S_{n}$ that sends $i_{1},i_{2},\ldots,i_{k-1},i_{k}$ to
$i_{2},i_{3},\ldots,i_{k},i_{1}$, respectively while leaving all remaining
elements of $\left[  n\right]  $ unchanged. This permutation is known as a
\emph{cycle}. Note that $\operatorname*{cyc}\nolimits_{i}=\operatorname*{id}$
for any single $i\in\left[  n\right]  $.

For any $i\in\left[  n-1\right]  $, we denote the cycle $\operatorname*{cyc}%
\nolimits_{i,i+1}$ by $s_{i}$ and call it a \emph{simple transposition}.

\subsection{Somewhere-to-below shuffles, $\mathcal{A}$ and $\mathcal{T}$}

Let $\mathcal{A}$ be the group algebra $\mathbf{k}\left[  S_{n}\right]  $. In
this algebra, define $n$ elements $t_{1},t_{2},\ldots,t_{n}$ by
setting\footnote{We view $S_{n}$ as a subset of $\mathbf{k}\left[
S_{n}\right]  $ in the obvious way.}%
\[
t_{\ell}:=\operatorname*{cyc}\nolimits_{\ell}+\operatorname*{cyc}%
\nolimits_{\ell,\ell+1}+\operatorname*{cyc}\nolimits_{\ell,\ell+1,\ell
+2}+\cdots+\operatorname*{cyc}\nolimits_{\ell,\ell+1,\ldots,n}\in
\mathbf{k}\left[  S_{n}\right]
\]
for each $\ell\in\left[  n\right]  $. Thus, in particular, $t_{n}%
=\operatorname*{cyc}\nolimits_{n}=\operatorname*{id}=1$ (where $1$ means the
unity of $\mathbf{k}\left[  S_{n}\right]  $). The $n$ elements $t_{1}%
,t_{2},\ldots,t_{n}$ are known as the \emph{somewhere-to-below shuffles}.

We let $\mathcal{T}$ be the $\mathbf{k}$-subalgebra of $\mathcal{A}$ generated
by these $n$ somewhere-to-below shuffles $t_{1},t_{2},\ldots,t_{n}$. Clearly,
$\mathcal{A}$ is an $\left(  \mathcal{A},\mathcal{T}\right)  $-bimodule (with
$\mathcal{A}$ acting from the left by multiplication, and $\mathcal{T}$ acting
from the right by multiplication).

\subsection{Some $S_{n}$-representation theory}

We recall that the representations of the symmetric group $S_{n}$ (over
$\mathbf{k}$) are precisely the left $\mathbf{k}\left[  S_{n}\right]
$-modules, i.e., the left $\mathcal{A}$-modules. We will use the following
four classes of $S_{n}$-representations in particular:

\begin{enumerate}
\item \emph{The Specht modules }$\mathcal{S}^{\lambda}$\emph{:} If $\lambda$
is any partition of $n$, then the \emph{Specht module} $\mathcal{S}^{\lambda}$
is a representation of $S_{n}$ constructed using the Young diagram of shape
$\lambda$. For its definition, see \cite[Definition 5.4.1 \textbf{(b)}]{sga}
(where it is called $\mathcal{S}^{Y\left(  \lambda\right)  }$) or
\cite[\S 7.2]{Fulton97} (where it is called $S^{\lambda}$). If $\mathbf{k}$ is
a field of characteristic $0$, then the Specht module $\mathcal{S}^{\lambda}$
is irreducible.

\item \emph{The trivial representation }$\mathcal{H}_{n}$\emph{:} We let
$\mathcal{H}_{n}$ denote the $\mathbf{k}$-module $\mathbf{k}$, equipped with a
trivial $S_{n}$-action (that is, $\sigma\cdot v=v$ for all $\sigma\in S_{n}$
and $v\in\mathbf{k}$). This is called the \emph{trivial representation} of
$S_{n}$. It is isomorphic to the Specht module $\mathcal{S}^{\left(  n\right)
}$.

\item \emph{The natural representation }$\mathcal{N}_{n}$\emph{:} We let
$\mathcal{N}_{n}$ denote the free $\mathbf{k}$-module $\mathbf{k}^{n}=\left\{
\left(  v_{1},v_{2},\ldots,v_{n}\right)  \ \mid\ \text{all }v_{i}\in
\mathbf{k}\right\}  $, on which $S_{n}$ acts by permuting the coordinates:%
\[
\sigma\cdot\left(  v_{1},v_{2},\ldots,v_{n}\right)  =\left(  v_{\sigma
^{-1}\left(  1\right)  },v_{\sigma^{-1}\left(  2\right)  },\ldots
,v_{\sigma^{-1}\left(  n\right)  }\right)  \ \ \ \ \ \ \ \ \ \ \text{for all
}\sigma\in S_{n}.
\]
This is called the \emph{natural representation} of $S_{n}$.

\item \emph{The reflection quotient representation }$\mathcal{Z}_{n}$\emph{:}
If $n>0$, then the natural representation $\mathcal{N}_{n}$ has a
$1$-dimensional subrepresentation
\[
\mathcal{D}_{n}:=\left\{  \left(  v_{1},v_{2},\ldots,v_{n}\right)
\ \mid\ \text{all }v_{i}\text{ are equal}\right\}  =\left\{  \left(
a,a,\ldots,a\right)  \ \mid\ a\in\mathbf{k}\right\}  .
\]
The quotient
\[
\mathcal{Z}_{n}:=\mathcal{N}_{n}/\mathcal{D}_{n}%
\]
is thus another representation of $S_{n}$. This $\mathcal{Z}_{n}$ is called
the \emph{reflection quotient representation} of $S_{n}$. As a $\mathbf{k}%
$-module, it is free of rank $n-1$ (with basis $\left(  \overline{e_{1}%
},\overline{e_{2}},\ldots,\overline{e_{n-1}}\right)  $, where $e_{1}%
,e_{2},\ldots,e_{n}$ are the standard basis vectors of $\mathbf{k}^{n}$). Here
and in the following, the notation $\overline{v}$ denotes the residue class of
a vector $v$ modulo some submodule (the submodule is to be inferred from the context).

If $\mathbf{k}$ is a field of characteristic $0$ (or, more generally, if $n$
is invertible in $\mathbf{k}$), then this representation $\mathcal{Z}_{n}$ is
isomorphic to the Specht module $\mathcal{S}^{\left(  n-1,1\right)  }$.
Without any such assumptions, $\mathcal{Z}_{n}$ is isomorphic to the dual
$\left(  \mathcal{S}^{\left(  n-1,1\right)  }\right)  ^{\ast}$ of this Specht
module. (See Proposition \ref{prop.Z-iso} below.)
\end{enumerate}

If $V$ is any $\mathbf{k}$-module, then $V^{\ast}$ shall denote its dual
$\mathbf{k}$-module $\operatorname*{Hom}\nolimits_{\mathbf{k}}\left(
V,\mathbf{k}\right)  $. If $V$ is an $S_{n}$-representation, then its dual
$V^{\ast}$ becomes an $S_{n}$-representation as well (see \cite[\S 5.19.3]%
{sga}).

\begin{proposition}
\label{prop.Z-iso}Let $n>1$ be an integer. Then:

\begin{enumerate}
\item[\textbf{(a)}] The reflection quotient representation $\mathcal{Z}_{n}$
is isomorphic (as an $S_{n}$-representation) to the dual $\left(
\mathcal{S}^{\left(  n-1,1\right)  }\right)  ^{\ast}$ of the Specht module
$\mathcal{S}^{\left(  n-1,1\right)  }$.

\item[\textbf{(b)}] If $n$ is invertible in $\mathbf{k}$, then $\mathcal{Z}%
_{n}$ is isomorphic (as an $S_{n}$-representation) to the Specht module
$\mathcal{S}^{\left(  n-1,1\right)  }$.
\end{enumerate}
\end{proposition}

This proposition is clearly part of the folklore, but we outline a proof in
the Appendix (Subsection \ref{subsec.apx.pf.prop.Z-iso}) for the sake of completeness.

\subsection{Tensor products, induction and induction products}

We shall now discuss certain ways to produce new representations from old.

The symbol \textquotedblleft$\otimes$\textquotedblright\ shall always mean a
tensor product over $\mathbf{k}$, unless a different base ring is provided as
a subscript.

It is well-known that if $A$ and $B$ are two $\mathbf{k}$-algebras, then the
tensor product $U\otimes V$ of any left $A$-module $U$ and any left $B$-module
$V$ is canonically a left $A\otimes B$-module. An analogous construction
exists for tensor products of $k$ left modules. Thus, if $U$ is a
representation of a group $G$, and if $V$ is a representation of a group $H$,
then $U\otimes V$ is a representation of $G\times H$, and a similar fact holds
for tensor products of $k$ representations.

We recall the notion of an induced representation: If $G$ is a group, and if
$H$ is a subgroup of $G$, then any $H$-representation $V$ gives rise to a
$G$-representation $\operatorname*{Ind}\nolimits_{H}^{G}V$ defined by%
\begin{equation}
\operatorname*{Ind}\nolimits_{H}^{G}V=\mathbf{k}\left[  G\right]
\otimes_{\mathbf{k}\left[  H\right]  }V, \label{eq.IndGH.def}%
\end{equation}
where we view $\mathbf{k}\left[  G\right]  $ as a $\left(  \mathbf{k}\left[
G\right]  ,\mathbf{k}\left[  H\right]  \right)  $-bimodule while viewing $V$
as a left $\mathbf{k}\left[  H\right]  $-module (so that the tensor product
over $\mathbf{k}\left[  H\right]  $ becomes a left $\mathbf{k}\left[
G\right]  $-module). This $G$-representation $\operatorname*{Ind}%
\nolimits_{H}^{G}V$ is called the \emph{induced representation} of $V$ to $G$.

We furthermore recall the notion of an induction product (\cite[\S 7.3]%
{Fulton97}):

\begin{definition}
Let $n$ and $m$ be two nonnegative integers. Then, the direct product
$S_{n}\times S_{m}$ can be canonically embedded as a subgroup into $S_{n+m}$,
by the group morphism that sends each pair $\left(  \sigma,\tau\right)  \in
S_{n}\times S_{m}$ to the permutation $\sigma\ast\tau\in S_{n+m}$ that applies
$\sigma$ to the first $n$ elements while applying $\tau$ (appropriately
shifted) to the last $m$ elements of $\left[  n+m\right]  $. (To be fully
precise: $\sigma\ast\tau$ is the permutation of $\left[  n+m\right]  $ that
sends $1,2,\ldots,n$ to $\sigma\left(  1\right)  ,\sigma\left(  2\right)
,\ldots,\sigma\left(  n\right)  $ while sending $n+1,n+2,\ldots,n+m$ to
$n+\tau\left(  1\right)  ,n+\tau\left(  2\right)  ,\ldots,n+\tau\left(
m\right)  $.) This is called the \emph{parabolic embedding} of $S_{n}\times
S_{m}$ into $S_{n+m}$.

Now, if $U$ is an $S_{n}$-representation and if $V$ is an $S_{m}%
$-representation, then the tensor product $U\otimes V$ is an $S_{n}\times
S_{m}$-representation, and thus (by the embedding of $S_{n}\times S_{m}$ into
$S_{n+m}$ we just explained) we can construct the induced representation%
\[
U\ast V:=\operatorname*{Ind}\nolimits_{S_{n}\times S_{m}}^{S_{n+m}}\left(
U\otimes V\right)
\]
of $S_{n+m}$. This induced representation $U\ast V$ is called the
\emph{induction product} of $U$ and $V$.

More generally, if $n_{1},n_{2},\ldots,n_{k}$ are any $k$ nonnegative
integers, and if $U_{i}$ is an $S_{n_{i}}$-representation for each
$i\in\left[  k\right]  $, then the \emph{induction product }$U_{1}\ast
U_{2}\ast\cdots\ast U_{k}$ is defined to be the $S_{n_{1}+n_{2}+\cdots+n_{k}}%
$-representation%
\[
\operatorname*{Ind}\nolimits_{S_{n_{1}}\times S_{n_{2}}\times\cdots\times
S_{n_{k}}}^{S_{n_{1}+n_{2}+\cdots+n_{k}}}\left(  U_{1}\otimes U_{2}%
\otimes\cdots\otimes U_{k}\right)  ,
\]
where we embed $S_{n_{1}}\times S_{n_{2}}\times\cdots\times S_{n_{k}}$ into
$S_{n_{1}+n_{2}+\cdots+n_{k}}$ in the obvious way (having each $S_{n_{i}}$ act
on an appropriate interval\footnotemark). The latter embedding is again called
the \emph{parabolic embedding} of $S_{n_{1}}\times S_{n_{2}}\times\cdots\times
S_{n_{k}}$ into $S_{n_{1}+n_{2}+\cdots+n_{k}}$.
\end{definition}

\footnotetext{To make this precise: Let $m_{i}:=n_{1}+n_{2}+\cdots+n_{i}$ for
each $i\in\left[  0,k\right]  $. Then, the integer interval $\left[
n_{1}+n_{2}+\cdots+n_{k}\right]  $ is partitioned into the intervals $\left[
m_{i-1}+1,\ m_{i}\right]  $ for all $i\in\left[  k\right]  $. The embedding of
$S_{n_{1}}\times S_{n_{2}}\times\cdots\times S_{n_{k}}$ into $S_{n_{1}%
+n_{2}+\cdots+n_{k}}$ sends each $k$-tuple $\left(  \sigma_{1},\sigma
_{2},\ldots,\sigma_{k}\right)  \in S_{n_{1}}\times S_{n_{2}}\times\cdots\times
S_{n_{k}}$ to the permutation $\sigma_{1}\ast\sigma_{2}\ast\cdots\ast
\sigma_{k}\in S_{n_{1}+n_{2}+\cdots+n_{k}}$ defined by%
\[
\left(  \sigma_{1}\ast\sigma_{2}\ast\cdots\ast\sigma_{k}\right)  \left(
m_{i-1}+x\right)  :=m_{i-1}+\sigma_{i}\left(  x\right)
\ \ \ \ \ \ \ \ \ \ \text{for each }i\in\left[  k\right]  \text{ and each
}x\in\left[  n_{i}\right]  .
\]
} These induction products satisfy associativity up to isomorphism: e.g., we
have isomorphisms
$\left(  U\ast V\right)  \ast W\cong U\ast V\ast W\cong U\ast\left(
V\ast W\right)  $ for all $U,V,W$. More generally:

\begin{proposition}
\label{prop.indprod.ass}Let $n_{1},n_{2},\ldots,n_{k}$ be any $k$ nonnegative
integers, and let $U_{i}$ be an $S_{n_{i}}$-representation for each
$i\in\left[  k\right]  $. Let $i\in\left[  0,k\right]  $. Then,%
\[
U_{1}\ast U_{2}\ast\cdots\ast U_{k}\cong\left(  U_{1}\ast U_{2}\ast\cdots\ast
U_{i}\right)  \ast\left(  U_{i+1}\ast U_{i+2}\ast\cdots\ast U_{k}\right)  .
\]

\end{proposition}

This is again a folklore result, but we sketch a proof in the Appendix
(Subsection \ref{subsec.apx.pf.prop.indprod.ass}) to fill a little gap in the literature.

\subsection{Lacunar sets and the submodules $F\left(  I\right)  $}

Next, we recall some more concepts from \cite{s2b1}.

If $I$ is a finite set of integers, then we let $\operatorname*{sum}I$ denote
the sum of all elements of $I$. For instance, $\operatorname*{sum}\left\{
3,7\right\}  =3+7=10$.

Let $\left(  f_{0},f_{1},f_{2},\ldots\right)  $ be the \emph{Fibonacci
sequence}. This is the sequence of integers defined recursively by%
\[
f_{0}=0,\ \ \ \ \ \ \ \ \ \ f_{1}=1,\ \ \ \ \ \ \ \ \ \ \text{and}%
\ \ \ \ \ \ \ \ \ \ f_{m}=f_{m-1}+f_{m-2}\text{ for all }m\geq2.
\]

We shall say that a set $I\subseteq\mathbb{Z}$ is \emph{lacunar} if it
contains no two consecutive integers (i.e., there exists no $i\in I$ such that
$i+1\in I$). For instance, the set $\left\{  1,4,6\right\}  $ is lacunar,
while the set $\left\{  1,4,5\right\}  $ is not.

The number of lacunar subsets of $\left[  n-1\right]  $ is the Fibonacci
number $f_{n+1}$. Let $Q_{1},Q_{2},\ldots,Q_{f_{n+1}}$ be all these $f_{n+1}$
lacunar subsets of $\left[  n-1\right]  $, listed in an order that satisfies%
\begin{equation}
\operatorname*{sum}\left(  Q_{1}\right)  \leq\operatorname*{sum}\left(
Q_{2}\right)  \leq\cdots\leq\operatorname*{sum}\left(  Q_{f_{n+1}}\right)  .
\label{pf.thm.t-simultri.sum-order}%
\end{equation}
We fix this order once and for all\footnote{For $n\leq3$, this order is
uniquely defined. For $n>3$, we need to make a choice.}. Many of our
constructions will formally (though rather shallowly) depend on this order.

For any subset $I$ of $\left[  n\right]  $, we define the following:

\begin{itemize}
\item We let $I-1$ denote the set $\left\{  i-1\ \mid\ i\in I\right\}
=\left\{  j\in\mathbb{Z}\ \mid\ j+1\in I\right\}  $. For instance, $\left\{
2,4,5\right\}  -1=\left\{  1,3,4\right\}  $. Note that $I$ is lacunar if and
only if $I\cap\left(  I-1\right)  =\varnothing$.

\item We let $I^{\prime}$ be the set $\left[  n-1\right]  \setminus\left(
I\cup\left(  I-1\right)  \right)  $. This is the set of all $i\in\left[
n-1\right]  $ satisfying $i\notin I$ and $i+1\notin I$. We shall refer to
$I^{\prime}$ as the \emph{non-shadow} of $I$.

For example, if $n=5$, then $\left\{  2,3\right\}  ^{\prime}=\left[  4\right]
\setminus\left\{  1,2,3\right\}  =\left\{  4\right\}  $.

\item We let
\[
F\left(  I\right)  :=\left\{  q\in\mathbf{k}\left[  S_{n}\right]
\ \mid\ qs_{i}=q\text{ for all }i\in I^{\prime}\right\}  .
\]
We can rewrite this equality as%
\begin{align}
F\left(  I\right)   &  =\left\{  t\in\mathbf{k}\left[  S_{n}\right]
\ \mid\ ts_{j}=t\text{ for all }j\in I^{\prime}\right\} \nonumber\\
&  =\left\{  t\in\mathcal{A}\ \mid\ ts_{j}=t\text{ for all }j\in I^{\prime
}\right\}  \label{eq.FI=.2}%
\end{align}
(since $\mathbf{k}\left[  S_{n}\right]  =\mathcal{A}$).
\end{itemize}

\section{The first main theorem: the Fibonacci filtration}

\subsection{The theorem}

For each $i\in\left[  0,f_{n+1}\right]  $, we define a $\mathbf{k}$-submodule%
\[
F_{i}:=F\left(  Q_{1}\right)  +F\left(  Q_{2}\right)  +\cdots+F\left(
Q_{i}\right)  \ \ \ \ \ \ \ \ \ \ \text{of }\mathbf{k}\left[  S_{n}\right]
\]
(so that $F_{0}=0$). In \cite[Theorem 8.1]{s2b1}, the following is shown:

\begin{theorem}
\label{thm.t-simultri}\ 

\begin{enumerate}
\item[\textbf{(a)}] We have%
\[
0=F_{0}\subseteq F_{1}\subseteq F_{2}\subseteq\cdots\subseteq F_{f_{n+1}%
}=\mathbf{k}\left[  S_{n}\right]  .
\]
In other words, the $\mathbf{k}$-submodules $F_{0},F_{1},\ldots,F_{f_{n+1}}$
form a $\mathbf{k}$-module filtration of $\mathbf{k}\left[  S_{n}\right]  $.

\item[\textbf{(b)}] We have $F_{i}\cdot t_{\ell}\subseteq F_{i}$ for each
$i\in\left[  0,f_{n+1}\right]  $ and $\ell\in\left[  n\right]  $.

\item[\textbf{(c)}] For each $i\in\left[  f_{n+1}\right]  $ and $\ell
\in\left[  n\right]  $, we have%
\[
F_{i}\cdot\left(  t_{\ell}-m_{Q_{i},\ell}\right)  \subseteq F_{i-1}.
\]
Here, $m_{Q_{i},\ell}$ is a certain integer whose definition we will give in
Subsection \ref{subsec.specht-spec.thm} (as we will not use it until then).
\end{enumerate}
\end{theorem}

The filtration $0=F_{0}\subseteq F_{1}\subseteq F_{2}\subseteq\cdots\subseteq
F_{f_{n+1}}=\mathbf{k}\left[  S_{n}\right]  $ will be called the
\emph{Fibonacci filtration} of $\mathcal{A}$. We can easily see that it is a
filtration of $\left(  \mathcal{A},\mathcal{T}\right)  $-bimodules:

\begin{proposition}
\label{prop.fibfilt.A}Let $i\in\left[  0,f_{n+1}\right]  $. Then, $F_{i}$ is
an $\left(  \mathcal{A},\mathcal{T}\right)  $-subbimodule of $\mathcal{A}$.
\end{proposition}

\begin{proof}
For any $I\subseteq\left[  n\right]  $, the set $F\left(  I\right)  $ is
closed under addition and left action of $\mathcal{A}$ (by its very
definition), hence is a left $\mathcal{A}$-submodule of $\mathcal{A}$. Thus,
$F_{i}$ (being defined as a sum of such sets $F\left(  I\right)  $) is also a
left $\mathcal{A}$-submodule of $\mathcal{A}$. Moreover, $F_{i}$ is also
closed under right multiplication by each $t_{\ell}$ (by Theorem
\ref{thm.t-simultri} \textbf{(b)}), and hence under the right action of
$\mathcal{T}$ (since $\mathcal{T}$ is the subalgebra generated by $t_{1}%
,t_{2},\ldots,t_{n}$). Thus, $F_{i}$ is also a right $\mathcal{T}$-submodule
of $\mathcal{A}$. Altogether, we conclude that $F_{i}$ is an $\left(
\mathcal{A},\mathcal{T}\right)  $-subbimodule of $\mathcal{A}$.
\end{proof}

Proposition \ref{prop.fibfilt.A} shows that the subquotients $F_{i}/F_{i-1}$
are $\left(  \mathcal{A},\mathcal{T}\right)  $-bimodules as well. In
particular, they are therefore left $\mathcal{A}$-modules, i.e.,
representations of $S_{n}$. Our second main theorem characterizes these representations:

\begin{theorem}
\label{thm.Fi/Fi-1.as-ind}Let $i\in\left[  f_{n+1}\right]  $. Consider the
lacunar subset $Q_{i}$ of $\left[  n-1\right]  $. Write the set $Q_{i}%
\cup\left\{  n+1\right\}  $ as $\left\{  i_{1}<i_{2}<\cdots<i_{m}\right\}  $,
so that $i_{m}=n+1$. Furthermore, set $i_{0}:=1$. Set $j_{k}:=i_{k}-i_{k-1}$
for each $k\in\left[  m\right]  $. (Note that $j_{1}\geq0$ and $j_{2}%
,j_{3},\ldots,j_{m}>1$ and $j_{1}+j_{2}+\cdots+j_{m}=n$; this follows from
Lemma \ref{lem.Fi/Fi-1.sum-js} below (applied to $I=Q_{i}$).) Then,%
\[
F_{i}/F_{i-1}\cong\underbrace{\mathcal{H}_{j_{1}}\ast\mathcal{Z}_{j_{2}}%
\ast\mathcal{Z}_{j_{3}}\ast\cdots\ast\mathcal{Z}_{j_{m}}}_{\substack{\text{the
first factor is an }\mathcal{H}\text{,}\\\text{while all others are
}\mathcal{Z}\text{'s}}}
\]
as $S_{n}$-representations.\footnotemark
\end{theorem}

\footnotetext{Note that the factor $\mathcal{H}_{j_{1}}$ can be omitted when
$j_{1}=0$, since $\mathcal{H}_{0}\cong\mathbf{k}$ with the trivial $S_{0}%
$-action.}We will spend the rest of this section proving this theorem, then
restating it (in the characteristic-$0$ case) using Littlewood--Richardson coefficients.

\subsection{Lemmas on $F\left(  Q_{i}\right)  $}

First, let us show some lemmas about lacunar sets $I$ and the corresponding
$\mathbf{k}$-modules $F\left(  I\right)  $:

\begin{lemma}
\label{lem.Fi/Fi-1.sum-js}Let $I$ be a lacunar subset of $\left[  n-1\right]
$. Write the set $I\cup\left\{  n+1\right\}  $ as $\left\{  i_{1}<i_{2}%
<\cdots<i_{m}\right\}  $, so that $i_{m}=n+1$. Furthermore, set $i_{0}:=1$.
Set $j_{k}:=i_{k}-i_{k-1}$ for each $k\in\left[  m\right]  $. Then, $j_{1}%
\geq0$ and $j_{2},j_{3},\ldots,j_{m}>1$ and $j_{1}+j_{2}+\cdots+j_{m}=n$.
\end{lemma}

\begin{proof}
By definition, we have $j_{1}=i_{1}-i_{0}\geq0$, since $i_{1}\geq1=i_{0}$.

Next, we recall that the set $I$ is lacunar. This lacunarity is preserved even
when we insert the new element $n+1$ into this set, since all existing
elements of $I$ are $\leq n-1$ (since $I\subseteq\left[  n-1\right]  $) and
thus cannot be consecutive with $n+1$. That is, the set $I\cup\left\{
n+1\right\}  $ is again lacunar. Since we have written this set as $\left\{
i_{1}<i_{2}<\cdots<i_{m}\right\}  $, this yields that any $k\in\left[
2,m\right]  $ satisfies $i_{k}-i_{k-1}>1$. In other words, any $k\in\left[
2,m\right]  $ satisfies $j_{k}>1$ (since $j_{k}=i_{k}-i_{k-1}$). In other
words, $j_{2},j_{3},\ldots,j_{m}>1$.

It remains to prove that $j_{1}+j_{2}+\cdots+j_{m}=n$. But recall that
$j_{k}=i_{k}-i_{k-1}$ for each $k\in\left[  m\right]  $. Hence,%
\begin{align*}
\sum_{k=1}^{m}j_{k}  &  =\sum_{k=1}^{m}\left(  i_{k}-i_{k-1}\right)
=\underbrace{i_{m}}_{=n+1}-\underbrace{i_{0}}_{=1}\ \ \ \ \ \ \ \ \ \ \left(
\text{by the telescope principle}\right) \\
&  =n+1-1=n.
\end{align*}
In other words, $j_{1}+j_{2}+\cdots+j_{m}=n$. Thus, Lemma
\ref{lem.Fi/Fi-1.sum-js} is fully proved.
\end{proof}

\begin{lemma}
\label{lem.Fi/Fi-1.dim}Let $i\in\left[  f_{n+1}\right]  $. Consider the
lacunar subset $Q_{i}$ of $\left[  n-1\right]  $. Write the set $Q_{i}%
\cup\left\{  n+1\right\}  $ as $\left\{  i_{1}<i_{2}<\cdots<i_{m}\right\}  $.
Furthermore, set $i_{0}:=1$. Set $j_{k}:=i_{k}-i_{k-1}$ for each $k\in\left[
m\right]  $. Then, the $\mathbf{k}$-module $F_{i}/F_{i-1}$ is free of rank%
\[
\dfrac{n!}{j_{1}!j_{2}!\cdots j_{m}!}\cdot\prod_{k=2}^{m}\left(
j_{k}-1\right)  .
\]

\end{lemma}

\begin{proof}
We have $Q_{i}\subseteq\left[  n-1\right]  $. Hence, $n+1$ is the largest
element of $Q_{i}\cup\left\{  n+1\right\}  $. Thus, from $Q_{i}\cup\left\{
n+1\right\}  =\left\{  i_{1}<i_{2}<\cdots<i_{m}\right\}  $, we obtain
\[
i_{m}=n+1\ \ \ \ \ \ \ \ \ \ \text{and}\ \ \ \ \ \ \ \ \ \ Q_{i}=\left\{
i_{1}<i_{2}<\cdots<i_{m-1}\right\}  .
\]

Lemma \ref{lem.Fi/Fi-1.sum-js} (applied to $I=Q_{i}$) shows that $j_{1}%
+j_{2}+\cdots+j_{m}=n$. Let $\dbinom{n}{j_{1},j_{2},\ldots,j_{m}}$ denote the
multinomial coefficient $\dfrac{n!}{j_{1}!j_{2}!\cdots j_{m}!}$. We know from
\cite[Theorem 13.1 \textbf{(a)} and \textbf{(c)}]{s2b1} (applied to $p=m-1$)
that the $\mathbf{k}$-module $F_{i}/F_{i-1}$ is free of rank%
\[
\delta_{i}=\dbinom{n}{j_{1},j_{2},\ldots,j_{m}}\cdot\prod_{k=2}^{m}\left(
j_{k}-1\right)  .
\]
In view of $\dbinom{n}{j_{1},j_{2},\ldots,j_{m}}=\dfrac{n!}{j_{1}!j_{2}!\cdots
j_{m}!}$, this is precisely the claim of Lemma \ref{lem.Fi/Fi-1.dim}.
\end{proof}

\begin{lemma}
\label{lem.Fi/Fi-1.L2a}Let $I$ be a subset of $\left[  n\right]  $. Let $j\in
I$. Then, there exists a lacunar subset $J$ of $\left[  n-1\right]  $ such
that $\operatorname*{sum}J<\operatorname*{sum}I$ and $J^{\prime}\subseteq
I^{\prime}\cup\left\{  j\right\}  $.
\end{lemma}

\begin{proof}
Set $K:=\left(  I\setminus\left\{  j\right\}  \right)  \cup\left\{
j-1\right\}  $ if $j>1$, and otherwise set $K:=I\setminus\left\{  j\right\}
$. Then, $K$ is a subset of $\left[  n\right]  $ and satisfies
$\operatorname*{sum}K<\operatorname*{sum}I$ (since $K$ is obtained from $I$ by
removing the element $j$ and possibly inserting the smaller element $j-1$).
Furthermore, \cite[Proposition 8.6 \textbf{(a)}]{s2b1} says that $K^{\prime
}\subseteq I^{\prime}\cup\left\{  j\right\}  $.

Now, \cite[Corollary 8.8]{s2b1} (applied to $K$ instead of $I$) shows that
there exists a lacunar subset $J$ of $\left[  n-1\right]  $ such that
$\operatorname*{sum}J\leq\operatorname*{sum}K$ and $J^{\prime}\subseteq
K^{\prime}$. Consider this $J$.

The set $J$ is a lacunar subset of $\left[  n-1\right]  $ and satisfies
$\operatorname*{sum}J<\operatorname*{sum}I$ (since $\operatorname*{sum}%
J\leq\operatorname*{sum}K<\operatorname*{sum}I$) and $J^{\prime}\subseteq
I^{\prime}\cup\left\{  j\right\}  $ (since $J^{\prime}\subseteq K^{\prime
}\subseteq I^{\prime}\cup\left\{  j\right\}  $). Hence, such a $J$ exists.
This proves Lemma \ref{lem.Fi/Fi-1.L2a}.
\end{proof}

\begin{lemma}
\label{lem.Fi/Fi-1.L2}Let $i\in\left[  f_{n+1}\right]  $. Consider the lacunar
subset $Q_{i}$ of $\left[  n-1\right]  $. Write the set $Q_{i}\cup\left\{
n+1\right\}  $ as $\left\{  i_{1}<i_{2}<\cdots<i_{m}\right\}  $. Let
$k\in\left[  m-1\right]  $. Then,%
\[
\left\{  t\in F\left(  Q_{i}\right)  \ \mid\ ts_{i_{k}}=t\right\}  \subseteq
F_{i-1}.
\]

\end{lemma}

\begin{proof}
As in the proof of Lemma \ref{lem.Fi/Fi-1.dim}, we find $Q_{i}=\left\{
i_{1}<i_{2}<\cdots<i_{m-1}\right\}  $. Thus, $i_{k}\in Q_{i}$ (since
$k\in\left[  m-1\right]  $). Hence, Lemma \ref{lem.Fi/Fi-1.L2a} (applied to
$I=Q_{i}$ and $j=i_{k}$) shows that there exists a lacunar subset $J$ of
$\left[  n-1\right]  $ such that $\operatorname*{sum}J<\operatorname*{sum}%
\left(  Q_{i}\right)  $ and $J^{\prime}\subseteq Q_{i}^{\prime}\cup\left\{
i_{k}\right\}  $. Consider this $J$. Since $J$ is lacunar, we have $J=Q_{s}$
for some $s\in\left[  f_{n+1}\right]  $. Consider this $s$. Thus, $Q_{s}=J$,
so that $\operatorname*{sum}\left(  Q_{s}\right)  =\operatorname*{sum}%
J<\operatorname*{sum}\left(  Q_{i}\right)  $ and therefore $s<i$ (by
(\ref{pf.thm.t-simultri.sum-order})). Hence, $s\leq i-1$, so that $F\left(
Q_{s}\right)  \subseteq F_{i-1}$ (since the definition of $F_{i-1}$ says that
$F_{i-1}=F\left(  Q_{1}\right)  +F\left(  Q_{2}\right)  +\cdots+F\left(
Q_{i-1}\right)  $). In view of $Q_{s}=J$, we can rewrite this as%
\[
F\left(  J\right)  \subseteq F_{i-1}.
\]

Now, (\ref{eq.FI=.2}) (applied to $Q_{i}$ instead of $I$) shows that%
\[
F\left(  Q_{i}\right)  =\left\{  t\in\mathcal{A}\ \mid\ ts_{j}=t\text{ for all
}j\in Q_{i}^{\prime}\right\}  ,
\]
so that%
\begin{align*}
&  \left\{  t\in F\left(  Q_{i}\right)  \ \mid\ ts_{i_{k}}=t\right\} \\
&  =\left\{  t\in\mathcal{A}\ \mid\ ts_{j}=t\text{ for all }j\in Q_{i}%
^{\prime}\text{, and also }ts_{i_{k}}=t\right\} \\
&  =\left\{  t\in\mathcal{A}\ \mid\ ts_{j}=t\text{ for all }j\in Q_{i}%
^{\prime}\cup\left\{  i_{k}\right\}  \right\} \\
&  \subseteq\left\{  t\in\mathcal{A}\ \mid\ ts_{j}=t\text{ for all }j\in
J^{\prime}\right\}  \ \ \ \ \ \ \ \ \ \ \left(  \text{since }J^{\prime
}\subseteq Q_{i}^{\prime}\cup\left\{  i_{k}\right\}  \right) \\
&  =F\left(  J\right)  \ \ \ \ \ \ \ \ \ \ \left(  \text{by (\ref{eq.FI=.2}),
applied to }J\text{ instead of }I\right) \\
&  \subseteq F_{i-1}.
\end{align*}
Thus, Lemma \ref{lem.Fi/Fi-1.L2} follows.
\end{proof}

\subsection{The elements $\nabla_{\mathbf{p}}$}

\begin{lemma}
\label{lem.Fi/Fi-1.L3}Let $i\in\left[  f_{n+1}\right]  $. Consider the lacunar
subset $Q_{i}$ of $\left[  n-1\right]  $. Write the set $Q_{i}\cup\left\{
n+1\right\}  $ as $\left\{  i_{1}<i_{2}<\cdots<i_{m}\right\}  $. Furthermore,
set $i_{0}:=1$.

For each $k\in\left[  m\right]  $, let $J_{k}$ denote the integer interval
$\left[  i_{k-1},\ i_{k}-1\right]  $. Note that the intervals $J_{1}%
,J_{2},\ldots,J_{m}$ are disjoint and -- except possibly for $J_{1}$ --
nonempty ($J_{1}$ is empty if and only if $1\in Q_{i}$), and their union is
$\left[  n\right]  $. Thus, we can view the direct product $S_{J_{1}}\times
S_{J_{2}}\times\cdots\times S_{J_{m}}$ as a subgroup of $S_{n}$ in the obvious
way (each factor $S_{J_{k}}$ acts on the elements of $J_{k}$ while leaving all
remaining elements of $\left[  n\right]  $ unchanged).

For each $\left(  m-1\right)  $-tuple $\mathbf{p}=\left(  p_{2},p_{3}%
,\ldots,p_{m}\right)  \in J_{2}\times J_{3}\times\cdots\times J_{m}$ (that is,
with $p_{k}\in J_{k}$ for each $k\in\left[  2,m\right]  $), we define an
element%
\[
\nabla_{\mathbf{p}}:=\sum_{\substack{\sigma\in S_{n};\\\sigma\left(
J_{k}\right)  =J_{k}\text{ for each }k\in\left[  m\right]  ;\\\sigma\left(
i_{k-1}\right)  =p_{k}\text{ for each }k\in\left[  2,m\right]  }}\sigma
\in\mathcal{A}.
\]
(Note that this also depends on $i$, not just on $\mathbf{p}$.)

Then:

\begin{enumerate}
\item[\textbf{(a)}] For any $\tau=\left(  \tau_{1},\tau_{2},\ldots,\tau
_{m}\right)  \in S_{J_{1}}\times S_{J_{2}}\times\cdots\times S_{J_{m}}$ and
$\mathbf{p}=\left(  p_{2},p_{3},\ldots,p_{m}\right)  \in J_{2}\times
J_{3}\times\cdots\times J_{m}$, we have%
\[
\tau\nabla_{\mathbf{p}}=\nabla_{\tau\mathbf{p}},
\]
where%
\[
\tau\mathbf{p}:=\left(  \tau_{2}\left(  p_{2}\right)  ,\tau_{3}\left(
p_{3}\right)  ,\ldots,\tau_{m}\left(  p_{m}\right)  \right)  =\left(
\tau\left(  p_{2}\right)  ,\tau\left(  p_{3}\right)  ,\ldots,\tau\left(
p_{m}\right)  \right)  .
\]

\item[\textbf{(b)}] The left $\mathcal{A}$-module $F\left(  Q_{i}\right)  $ is
generated (as a left $\mathcal{A}$-module) by any single element of the form
$\nabla_{\mathbf{p}}$ (with $\mathbf{p}\in J_{2}\times J_{3}\times\cdots\times
J_{m}$).

\item[\textbf{(c)}] Let $\ell\in\left[  2,m\right]  $. For each $k\in\left[
2,m\right]  \setminus\left\{  \ell\right\}  $, let $p_{k}\in J_{k}$ be an
element. Then,%
\[
\sum_{p_{\ell}\in J_{\ell}}\nabla_{\left(  p_{2},p_{3},\ldots,p_{m}\right)
}\in F_{i-1}.
\]
(Note that the elements $p_{2},p_{3},\ldots,p_{\ell-1},p_{\ell+1},p_{\ell
+2},\ldots,p_{m}$ in this sum are fixed, whereas $p_{\ell}$ runs through the
set $J_{\ell}$.)
\end{enumerate}
\end{lemma}

\begin{example}
\label{exa.Fi/Fi-1.L3.e1}Let $n=7$ and $Q_{i}=\left\{  3,6\right\}  $ (clearly
a lacunar subset of $\left[  n-1\right]  =\left[  6\right]  $). Then,
following the notations of Lemma \ref{lem.Fi/Fi-1.L3}, we have $\left\{
i_{1}<i_{2}<\cdots<i_{m}\right\}  =Q_{i}\cup\left\{  n+1\right\}  =\left\{
3,6,8\right\}  $, so that $i_{1}=3$ and $i_{2}=6$ and $i_{3}=8$ and $i_{0}=1$.
Thus, the integer intervals $J_{k}=\left[  i_{k-1},\ i_{k}-1\right]  $ are
\[
J_{1}=\left[  1,2\right]  ,\ \ \ \ \ \ \ \ \ \ J_{2}=\left[  3,5\right]
,\ \ \ \ \ \ \ \ \ \ J_{3}=\left[  6,7\right]  .
\]
Taking $\mathbf{p}$ to be the $2$-tuple $\left(  p_{2},p_{3}\right)  =\left(
4,7\right)  \in J_{2}\times J_{3}$, we then have%
\begin{align*}
\nabla_{\mathbf{p}}  & =\sum_{\substack{\sigma\in S_{n};\\\sigma\left(
J_{k}\right)  =J_{k}\text{ for each }k\in\left[  m\right]  ;\\\sigma\left(
i_{k-1}\right)  =p_{k}\text{ for each }k\in\left[  2,m\right]  }}\sigma
=\sum_{\substack{\sigma\in S_{n};\\\sigma\left(  \left[  1,2\right]  \right)
=\left[  1,2\right]  ;\ \sigma\left(  \left[  3,5\right]  \right)  =\left[
3,5\right]  ;\ \sigma\left(  \left[  6,7\right]  \right)  =\left[  6,7\right]
;\\\sigma\left(  3\right)  =4;\ \sigma\left(  6\right)  =7}}\sigma\\
& =\left[  1243576\right]  +\left[  1245376\right]  +\left[  2143576\right]
+\left[  2145376\right]
\end{align*}
(writing permutations in one-line notation). Lemma \ref{lem.Fi/Fi-1.L3}
\textbf{(b)} says that this element $\nabla_{\mathbf{p}}$ generates the left
$\mathcal{A}$-module $F\left(  Q_{i}\right)  $. Applying Lemma
\ref{lem.Fi/Fi-1.L3} \textbf{(c)} to $\ell=2$ and $p_{3}=7$, we obtain%
\[
\sum_{p_{2}\in J_{2}}\nabla_{\left(  p_{2},7\right)  }\in F_{i-1}%
,\ \ \ \ \ \ \ \ \ \ \text{that is,}\ \ \ \ \ \ \ \ \ \ \nabla_{\left(
3,7\right)  }+\nabla_{\left(  4,7\right)  }+\nabla_{\left(  5,7\right)  }\in
F_{i-1}.
\]

\end{example}

\begin{proof}
[Proof of Lemma \ref{lem.Fi/Fi-1.L3}.] As in the proof of Lemma
\ref{lem.Fi/Fi-1.dim}, we find $i_{m}=n+1$ and $Q_{i}=\left\{  i_{1}%
<i_{2}<\cdots<i_{m-1}\right\}  $. \medskip

\textbf{(a)} Let $\tau=\left(  \tau_{1},\tau_{2},\ldots,\tau_{m}\right)  \in
S_{J_{1}}\times S_{J_{2}}\times\cdots\times S_{J_{m}}$ and $\mathbf{p}=\left(
p_{2},p_{3},\ldots,p_{m}\right)  \in J_{2}\times J_{3}\times\cdots\times
J_{m}$. Recall that
\[
\nabla_{\mathbf{p}}=\sum_{\substack{\sigma\in S_{n};\\\sigma\left(
J_{k}\right)  =J_{k}\text{ for each }k\in\left[  m\right]  ;\\\sigma\left(
i_{k-1}\right)  =p_{k}\text{ for each }k\in\left[  2,m\right]  }}\sigma.
\]
Multiplying this equality by $\tau$ from the left, we obtain%
\begin{align*}
\tau\nabla_{\mathbf{p}} &  =\sum_{\substack{\sigma\in S_{n};\\\sigma\left(
J_{k}\right)  =J_{k}\text{ for each }k\in\left[  m\right]  ;\\\sigma\left(
i_{k-1}\right)  =p_{k}\text{ for each }k\in\left[  2,m\right]  }}\tau
\sigma=\sum_{\substack{\sigma\in S_{n};\\\tau\left(  \sigma\left(
J_{k}\right)  \right)  =\tau\left(  J_{k}\right)  \text{ for each }k\in\left[
m\right]  ;\\\tau\left(  \sigma\left(  i_{k-1}\right)  \right)  =\tau\left(
p_{k}\right)  \text{ for each }k\in\left[  2,m\right]  }}\tau\sigma\\
&  \ \ \ \ \ \ \ \ \ \ \ \ \ \ \ \ \ \ \ \ \left(
\begin{array}
[c]{c}%
\text{here, we have replaced the conditions \textquotedblleft}\sigma\left(
J_{k}\right)  =J_{k}\text{\textquotedblright}\\
\text{and \textquotedblleft}\sigma\left(  i_{k-1}\right)  =p_{k}%
\text{\textquotedblright\ under the summation sign}\\
\text{by the conditions \textquotedblleft}\tau\left(  \sigma\left(
J_{k}\right)  \right)  =\tau\left(  J_{k}\right)  \text{\textquotedblright}\\
\text{and \textquotedblleft}\tau\left(  \sigma\left(  i_{k-1}\right)  \right)
=\tau\left(  p_{k}\right)  \text{\textquotedblright\ (which are equivalent
to}\\
\text{the former two conditions because }\tau\text{ is injective)}%
\end{array}
\right)  \\
&  =\sum_{\substack{\sigma\in S_{n};\\\left(  \tau\sigma\right)  \left(
J_{k}\right)  =J_{k}\text{ for each }k\in\left[  m\right]  ;\\\left(
\tau\sigma\right)  \left(  i_{k-1}\right)  =\tau\left(  p_{k}\right)  \text{
for each }k\in\left[  2,m\right]  }}\tau\sigma\\
&  \ \ \ \ \ \ \ \ \ \ \ \ \ \ \ \ \ \ \ \ \left(
\begin{array}
[c]{c}%
\text{since each }k\in\left[  m\right]  \text{ satisfies }\tau\left(
\sigma\left(  J_{k}\right)  \right)  =\left(  \tau\sigma\right)  \left(
J_{k}\right)  \\
\text{and }\tau\left(  \sigma\left(  i_{k-1}\right)  \right)  =\left(
\tau\sigma\right)  \left(  i_{k-1}\right)  \text{ (if }k>1\text{)}\\
\text{and }\tau\left(  J_{k}\right)  =J_{k}\text{ (since }\tau\in S_{J_{1}%
}\times S_{J_{2}}\times\cdots\times S_{J_{m}}\text{)}%
\end{array}
\right)  \\
&  =\sum_{\substack{\sigma\in S_{n};\\\sigma\left(  J_{k}\right)  =J_{k}\text{
for each }k\in\left[  m\right]  ;\\\sigma\left(  i_{k-1}\right)  =\tau\left(
p_{k}\right)  \text{ for each }k\in\left[  2,m\right]  }}\sigma\\
&  \ \ \ \ \ \ \ \ \ \ \ \ \ \ \ \ \ \ \ \ \left(
\begin{array}
[c]{c}%
\text{here, we have substituted }\sigma\text{ for }\tau\sigma\text{ in the
sum,}\\
\text{since the map }S_{n}\rightarrow S_{n},\ \sigma\mapsto\tau\sigma\text{ is
a bijection}%
\end{array}
\right)  \\
&  =\nabla_{\tau\mathbf{p}}%
\end{align*}
(by the definition of $\nabla_{\tau\mathbf{p}}$, since $\tau\mathbf{p}=\left(
\tau\left(  p_{2}\right)  ,\tau\left(  p_{3}\right)  ,\ldots,\tau\left(
p_{m}\right)  \right)  $). This proves Lemma \ref{lem.Fi/Fi-1.L3}
\textbf{(a)}. \medskip

\textbf{(b)} The definition of the non-shadow $Q_{i}^{\prime}$ yields%
\begin{align*}
Q_{i}^{\prime} &  =\left[  n-1\right]  \setminus\left(  Q_{i}\cup\left(
Q_{i}-1\right)  \right)  \\
&  =\left[  n-1\right]  \setminus\left(  \left\{  i_{1}<i_{2}<\cdots
<i_{m-1}\right\}  \cup\left\{  i_{1}-1<i_{2}-1<\cdots<i_{m-1}-1\right\}
\right)
\end{align*}
(since $Q_{i}=\left\{  i_{1}<i_{2}<\cdots<i_{m-1}\right\}  $). In other words,
$Q_{i}^{\prime}$ consists of all elements of $\left[  n-1\right]  $ except for
those of the forms $i_{k}-1$ and $i_{k}$ for $k\in\left[  m-1\right]  $.

Let $\Gamma$ be the subgroup of $S_{n}$ generated by the simple transpositions
$s_{j}$ with $j\in Q_{i}^{\prime}$. Thus, $\Gamma$ is generated by all simple
transpositions $s_{1},s_{2},\ldots,s_{n-1}$ except for those of the forms
$s_{i_{k}-1}$ and $s_{i_{k}}$ for $k\in\left[  m-1\right]  $ (by the
description of $Q_{i}^{\prime}$ in the previous paragraph). Hence, every
permutation $\omega\in\Gamma$ preserves the intervals $J_{1},J_{2}%
,\ldots,J_{m}$ as well as the elements $i_{1},i_{2},\ldots,i_{m-1}$.

Conversely, if some permutation $\omega\in S_{n}$ preserves the intervals
$J_{1},J_{2},\ldots,J_{m}$ as well as the elements $i_{1},i_{2},\ldots
,i_{m-1}$, then $\omega$ must belong to $\Gamma$ (because such a permutation
$\omega$ must preserve the intervals $J_{1}=\left[  i_{0},\ i_{1}-1\right]  $
as well as $J_{k}\setminus\left\{  i_{k-1}\right\}  =\left[  i_{k-1}%
+1,\ i_{k}-1\right]  $ for all $k\in\left[  2,m\right]  $ (since it preserves
both $J_{k}$ and $i_{k-1}$) as well as the length-$1$ intervals $\left\{
i_{k-1}\right\}  $ for all $k\in\left[  2,m\right]  $, and thus must be a
composition of permutations of these intervals; but any such permutation
belongs to $\Gamma$ (since any permutation of an integer interval $\left[
a,b\right]  $ can be written as a product of simple transpositions $s_{j}$
with $j\in\left[  a,b-1\right]  $)).

The subgroup $\Gamma$ of $S_{n}$ acts from the right on $S_{n}$ (simply by
right multiplication), and thus also acts $\mathbf{k}$-linearly from the right
on $\mathcal{A}=\mathbf{k}\left[  S_{n}\right]  $ (by linear extension),
making $\mathcal{A}$ into a permutation module\footnote{Recall the definition
of a permutation module:
\par
Let $G$ be a finite group. Let $X$ be a right $G$-set. Let $\mathbf{k}%
^{\left(  X\right)  }$ be the free $\mathbf{k}$-module with basis $X$. Then,
$\mathbf{k}^{\left(  X\right)  }$ becomes a right $\mathbf{k}\left[  G\right]
$-module, where the action of $\mathbf{k}\left[  G\right]  $ on $\mathbf{k}%
^{\left(  X\right)  }$ is given by bilinearly extending the action of $G$ on
$X$ (that is, by the rule $\left(  \sum_{g\in G}\alpha_{g}g\right)  \left(
\sum_{x\in X}\beta_{x}x\right)  :=\sum_{g\in G}\ \ \sum_{x\in X}\alpha
_{g}\beta_{x}gx$). This is called the \emph{permutation module} corresponding
to the right $G$-set $X$.
\par
In our present setup, we apply this construction to $G=\Gamma$ and $X=S_{n}$.}
of $\Gamma$. Applying (\ref{eq.FI=.2}) to $I=Q_{i}$, we see that%
\begin{align*}
F\left(  Q_{i}\right)   &  =\left\{  t\in\mathcal{A}\ \mid\ ts_{j}=t\text{ for
all }j\in Q_{i}^{\prime}\right\}  \\
&  =\left\{  t\in\mathcal{A}\ \mid\ t\omega=t\text{ for all }\omega\in
\Gamma\right\}
\end{align*}
(since $\Gamma$ is the group generated by the $s_{j}$ with $j\in Q_{i}%
^{\prime}$, and therefore the condition \textquotedblleft$ts_{j}=t$ for all
$j\in Q_{i}^{\prime}$\textquotedblright\ is equivalent to \textquotedblleft%
$t\omega=t$ for all $\omega\in\Gamma$\textquotedblright). Thus, $F\left(
Q_{i}\right)  $ is the space of fixed points\footnote{Recall the definition of
a space of fixed points: If a $\mathbf{k}$-module $V$ is equipped with a
linear right action of a group $G$ (that is, if $V$ is a right $\mathbf{k}%
\left[  G\right]  $-module), then its \emph{space of fixed points} is defined
to be the set $\left\{  a\in V\ \mid\ ag=a\text{ for all }g\in G\right\}  $.
This is a $\mathbf{k}$-submodule of $V$.} of the right $\Gamma$-action on
$\mathcal{A}$.

However, we know from the basic theory of group actions (see, e.g.,
\cite[\S 3.3.1, \textquotedblleft Invariants of Permutation
Representations\textquotedblright]{Lorenz18} or \cite[Proposition
A.2]{GriPar25}) that when a finite group $G$ acts on a set $X$, the space of
fixed points of the corresponding permutation module is spanned by the orbit
sums\footnote{In more details:
\par
Let $G$ be a finite group. Let $X$ be a right $G$-set. Consider the
corresponding permutation module $\mathbf{k}^{\left(  X\right)  }$, with its
right $G$-action.
\par
For each $G$-orbit $\mathcal{O}$ on $X$, we define the \emph{orbit sum}
$z_{\mathcal{O}}:=\sum_{x\in\mathcal{O}}x\in\mathbf{k}^{\left(  X\right)  }$.
Now, the known fact that we are citing here is saying that these orbit sums
$z_{\mathcal{O}}$ (as $\mathcal{O}$ ranges over all $G$-orbits on $X$) form a
basis of the space of fixed points of $\mathbf{k}^{\left(  X\right)  }$ (as a
$\mathbf{k}$-module).
\par
In \cite[\S 3.3.1, \textquotedblleft Invariants of Permutation
Representations\textquotedblright]{Lorenz18} and in \cite[Proposition
A.2]{GriPar25}, this is stated for left $G$-actions, but the case of right
$G$-actions is analogous.}. Hence, the $\mathbf{k}$-module $F\left(
Q_{i}\right)  $ is spanned by the orbit sums of the right $\Gamma$-action on
$S_{n}$ (since $F\left(  Q_{i}\right)  $ is the set of fixed points of the
right $\Gamma$-action on $\mathcal{A}$, which is the permutation module
corresponding to the right $\Gamma$-action on $S_{n}$). In other words,
$F\left(  Q_{i}\right)  $ is spanned by the orbit sums $\sum_{\sigma\in
\tau\Gamma}\sigma$ for $\tau\in S_{n}$ (since each orbit of the right $\Gamma
$-action on $S_{n}$ has the form $\tau\Gamma$ for some $\tau\in S_{n}$). As a
left $\mathcal{A}$-module, $F\left(  Q_{i}\right)  $ is therefore generated by
any \textbf{one} of these orbit sums (since any two orbit sums $\sum
_{\sigma\in\tau_{1}\Gamma}\sigma$ and $\sum_{\sigma\in\tau_{2}\Gamma}\sigma$
can be transformed into each other by left multiplication by $\tau_{1}\tau
_{2}^{-1}\in S_{n}\subseteq\mathcal{A}$, and therefore each of them generates
the other).

Now, let $\mathbf{p}=\left(  p_{2},p_{3},\ldots,p_{m}\right)  \in J_{2}\times
J_{3}\times\cdots\times J_{m}$. We shall now show that $\nabla_{\mathbf{p}}$
is one of these orbit sums we just mentioned. Indeed, let $\Omega_{\mathbf{p}%
}$ be the set of all permutations $\sigma\in S_{n}$ that satisfy
\textquotedblleft$\sigma\left(  J_{k}\right)  =J_{k}$ for each $k\in\left[
m\right]  $\textquotedblright\ and \textquotedblleft$\sigma\left(
i_{k-1}\right)  =p_{k}$ for each $k\in\left[  2,m\right]  $\textquotedblright.
Then, the definition of $\nabla_{\mathbf{p}}$ can be rewritten as%
\begin{equation}
\nabla_{\mathbf{p}}=\sum_{\sigma\in\Omega_{\mathbf{p}}}\sigma
.\label{pf.lem.Fi/Fi-1.L3.nap=}%
\end{equation}
We shall now show that $\Omega_{\mathbf{p}}$ is an orbit of the right $\Gamma
$-action on $S_{n}$ (that is, a left coset of $\Gamma$ in $S_{n}$).

First, we show that the set $\Omega_{\mathbf{p}}$ is nonempty. Indeed, it is
easy to construct some permutation $\tau\in\Omega_{\mathbf{p}}$: Namely, we
pick a permutation $\tau_{1}\in S_{J_{1}}$ arbitrarily. Furthermore, for each
$k\in\left[  2,m\right]  $, we pick a permutation $\tau_{k}\in S_{J_{k}}$ that
sends $i_{k-1}\in J_{k}$ to $p_{k}\in J_{k}$. The $m$-tuple $\left(  \tau
_{1},\tau_{2},\ldots,\tau_{m}\right)  $ then belongs to $S_{J_{1}}\times
S_{J_{2}}\times\cdots\times S_{J_{m}}$ and -- viewed as an element of $S_{n}$
via the embedding $S_{J_{1}}\times S_{J_{2}}\times\cdots\times S_{J_{m}%
}\rightarrow S_{n}$ -- belongs to $\Omega_{\mathbf{p}}$.

Hence, $\Omega_{\mathbf{p}}$ is nonempty. Pick any $\tau\in\Omega_{\mathbf{p}%
}$. Then, $\tau\left(  J_{k}\right)  =J_{k}$ for each $k\in\left[  m\right]
$, and $\tau\left(  i_{k-1}\right)  =p_{k}$ for each $k\in\left[  2,m\right]
$. Moreover, these equalities remain valid if we replace $\tau$ by $\tau
\omega$ for any $\omega\in\Gamma$ (because every permutation $\omega\in\Gamma$
preserves the sets $J_{1},J_{2},\ldots,J_{m}$ as well as the elements
$i_{1},i_{2},\ldots,i_{m-1}$). Thus, for each $\omega\in\Gamma$, we have
$\tau\omega\in\Omega_{\mathbf{p}}$ as well. In other words, $\tau
\Gamma\subseteq\Omega_{\mathbf{p}}$.

Conversely, we claim that $\Omega_{\mathbf{p}}\subseteq\tau\Gamma$. Indeed,
let $\sigma\in\Omega_{\mathbf{p}}$ be arbitrary. Then, each $k\in\left[
m\right]  $ satisfies $\sigma\left(  J_{k}\right)  =J_{k}=\tau\left(
J_{k}\right)  $, whereas each $k\in\left[  2,m\right]  $ satisfies
$\sigma\left(  i_{k-1}\right)  =p_{k}=\tau\left(  i_{k-1}\right)  $. Set
$\omega=\tau^{-1}\sigma\in S_{n}$; thus, each $k\in\left[  m\right]  $
satisfies $\omega\left(  J_{k}\right)  =\tau^{-1}\left(  \sigma\left(
J_{k}\right)  \right)  =J_{k}$ (since we just saw that $\sigma\left(
J_{k}\right)  =\tau\left(  J_{k}\right)  $), and each $k\in\left[  2,m\right]
$ satisfies $\omega\left(  i_{k-1}\right)  =\tau^{-1}\left(  \sigma\left(
i_{k-1}\right)  \right)  =i_{k-1}$ (since we just saw that $\sigma\left(
i_{k-1}\right)  =\tau\left(  i_{k-1}\right)  $). Thus, the permutation
$\omega\in S_{n}$ preserves the intervals $J_{1},J_{2},\ldots,J_{m}$ as well
as the elements $i_{1},i_{2},\ldots,i_{m-1}$. Hence, $\omega\in\Gamma$
(because if some permutation $\omega\in S_{n}$ preserves the intervals
$J_{1},J_{2},\ldots,J_{m}$ as well as the elements $i_{1},i_{2},\ldots
,i_{m-1}$, then $\omega$ must belong to $\Gamma$). Now, from $\omega=\tau
^{-1}\sigma$, we obtain $\sigma=\tau\omega\in\tau\Gamma$ (since $\omega
\in\Gamma$). Forget that we fixed $\sigma$. We thus have proved that
$\sigma\in\tau\Gamma$ for each $\sigma\in\Omega_{\mathbf{p}}$. In other words,
$\Omega_{\mathbf{p}}\subseteq\tau\Gamma$.

Combining this with $\tau\Gamma\subseteq\Omega_{\mathbf{p}}$, we obtain
$\Omega_{\mathbf{p}}=\tau\Gamma$. Hence, $\Omega_{\mathbf{p}}$ is an orbit of
the right $\Gamma$-action on $S_{n}$. Thus, $\sum_{\sigma\in\Omega
_{\mathbf{p}}}\sigma$ is an orbit sum of this action. In view of
(\ref{pf.lem.Fi/Fi-1.L3.nap=}), this means that $\nabla_{\mathbf{p}}$ is an
orbit sum of this action. Hence, as a left $\mathcal{A}$-module, $F\left(
Q_{i}\right)  $ is generated by $\nabla_{\mathbf{p}}$ (since we have shown
that $F\left(  Q_{i}\right)  $ is generated by any \textbf{one} of the orbit
sums). This proves Lemma \ref{lem.Fi/Fi-1.L3} \textbf{(b)}. \medskip

\textbf{(c)} Let
\begin{equation}
\vartheta:=\sum_{p_{\ell}\in J_{\ell}}\nabla_{\left(  p_{2},p_{3},\ldots
,p_{m}\right)  }.\label{pf.lem.Fi/Fi-1.L3.c.1}%
\end{equation}
We then must show that $\vartheta\in F_{i-1}$.

We have%
\[
\vartheta=\sum_{p_{\ell}\in J_{\ell}}\underbrace{\nabla_{\left(  p_{2}%
,p_{3},\ldots,p_{m}\right)  }}_{\substack{\in F\left(  Q_{i}\right)
\\\text{(since Lemma \ref{lem.Fi/Fi-1.L3} \textbf{(b)} shows that }%
\nabla_{\mathbf{q}}\in F\left(  Q_{i}\right)  \\\text{for any }\mathbf{q}\in
J_{2}\times J_{3}\times\cdots\times J_{m}\text{)}}}\in F\left(  Q_{i}\right)
\]
(since $F\left(  Q_{i}\right)  $ is a $\mathbf{k}$-module). On the other hand,%
\begin{align}
\vartheta &  =\sum_{p_{\ell}\in J_{\ell}}\nabla_{\left(  p_{2},p_{3}%
,\ldots,p_{m}\right)  }\nonumber\\
&  =\sum_{p_{\ell}\in J_{\ell}}\ \ \sum_{\substack{\sigma\in S_{n}%
;\\\sigma\left(  J_{k}\right)  =J_{k}\text{ for each }k\in\left[  m\right]
;\\\sigma\left(  i_{k-1}\right)  =p_{k}\text{ for each }k\in\left[
2,m\right]  }}\sigma\ \ \ \ \ \ \ \ \ \ \left(  \text{by the definition of
}\nabla_{\left(  p_{2},p_{3},\ldots,p_{m}\right)  }\right)  \nonumber\\
&  =\sum_{p_{\ell}\in J_{\ell}}\ \ \sum_{\substack{\sigma\in S_{n}%
;\\\sigma\left(  J_{k}\right)  =J_{k}\text{ for each }k\in\left[  m\right]
;\\\sigma\left(  i_{k-1}\right)  =p_{k}\text{ for each }k\in\left[
2,m\right]  \setminus\left\{  \ell\right\}  ;\\\sigma\left(  i_{\ell
-1}\right)  =p_{\ell}}}\sigma\nonumber\\
&  \ \ \ \ \ \ \ \ \ \ \ \ \ \ \ \ \ \ \ \ \left(
\begin{array}
[c]{c}%
\text{here, we have split up the}\\
\text{condition \textquotedblleft}\sigma\left(  i_{k-1}\right)  =p_{k}\text{
for each }k\in\left[  2,m\right]  \text{\textquotedblright}\\
\text{under the second summation sign}\\
\text{into two: one for }k\neq\ell\text{ and one for }k=\ell
\end{array}
\right)  \nonumber\\
&  =\sum_{\substack{\sigma\in S_{n};\\\sigma\left(  J_{k}\right)  =J_{k}\text{
for each }k\in\left[  m\right]  ;\\\sigma\left(  i_{k-1}\right)  =p_{k}\text{
for each }k\in\left[  2,m\right]  \setminus\left\{  \ell\right\}
;\\\sigma\left(  i_{\ell-1}\right)  \in J_{\ell}}}\sigma
\label{pf.lem.Fi/Fi-1.L3.c.5}%
\end{align}
(here, we have subsumed the two summation signs into one by removing the
variable $p_{\ell}$). The condition \textquotedblleft$\sigma\left(  i_{\ell
-1}\right)  \in J_{\ell}$\textquotedblright\ under the summation sign in
(\ref{pf.lem.Fi/Fi-1.L3.c.5}) is redundant, since it follows from the
condition \textquotedblleft$\sigma\left(  J_{k}\right)  =J_{k}$ for each
$k\in\left[  m\right]  $\textquotedblright\ (indeed, the latter condition
implies that $\sigma\left(  J_{\ell}\right)  =J_{\ell}$ and therefore
$\sigma\left(  \underbrace{i_{\ell-1}}_{\in J_{\ell}}\right)  \in\sigma\left(
J_{\ell}\right)  =J_{\ell}$). Hence, we can remove this condition. Thus,
(\ref{pf.lem.Fi/Fi-1.L3.c.5}) rewrites as
\begin{equation}
\vartheta=\sum_{\substack{\sigma\in S_{n};\\\sigma\left(  J_{k}\right)
=J_{k}\text{ for each }k\in\left[  m\right]  ;\\\sigma\left(  i_{k-1}\right)
=p_{k}\text{ for each }k\in\left[  2,m\right]  \setminus\left\{  \ell\right\}
}}\sigma.\label{pf.lem.Fi/Fi-1.L3.c.2}%
\end{equation}
However, the two conditions \textquotedblleft$\sigma\left(  J_{k}\right)
=J_{k}$ for each $k\in\left[  m\right]  $\textquotedblright\ and
\textquotedblleft$\sigma\left(  i_{k-1}\right)  =p_{k}$ for each $k\in\left[
2,m\right]  \setminus\left\{  \ell\right\}  $\textquotedblright\ under the
summation sign in (\ref{pf.lem.Fi/Fi-1.L3.c.2}) remain unchanged if we replace
$\sigma$ by $\sigma s_{i_{\ell-1}}$ (since this replacement merely swaps the
values of $\sigma$ on $i_{\ell-1}$ and $i_{\ell-1}+1$, but this does not break
any of the two conditions\footnote{Here we use the fact that the two elements
$i_{\ell-1}$ and $i_{\ell-1}+1$ lie in the same $J_{k}$ (namely, in $J_{\ell
}=\left[  i_{\ell-1},\ i_{\ell}-1\right]  $). This is because $Q_{i}$ is
lacunar, so that $i_{\ell-1}<i_{\ell}-1$.}). Hence, the set of the
permutations $\sigma$ over which we sum in (\ref{pf.lem.Fi/Fi-1.L3.c.2}) is
fixed under right multiplication by $s_{i_{\ell-1}}$. Therefore, the whole sum
is fixed under right multiplication by $s_{i_{\ell-1}}$. Because of
(\ref{pf.lem.Fi/Fi-1.L3.c.2}), this shows that $\vartheta s_{i_{\ell-1}%
}=\vartheta$. Combining this with $\vartheta\in F\left(  Q_{i}\right)  $, we
obtain%
\[
\vartheta\in\left\{  t\in F\left(  Q_{i}\right)  \ \mid\ ts_{i_{\ell-1}%
}=t\right\}  \subseteq F_{i-1}%
\]
(by Lemma \ref{lem.Fi/Fi-1.L2}, applied to $k=\ell-1$). This proves Lemma
\ref{lem.Fi/Fi-1.L3} \textbf{(c)}.
\end{proof}

\subsection{Linear algebra lemmas}

We shall furthermore use two facts from linear algebra over any commutative
ring $\mathbf{k}$:

\begin{lemma}
\label{lem.surj=bij}Let $s\in\mathbb{N}$. Let $M$ and $N$ be two free
$\mathbf{k}$-modules of rank $s$. Then, any surjective $\mathbf{k}$-linear map
$\rho:M\rightarrow N$ is an isomorphism.
\end{lemma}

\begin{proof}
This is a well-known folklore result, and follows easily from the known fact
(\textquotedblleft Orzech's theorem\textquotedblright\ in one of its simplest
forms -- see, e.g., \cite[Exercise 2.5.18 (a)]{GriRei}, or \cite[Corollary
0.2]{Grinberg-Orzech} for a more general result) that any surjective
endomorphism of a free $\mathbf{k}$-module of finite rank is an isomorphism.
For the sake of self-containedness, let me nevertheless give a direct proof:

Let $\rho:M\rightarrow N$ be a surjective $\mathbf{k}$-linear map. We must
show that $\rho$ is an isomorphism.

Pick bases $\left(  e_{1},e_{2},\ldots,e_{s}\right)  $ and $\left(
f_{1},f_{2},\ldots,f_{s}\right)  $ of $M$ and $N$ (these exist, since $M$ and
$N$ are free of rank $s$). The surjectivity of $\rho$ then shows that every
basis vector $f_{i}$ of $N$ lies in the image of $\rho$. That is, every
$f_{i}$ can be written as $\rho\left(  g_{i}\right)  $ for some vector
$g_{i}\in M$. Choose such vectors $g_{i}$, and let $\tau:N\rightarrow M$ be
the $\mathbf{k}$-linear map that sends the basis vectors $f_{1},f_{2}%
,\ldots,f_{s}$ to $g_{1},g_{2},\ldots,g_{s}$, respectively. Then, the
composition $\rho\circ\tau:N\rightarrow N$ sends each vector $f_{i}$ to
$f_{i}$ (since $f_{i}\overset{\tau}{\mapsto}g_{i}\overset{\rho}{\mapsto}%
\rho\left(  g_{i}\right)  =f_{i}$), and thus is the identity map
$\operatorname*{id}\nolimits_{N}$ (since $f_{1},f_{2},\ldots,f_{s}$ form a
basis of $N$).

Now, let $A\in\mathbf{k}^{s\times s}$ be the matrix that represents the linear
map $\rho:M\rightarrow N$ with respect to our bases of $M$ and $N$. Likewise,
let $B\in\mathbf{k}^{s\times s}$ be the matrix that represents the linear map
$\tau:N\rightarrow M$ with respect to our bases of $N$ and $M$. Then, $AB$ is
the matrix that represents the linear map $\rho\circ\tau:N\rightarrow N$ with
respect to our basis of $N$. Therefore, $AB$ is the identity matrix $I_{s}$
(since $\rho\circ\tau$ is the identity map $\operatorname*{id}\nolimits_{N}$,
which is represented by the identity matrix $I_{s}$). Hence, $\det\left(
AB\right)  =\det\left(  I_{s}\right)  =1$ and thus $1=\det\left(  AB\right)
=\det A\cdot\det B$. This shows that $\det A$ is invertible (with inverse
$\det B$). Hence, the matrix $A$ is invertible (with inverse $\dfrac{1}{\det
A}\operatorname*{adj}A$, by the well-known identity $\det A\cdot I_{n}%
=A\cdot\operatorname*{adj}A=\operatorname*{adj}A\cdot A$). In other words, the
$\mathbf{k}$-linear map $\rho$ is invertible (since it is represented by the
matrix $A$), thus an isomorphism. This proves Lemma \ref{lem.surj=bij}.
\end{proof}

\begin{lemma}
\label{lem.tensor-of-quots}Let $V_{1},V_{2},\ldots,V_{m}$ be any $\mathbf{k}%
$-modules. For each $\ell\in\left[  m\right]  $, let $W_{\ell}$ be a
$\mathbf{k}$-submodule of $V_{\ell}$. For each $\ell\in\left[  m\right]  $, we
consider the $\mathbf{k}$-submodule
\[
\underbrace{V_{1}\otimes V_{2}\otimes\cdots\otimes W_{\ell}\otimes
\cdots\otimes V_{m}}_{\substack{\text{This means the tensor product }%
V_{1}\otimes V_{2}\otimes\cdots\otimes V_{m}\text{,}\\\text{in which the }%
\ell\text{-th factor is replaced by }W_{\ell}}}\text{ of }V_{1}\otimes
V_{2}\otimes\cdots\otimes V_{m}.
\]

Then, there is a canonical $\mathbf{k}$-module isomorphism%
\begin{align*}
&  \left(  V_{1}\otimes V_{2}\otimes\cdots\otimes V_{m}\right)  \diagup
\sum_{\ell=1}^{m}\ \ \underbrace{\left(  V_{1}\otimes V_{2}\otimes
\cdots\otimes W_{\ell}\otimes\cdots\otimes V_{m}\right)  }%
_{\substack{\text{This means the tensor product }V_{1}\otimes V_{2}%
\otimes\cdots\otimes V_{m}\text{,}\\\text{in which the }\ell\text{-th factor
is replaced by }W_{\ell}}}\\
&  \cong\left(  V_{1}/W_{1}\right)  \otimes\left(  V_{2}/W_{2}\right)
\otimes\cdots\otimes\left(  V_{m}/W_{m}\right)  .
\end{align*}

\end{lemma}

\begin{noncompile}
\textbf{(b)} If $V_{1},V_{2},\ldots,V_{m}$ are furthermore left modules over
some $\mathbf{k}$-algebras $A_{1},A_{2},\ldots,A_{m}$, and if $W_{1}%
,W_{2},\ldots,W_{m}$ are their submodules (i.e., each $W_{\ell}$ is a left
$A_{\ell}$-submodule of $V_{\ell}$), then this isomorphism is a left
$A_{1}\otimes A_{2}\otimes\cdots\otimes A_{m}$-module isomorphism.
\end{noncompile}

\begin{proof}
We construct both the isomorphism and its inverse using the universal
properties of tensor products and quotients:

\begin{itemize}
\item There is a canonical $\mathbf{k}$-linear map%
\[
\Phi:V_{1}\otimes V_{2}\otimes\cdots\otimes V_{m}\rightarrow\left(
V_{1}/W_{1}\right)  \otimes\left(  V_{2}/W_{2}\right)  \otimes\cdots
\otimes\left(  V_{m}/W_{m}\right)  ,
\]
sending each pure tensor $v_{1}\otimes v_{2}\otimes\cdots\otimes v_{m}$ to
$\overline{v_{1}}\otimes\overline{v_{2}}\otimes\cdots\otimes\overline{v_{m}}$.
This $\mathbf{k}$-linear map $\Phi$ is easily seen to vanish on the submodule
$\sum_{\ell=1}^{m}\left(  V_{1}\otimes V_{2}\otimes\cdots\otimes W_{\ell
}\otimes\cdots\otimes V_{m}\right)  $, and thus factors through the quotient
module. Hence, we obtain a $\mathbf{k}$-linear map%
\begin{align*}
\overline{\Phi}  &  :\left(  V_{1}\otimes V_{2}\otimes\cdots\otimes
V_{m}\right)  \diagup\sum_{\ell=1}^{m}\left(  V_{1}\otimes V_{2}\otimes
\cdots\otimes W_{\ell}\otimes\cdots\otimes V_{m}\right) \\
&  \rightarrow\left(  V_{1}/W_{1}\right)  \otimes\left(  V_{2}/W_{2}\right)
\otimes\cdots\otimes\left(  V_{m}/W_{m}\right)
\end{align*}
sending each $\overline{v_{1}\otimes v_{2}\otimes\cdots\otimes v_{m}}$ to
$\overline{v_{1}}\otimes\overline{v_{2}}\otimes\cdots\otimes\overline{v_{m}}$.

\item Conversely, there is a canonical $\mathbf{k}$-linear map%
\begin{align*}
\Psi &  :\left(  V_{1}/W_{1}\right)  \otimes\left(  V_{2}/W_{2}\right)
\otimes\cdots\otimes\left(  V_{m}/W_{m}\right) \\
&  \rightarrow\left(  V_{1}\otimes V_{2}\otimes\cdots\otimes V_{m}\right)
\diagup\sum_{\ell=1}^{m}\left(  V_{1}\otimes V_{2}\otimes\cdots\otimes
W_{\ell}\otimes\cdots\otimes V_{m}\right)
\end{align*}
sending each $\overline{v_{1}}\otimes\overline{v_{2}}\otimes\cdots
\otimes\overline{v_{m}}$ to $\overline{v_{1}\otimes v_{2}\otimes\cdots\otimes
v_{m}}$. To show that this map is well-defined, we need to check that
$\overline{v_{1}\otimes v_{2}\otimes\cdots\otimes v_{m}}$ depends only on the
residue classes $\overline{v_{i}}$ rather than on the $v_{i}$ themselves (this
is easy: replacing $v_{i}$ by $v_{i}^{\prime}$ with $v_{i}-v_{i}^{\prime}\in
W_{i}$ only changes $v_{1}\otimes v_{2}\otimes\cdots\otimes v_{m}$ by an
element of $V_{1}\otimes V_{2}\otimes\cdots\otimes W_{i}\otimes\cdots\otimes
V_{m}$) and that this dependence is multilinear (this is again easy).
\end{itemize}

Clearly, the maps $\overline{\Phi}$ and $\Psi$ are mutually inverse, hence
isomorphisms. Thus, Lemma \ref{lem.tensor-of-quots} is proved.
\end{proof}

\begin{noncompile}
\textbf{(b)} This follows easily from the construction in part \textbf{(a)}.
\end{noncompile}

For our specific needs, we specialize Lemma \ref{lem.tensor-of-quots} to the
case $W_{1}=0$:

\begin{lemma}
\label{lem.tensor-of-quots2}Let $V_{1},V_{2},\ldots,V_{m}$ be any $\mathbf{k}%
$-modules with $m\geq1$. For each $\ell\in\left[  2,m\right]  $, let $W_{\ell
}$ be a $\mathbf{k}$-submodule of $V_{\ell}$. For each $\ell\in\left[
2,m\right]  $, we consider the $\mathbf{k}$-submodule
\[
\underbrace{V_{1}\otimes V_{2}\otimes\cdots\otimes W_{\ell}\otimes
\cdots\otimes V_{m}}_{\substack{\text{This means the tensor product }%
V_{1}\otimes V_{2}\otimes\cdots\otimes V_{m}\text{,}\\\text{in which the }%
\ell\text{-th factor is replaced by }W_{\ell}}}\text{ of }V_{1}\otimes
V_{2}\otimes\cdots\otimes V_{m}.
\]

Then, there is a canonical $\mathbf{k}$-module isomorphism%
\begin{align*}
&  \left(  V_{1}\otimes V_{2}\otimes\cdots\otimes V_{m}\right)  \diagup
\sum_{\ell=2}^{m}\ \ \underbrace{\left(  V_{1}\otimes V_{2}\otimes
\cdots\otimes W_{\ell}\otimes\cdots\otimes V_{m}\right)  }%
_{\substack{\text{This means the tensor product }V_{1}\otimes V_{2}%
\otimes\cdots\otimes V_{m}\text{,}\\\text{in which the }\ell\text{-th factor
is replaced by }W_{\ell}}}\\
&  \cong V_{1}\otimes\left(  V_{2}/W_{2}\right)  \otimes\left(  V_{3}%
/W_{3}\right)  \otimes\cdots\otimes\left(  V_{m}/W_{m}\right)  .
\end{align*}

\end{lemma}

\begin{noncompile}
\textbf{(b)} If $V_{1},V_{2},\ldots,V_{m}$ are furthermore left modules over
some $\mathbf{k}$-algebras $A_{1},A_{2},\ldots,A_{m}$, and if each $W_{\ell}$
with $\ell\in\left[  2,m\right]  $ is a left $A_{\ell}$-submodule of $V_{\ell
}$, then this isomorphism is a left $A_{1}\otimes A_{2}\otimes\cdots\otimes
A_{m}$-module isomorphism.
\end{noncompile}

\begin{proof}
Apply Lemma \ref{lem.tensor-of-quots} to $W_{1}=0$, and observe that
$V_{1}/0\cong V_{1}$.
\end{proof}

\subsection{Proof of Theorem \ref{thm.Fi/Fi-1.as-ind}}

We can now prove Theorem \ref{thm.Fi/Fi-1.as-ind}:

\begin{proof}
[Proof of Theorem \ref{thm.Fi/Fi-1.as-ind}.]We shall use the notations of
Lemma \ref{lem.Fi/Fi-1.L3}. Note that each $k\in\left[  m\right]  $ satisfies
$J_{k}=\left[  i_{k-1},\ i_{k}-1\right]  $ and thus
\begin{equation}
\left\vert J_{k}\right\vert =i_{k}-i_{k-1}=j_{k}.
\label{pf.thm.Fi/Fi-1.as-ind.jj}%
\end{equation}
Explicitly, there is a bijection%
\begin{align}
\left[  j_{k}\right]   &  \rightarrow J_{k},\nonumber\\
x  &  \mapsto i_{k-1}-1+x \label{pf.thm.Fi/Fi-1.as-ind.bij}%
\end{align}
for each $k\in\left[  m\right]  $.

Consider the tensor product $\mathcal{H}_{j_{1}}\otimes\mathcal{N}_{j_{2}%
}\otimes\mathcal{N}_{j_{3}}\otimes\cdots\otimes\mathcal{N}_{j_{m}}$. We recall
that the trivial representation $\mathcal{H}_{j_{1}}=\mathbf{k}$ has a
$1$-element basis $\left(  1\right)  $, while each natural representation
$\mathcal{N}_{j_{k}}$ has basis $\left(  e_{p}\right)  _{p\in\left[
j_{k}\right]  }=\left(  e_{1},e_{2},\ldots,e_{j_{k}}\right)  $. However, by
abuse of notation, we shall rename the latter basis of $\mathcal{N}_{j_{k}}$
as $\left(  e_{p}\right)  _{p\in J_{k}}=\left(  e_{i_{k-1}},e_{i_{k-1}%
+1},\ldots,e_{i_{k}-1}\right)  $ instead (by shifting all subscripts up by
$i_{k-1}-1$, that is, renaming each basis vector $e_{x}$ as $e_{i_{k-1}-1+x}%
$). Note that this can be done because $j_{k}=i_{k}-i_{k-1}$.

Having renamed the basis vectors of the $\mathbf{k}$-module $\mathcal{N}%
_{j_{k}}$, let us also replace the symmetric group $S_{j_{k}}$ acting on this
module accordingly. Namely, we reinterpret the symmetric group $S_{j_{k}}$
acting on $\mathcal{N}_{j_{k}}$ as the symmetric group $S_{J_{k}}$ using the
bijection (\ref{pf.thm.Fi/Fi-1.as-ind.bij}) between the corresponding sets
$\left[  j_{k}\right]  $ and $J_{k}$. Thus, the left action of $S_{j_{k}}$ on
$\mathcal{N}_{j_{k}}$ becomes a left action of $S_{J_{k}}$ instead; it is
still a permutation action (given on our now-renamed basis by the formula
$\sigma e_{p}=e_{\sigma\left(  p\right)  }$ for each $p\in J_{k}$ and
$\sigma\in S_{J_{k}}$). With these reinterpretations, the parabolic embedding
$S_{j_{1}}\times S_{j_{2}}\times\cdots\times S_{j_{m}}\rightarrow S_{n}$
becomes the usual embedding $S_{J_{1}}\times S_{J_{2}}\times\cdots\times
S_{J_{m}}\rightarrow S_{n}$, which simply combines the $m$ permutations
without any need for shifting (i.e., any $m$-tuple $\left(  \sigma_{1}%
,\sigma_{2},\ldots,\sigma_{m}\right)  \in S_{J_{1}}\times S_{J_{2}}%
\times\cdots\times S_{J_{m}}$ is identified with the permutation $\sigma\in
S_{n}$ that sends each element $x\in J_{k}$ to $\sigma_{k}\left(  x\right)  $
for each $k\in\left[  m\right]  $).

For each $\mathbf{p}=\left(  p_{2},p_{3},\ldots,p_{m}\right)  \in J_{2}\times
J_{3}\times\cdots\times J_{m}$, we have%
\begin{align*}
\nabla_{\mathbf{p}}  &  \in F\left(  Q_{i}\right)  \ \ \ \ \ \ \ \ \ \ \left(
\text{by Lemma \ref{lem.Fi/Fi-1.L3} \textbf{(b)}}\right) \\
&  \subseteq F_{i}\ \ \ \ \ \ \ \ \ \ \left(  \text{since }F_{i}=F\left(
Q_{1}\right)  +F\left(  Q_{2}\right)  +\cdots+F\left(  Q_{i}\right)  \right)
\end{align*}
and thus $\overline{\nabla_{\mathbf{p}}}\in F_{i}/F_{i-1}$ (where
$\overline{\nabla_{\mathbf{p}}}$ denotes the residue class of $\nabla
_{\mathbf{p}}\in F_{i}$ in the quotient $F_{i}/F_{i-1}$). Hence, we can define
a $\mathbf{k}$-linear map%
\begin{align*}
\Phi:\mathcal{H}_{j_{1}}\otimes\mathcal{N}_{j_{2}}\otimes\mathcal{N}_{j_{3}%
}\otimes\cdots\otimes\mathcal{N}_{j_{m}}  &  \rightarrow F_{i}/F_{i-1},\\
1\otimes e_{p_{2}}\otimes e_{p_{3}}\otimes\cdots\otimes e_{p_{m}}  &
\mapsto\overline{\nabla_{\mathbf{p}}}\\
\ \ \ \ \ \ \ \ \ \ \text{for any }\mathbf{p}  &  =\left(  p_{2},p_{3}%
,\ldots,p_{m}\right)  \in J_{2}\times J_{3}\times\cdots\times J_{m}.
\end{align*}
(This map is defined by linearity, since the pure tensors of the form
$1\otimes e_{p_{2}}\otimes e_{p_{3}}\otimes\cdots\otimes e_{p_{m}}$ with
$\mathbf{p}=\left(  p_{2},p_{3},\ldots,p_{m}\right)  \in J_{2}\times
J_{3}\times\cdots\times J_{m}$ form a basis of the $\mathbf{k}$-module
$\mathcal{H}_{j_{1}}\otimes\mathcal{N}_{j_{2}}\otimes\mathcal{N}_{j_{3}%
}\otimes\cdots\otimes\mathcal{N}_{j_{m}}$.) Consider this map $\Phi$.

For each $\ell\in\left[  2,m\right]  $, we can consider the $\mathbf{k}%
$-submodule $\mathcal{H}_{j_{1}}\otimes\mathcal{N}_{j_{2}}\otimes
\mathcal{N}_{j_{3}}\otimes\cdots\otimes\mathcal{D}_{j_{\ell}}\otimes
\cdots\otimes\mathcal{N}_{j_{m}}$ of $\mathcal{H}_{j_{1}}\otimes
\mathcal{N}_{j_{2}}\otimes\mathcal{N}_{j_{3}}\otimes\cdots\otimes
\mathcal{N}_{j_{m}}$, in which its $\ell$-th factor $\mathcal{N}_{j_{\ell}}$
is replaced by its submodule $\mathcal{D}_{j_{\ell}}=\left\{  \left(
a,a,\ldots,a\right)  \ \mid\ a\in\mathbf{k}\right\}  $. We claim that the map
$\Phi$ sends this submodule to $0$. Indeed, this submodule is spanned by sums
of the form%
\[
\sum_{p_{\ell}\in J_{\ell}}1\otimes e_{p_{2}}\otimes e_{p_{3}}\otimes
\cdots\otimes e_{p_{m}}%
\]
(for fixed $p_{2},p_{3},\ldots,p_{\ell-1},p_{\ell+1},\ldots,p_{m}$ in the
respective intervals $J_{k}$)\ \ \ \ \footnote{\textit{Proof.} The submodule
$\mathcal{D}_{j_{\ell}}$ is spanned by the single vector
\[
\left(  1,1,\ldots,1\right)  =e_{i_{\ell-1}}+e_{i_{\ell-1}+1}+\cdots
+e_{i_{\ell}-1}=\sum_{p_{\ell}\in J_{\ell}}e_{p_{\ell}},
\]
and thus the tensor product $\mathcal{H}_{j_{1}}\otimes\mathcal{N}_{j_{2}%
}\otimes\mathcal{N}_{j_{3}}\otimes\cdots\otimes\mathcal{D}_{j_{\ell}}%
\otimes\cdots\otimes\mathcal{N}_{j_{m}}$ is spanned by the pure tensors of the
form%
\begin{align*}
&  1\otimes e_{p_{2}}\otimes e_{p_{3}}\otimes\cdots\otimes e_{p_{\ell-1}%
}\otimes\left(  \sum_{p_{\ell}\in J_{\ell}}e_{p_{\ell}}\right)  \otimes
e_{p_{\ell+1}}\otimes\cdots\otimes e_{p_{m}}\\
&  =\sum_{p_{\ell}\in J_{\ell}}1\otimes e_{p_{2}}\otimes e_{p_{3}}%
\otimes\cdots\otimes e_{p_{m}}\ \ \ \ \ \ \ \ \ \ \text{for fixed }p_{2}%
,p_{3},\ldots,p_{\ell-1},p_{\ell+1},\ldots,p_{m}.
\end{align*}
}, and the map $\Phi$ sends such sums to%
\[
\sum_{p_{\ell}\in J_{\ell}}\overline{\nabla_{\left(  p_{2},p_{3},\ldots
,p_{m}\right)  }}=\overline{\sum_{p_{\ell}\in J_{\ell}}\nabla_{\left(
p_{2},p_{3},\ldots,p_{m}\right)  }}=0_{F_{i}/F_{i-1}},
\]
since Lemma \ref{lem.Fi/Fi-1.L3} \textbf{(c)} shows that $\sum_{p_{\ell}\in
J_{\ell}}\nabla_{\left(  p_{2},p_{3},\ldots,p_{m}\right)  }\in F_{i-1}$.

Thus, the $\mathbf{k}$-linear map%
\[
\Phi:\mathcal{H}_{j_{1}}\otimes\mathcal{N}_{j_{2}}\otimes\mathcal{N}_{j_{3}%
}\otimes\cdots\otimes\mathcal{N}_{j_{m}}\rightarrow F_{i}/F_{i-1}%
\]
sends all the $\mathbf{k}$-submodules $\mathcal{H}_{j_{1}}\otimes
\mathcal{N}_{j_{2}}\otimes\mathcal{N}_{j_{3}}\otimes\cdots\otimes
\mathcal{D}_{j_{\ell}}\otimes\cdots\otimes\mathcal{N}_{j_{m}}$ for $\ell
\in\left[  2,m\right]  $ to $0$. By linearity, we can thus conclude that
$\Phi$ also sends their sum \newline$\sum_{\ell=2}^{m}\left(  \mathcal{H}%
_{j_{1}}\otimes\mathcal{N}_{j_{2}}\otimes\mathcal{N}_{j_{3}}\otimes
\cdots\otimes\mathcal{D}_{j_{\ell}}\otimes\cdots\otimes\mathcal{N}_{j_{m}%
}\right)  $ to $0$. Therefore, $\Phi$ factors through the quotient
$\mathbf{k}$-module%
\begin{align*}
&  \left(  \mathcal{H}_{j_{1}}\otimes\mathcal{N}_{j_{2}}\otimes\mathcal{N}%
_{j_{3}}\otimes\cdots\otimes\mathcal{N}_{j_{m}}\right)  \diagup\sum_{\ell
=2}^{m}\left(  \mathcal{H}_{j_{1}}\otimes\mathcal{N}_{j_{2}}\otimes
\mathcal{N}_{j_{3}}\otimes\cdots\otimes\mathcal{D}_{j_{\ell}}\otimes
\cdots\otimes\mathcal{N}_{j_{m}}\right) \\
&  \cong\mathcal{H}_{j_{1}}\otimes\left(  \mathcal{N}_{j_{2}}/\mathcal{D}%
_{j_{2}}\right)  \otimes\left(  \mathcal{N}_{j_{3}}/\mathcal{D}_{j_{3}%
}\right)  \otimes\cdots\otimes\left(  \mathcal{N}_{j_{m}}/\mathcal{D}_{j_{m}%
}\right) \\
&  \ \ \ \ \ \ \ \ \ \ \ \ \ \ \ \ \ \ \ \ \left(
\begin{array}
[c]{c}%
\text{by Lemma \ref{lem.tensor-of-quots2}, applied to }V_{1}=\mathcal{H}%
_{j_{1}}\text{ and }V_{\ell}=\mathcal{N}_{j_{\ell}}\text{ for }\ell>1\\
\text{and }W_{\ell}=\mathcal{D}_{j_{\ell}}\text{ for }\ell>1
\end{array}
\right) \\
&  =\mathcal{H}_{j_{1}}\otimes\mathcal{Z}_{j_{2}}\otimes\mathcal{Z}_{j_{3}%
}\otimes\cdots\otimes\mathcal{Z}_{j_{m}}\ \ \ \ \ \ \ \ \ \ \left(
\text{since }\mathcal{N}_{p}/\mathcal{D}_{p}=\mathcal{Z}_{p}\text{ for each
}p>0\right)  .
\end{align*}
Thus, we obtain a $\mathbf{k}$-linear map%
\begin{align*}
\overline{\Phi}:\mathcal{H}_{j_{1}}\otimes\mathcal{Z}_{j_{2}}\otimes
\mathcal{Z}_{j_{3}}\otimes\cdots\otimes\mathcal{Z}_{j_{m}}  &  \rightarrow
F_{i}/F_{i-1},\\
1\otimes\overline{e_{p_{2}}}\otimes\overline{e_{p_{3}}}\otimes\cdots
\otimes\overline{e_{p_{m}}}  &  \mapsto\overline{\nabla_{\mathbf{p}}}\\
\ \ \ \ \ \ \ \ \ \ \text{for any }\mathbf{p}  &  =\left(  p_{2},p_{3}%
,\ldots,p_{m}\right)  \in J_{2}\times J_{3}\times\cdots\times J_{m}.
\end{align*}
Consider this map $\overline{\Phi}$. Using Lemma \ref{lem.Fi/Fi-1.L3}
\textbf{(a)}, it is easy to see that this map $\overline{\Phi}$ is $S_{j_{1}%
}\times S_{j_{2}}\times\cdots\times S_{j_{m}}$%
-equivariant\footnote{\textit{Proof.} Let $\tau=\left(  \tau_{1},\tau
_{2},\ldots,\tau_{m}\right)  \in S_{j_{1}}\times S_{j_{2}}\times\cdots\times
S_{j_{m}}$ be any $m$-tuple, and let $\mathbf{p}=\left(  p_{2},p_{3}%
,\ldots,p_{m}\right)  \in J_{2}\times J_{3}\times\cdots\times J_{m}$. We shall
show that%
\[
\overline{\Phi}\left(  \tau\cdot\left(  1\otimes\overline{e_{p_{2}}}%
\otimes\overline{e_{p_{3}}}\otimes\cdots\otimes\overline{e_{p_{m}}}\right)
\right)  =\tau\cdot\overline{\Phi}\left(  1\otimes\overline{e_{p_{2}}}%
\otimes\overline{e_{p_{3}}}\otimes\cdots\otimes\overline{e_{p_{m}}}\right)  .
\]
By linearity, this will entail that the map $\overline{\Phi}$ is $S_{j_{1}%
}\times S_{j_{2}}\times\cdots\times S_{j_{m}}$-equivariant (since elements of
the form $1\otimes\overline{e_{p_{2}}}\otimes\overline{e_{p_{3}}}\otimes
\cdots\otimes\overline{e_{p_{m}}}$ span $\mathcal{H}_{j_{1}}\otimes
\mathcal{Z}_{j_{2}}\otimes\mathcal{Z}_{j_{3}}\otimes\cdots\otimes
\mathcal{Z}_{j_{m}}$).
\par
Indeed, as we mentioned at the beginning of our proof, we regard each
$S_{j_{k}}$ as $S_{J_{k}}$, so that the permutations $\tau_{1},\tau_{2}%
,\ldots,\tau_{m}$ act not on the sets $\left[  j_{1}\right]  ,\left[
j_{2}\right]  ,\ldots,\left[  j_{m}\right]  $ but rather on the sets
$J_{1},J_{2},\ldots,J_{m}$. The embedding of $S_{J_{1}}\times S_{J_{2}}%
\times\cdots\times S_{J_{m}}$ into $S_{n}$ is the usual one, so that our
$m$-tuple $\tau=\left(  \tau_{1},\tau_{2},\ldots,\tau_{m}\right)  $ is equated
with the permutation $\tau\in S_{n}$ given by%
\begin{equation}
\tau\left(  x\right)  =\tau_{k}\left(  x\right)  \ \ \ \ \ \ \ \ \ \ \text{for
each }k\in\left[  m\right]  \text{ and }x\in J_{k}.
\label{pf.thm.Fi/Fi-1.as-ind.fn.six=}%
\end{equation}
As in Lemma \ref{lem.Fi/Fi-1.L3} \textbf{(a)}, we set%
\[
\tau\mathbf{p}:=\left(  \tau_{2}\left(  p_{2}\right)  ,\tau_{3}\left(
p_{3}\right)  ,\ldots,\tau_{m}\left(  p_{m}\right)  \right)  =\left(
\tau\left(  p_{2}\right)  ,\tau\left(  p_{3}\right)  ,\ldots,\tau\left(
p_{m}\right)  \right)  .
\]
\par
Now, we have $\tau=\left(  \tau_{1},\tau_{2},\ldots,\tau_{m}\right)  $ and
thus%
\begin{align*}
\tau\cdot\left(  1\otimes\overline{e_{p_{2}}}\otimes\overline{e_{p_{3}}%
}\otimes\cdots\otimes\overline{e_{p_{m}}}\right)   &  =\tau_{1}1\otimes
\tau_{2}\overline{e_{p_{2}}}\otimes\tau_{3}\overline{e_{p_{3}}}\otimes
\cdots\otimes\tau_{m}\overline{e_{p_{m}}}\\
&  =1\otimes\overline{e_{\tau_{2}\left(  p_{2}\right)  }}\otimes
\overline{e_{\tau_{3}\left(  p_{3}\right)  }}\otimes\cdots\otimes
\overline{e_{\tau_{m}\left(  p_{m}\right)  }}.
\end{align*}
Therefore,%
\begin{align*}
&  \overline{\Phi}\left(  \tau\cdot\left(  1\otimes\overline{e_{p_{2}}}%
\otimes\overline{e_{p_{3}}}\otimes\cdots\otimes\overline{e_{p_{m}}}\right)
\right) \\
&  =\overline{\Phi}\left(  1\otimes\overline{e_{\tau_{2}\left(  p_{2}\right)
}}\otimes\overline{e_{\tau_{3}\left(  p_{3}\right)  }}\otimes\cdots
\otimes\overline{e_{\tau_{m}\left(  p_{m}\right)  }}\right) \\
&  =\overline{\nabla_{\tau\mathbf{p}}}\ \ \ \ \ \ \ \ \ \ \left(  \text{by the
definition of }\overline{\Phi}\text{, since }\tau\mathbf{p}=\left(  \tau
_{2}\left(  p_{2}\right)  ,\tau_{3}\left(  p_{3}\right)  ,\ldots,\tau
_{m}\left(  p_{m}\right)  \right)  \right) \\
&  =\overline{\tau\nabla_{\mathbf{p}}}\ \ \ \ \ \ \ \ \ \ \left(  \text{since
Lemma \ref{lem.Fi/Fi-1.L3} \textbf{(a)} yields }\nabla_{\tau\mathbf{p}}%
=\tau\nabla_{\mathbf{p}}\right) \\
&  =\tau\cdot\overline{\nabla_{\mathbf{p}}}=\tau\cdot\overline{\Phi}\left(
1\otimes\overline{e_{p_{2}}}\otimes\overline{e_{p_{3}}}\otimes\cdots
\otimes\overline{e_{p_{m}}}\right)
\end{align*}
(since the definition of $\overline{\Phi}$ yields $\overline{\nabla
_{\mathbf{p}}}=\overline{\Phi}\left(  1\otimes\overline{e_{p_{2}}}%
\otimes\overline{e_{p_{3}}}\otimes\cdots\otimes\overline{e_{p_{m}}}\right)
$). This is precisely what we wanted to show. Hence, we have proved that the
map $\overline{\Phi}$ is $S_{j_{1}}\times S_{j_{2}}\times\cdots\times
S_{j_{m}}$-equivariant.}, and thus is a left $\mathbf{k}\left[  S_{j_{1}%
}\times S_{j_{2}}\times\cdots\times S_{j_{m}}\right]  $-module morphism.

But the definition of an induction product yields%
\begin{align*}
&  \mathcal{H}_{j_{1}}\ast\mathcal{Z}_{j_{2}}\ast\mathcal{Z}_{j_{3}}\ast
\cdots\ast\mathcal{Z}_{j_{m}}\\
&  =\operatorname*{Ind}\nolimits_{S_{j_{1}}\times S_{j_{2}}\times\cdots\times
S_{j_{m}}}^{S_{n}}\left(  \mathcal{H}_{j_{1}}\otimes\mathcal{Z}_{j_{2}}%
\otimes\mathcal{Z}_{j_{3}}\otimes\cdots\otimes\mathcal{Z}_{j_{m}}\right) \\
&  =\underbrace{\mathbf{k}\left[  S_{n}\right]  }_{=\mathcal{A}}%
\otimes_{\mathbf{k}\left[  S_{j_{1}}\times S_{j_{2}}\times\cdots\times
S_{j_{m}}\right]  }\left(  \mathcal{H}_{j_{1}}\otimes\mathcal{Z}_{j_{2}%
}\otimes\mathcal{Z}_{j_{3}}\otimes\cdots\otimes\mathcal{Z}_{j_{m}}\right)
\ \ \ \ \ \ \ \ \ \ \left(  \text{by (\ref{eq.IndGH.def})}\right) \\
&  =\mathcal{A}\otimes_{\mathbf{k}\left[  S_{j_{1}}\times S_{j_{2}}%
\times\cdots\times S_{j_{m}}\right]  }\left(  \mathcal{H}_{j_{1}}%
\otimes\mathcal{Z}_{j_{2}}\otimes\mathcal{Z}_{j_{3}}\otimes\cdots
\otimes\mathcal{Z}_{j_{m}}\right)  .
\end{align*}
Hence, we can define a left $\mathcal{A}$-module morphism%
\begin{align*}
\Psi:\mathcal{H}_{j_{1}}\ast\mathcal{Z}_{j_{2}}\ast\mathcal{Z}_{j_{3}}%
\ast\cdots\ast\mathcal{Z}_{j_{m}}  &  \rightarrow F_{i}/F_{i-1},\\
a\otimes_{\mathbf{k}\left[  S_{j_{1}}\times S_{j_{2}}\times\cdots\times
S_{j_{m}}\right]  }v  &  \mapsto a\cdot\overline{\Phi}\left(  v\right)
\end{align*}
(this is well-defined, since $\overline{\Phi}:\mathcal{H}_{j_{1}}%
\otimes\mathcal{Z}_{j_{2}}\otimes\mathcal{Z}_{j_{3}}\otimes\cdots
\otimes\mathcal{Z}_{j_{m}}\rightarrow F_{i}/F_{i-1}$ is a left $\mathbf{k}%
\left[  S_{j_{1}}\times S_{j_{2}}\times\cdots\times S_{j_{m}}\right]  $-module
morphism). Explicitly, $\Psi$ is given by%
\begin{align*}
&  \Psi\left(  a\otimes_{\mathbf{k}\left[  S_{j_{1}}\times S_{j_{2}}%
\times\cdots\times S_{j_{m}}\right]  }\left(  1\otimes\overline{e_{p_{2}}%
}\otimes\overline{e_{p_{3}}}\otimes\cdots\otimes\overline{e_{p_{m}}}\right)
\right) \\
&  =a\cdot\overline{\Phi}\left(  1\otimes\overline{e_{p_{2}}}\otimes
\overline{e_{p_{3}}}\otimes\cdots\otimes\overline{e_{p_{m}}}\right) \\
&  =a\cdot\overline{\nabla_{\mathbf{p}}}\ \ \ \ \ \ \ \ \ \ \text{for any
}a\in\mathcal{A}\text{ and }\mathbf{p}=\left(  p_{2},p_{3},\ldots
,p_{m}\right)  \in J_{2}\times J_{3}\times\cdots\times J_{m}%
\end{align*}
(by the definition of $\overline{\Phi}$). Hence, using Lemma
\ref{lem.Fi/Fi-1.L3} \textbf{(b)}, it is easy to see that this map $\Psi$ is
surjective\footnote{\textit{Proof.} The map $\Psi$ is left $\mathcal{A}%
$-linear. Hence, its image is a left $\mathcal{A}$-submodule of $F_{i}%
/F_{i-1}$.
\par
By the definition of $F_{i}$, we have%
\begin{align*}
F_{i}  &  =F\left(  Q_{1}\right)  +F\left(  Q_{2}\right)  +\cdots+F\left(
Q_{i}\right) \\
&  =\underbrace{F\left(  Q_{1}\right)  +F\left(  Q_{2}\right)  +\cdots
+F\left(  Q_{i-1}\right)  }_{=F_{i-1}}+\,F\left(  Q_{i}\right)  =F_{i-1}%
+F\left(  Q_{i}\right)  .
\end{align*}
Hence, the composition of canonical maps
\begin{equation}
F\left(  Q_{i}\right)  \overset{\text{inclusion}}{\longrightarrow}%
F_{i}\overset{\text{projection}}{\longrightarrow}F_{i}/F_{i-1}
\label{pf.thm.Fi/Fi-1.as-ind.fn4.comp}%
\end{equation}
is surjective.
\par
But Lemma \ref{lem.Fi/Fi-1.L3} \textbf{(b)} shows that the left $\mathcal{A}%
$-module $F\left(  Q_{i}\right)  $ is generated by a single element of the
form $\nabla_{\mathbf{p}}$. Hence, the quotient $\mathcal{A}$-module
$F_{i}/F_{i-1}$ is generated by a single element of the form $\overline
{\nabla_{\mathbf{p}}}$ (since the map (\ref{pf.thm.Fi/Fi-1.as-ind.fn4.comp})
is surjective). But any such element of the form $\overline{\nabla
_{\mathbf{p}}}$ lies in the image of $\Psi$ (since we have $\overline
{\nabla_{\mathbf{p}}}=\Psi\left(  1\otimes_{\mathbf{k}\left[  S_{j_{1}}\times
S_{j_{2}}\times\cdots\times S_{j_{m}}\right]  }\left(  1\otimes\overline
{e_{p_{2}}}\otimes\overline{e_{p_{3}}}\otimes\cdots\otimes\overline{e_{p_{m}}%
}\right)  \right)  $ when $\mathbf{p}=\left(  p_{2},p_{3},\ldots,p_{m}\right)
$). Thus, the image of $\Psi$ must contain a generator of $F_{i}/F_{i-1}$, and
thus must be the entire $\mathcal{A}$-module $F_{i}/F_{i-1}$ (since this image
is a left $\mathcal{A}$-submodule of $F_{i}/F_{i-1}$). In other words, $\Psi$
is surjective.}.

We now know that $\Psi$ is a surjective left $\mathcal{A}$-module morphism
from $\mathcal{H}_{j_{1}}\ast\mathcal{Z}_{j_{2}}\ast\mathcal{Z}_{j_{3}}%
\ast\cdots\ast\mathcal{Z}_{j_{m}}$ to $F_{i}/F_{i-1}$. We shall now show that
$\Psi$ is an isomorphism.

From Lemma \ref{lem.Fi/Fi-1.dim}, we know that the $\mathbf{k}$-module
$F_{i}/F_{i-1}$ is free of rank%
\[
\dfrac{n!}{j_{1}!j_{2}!\cdots j_{m}!}\cdot\prod_{k=2}^{m}\left(
j_{k}-1\right)  .
\]
But the $\mathbf{k}$-module
\[
\mathcal{H}_{j_{1}}\ast\mathcal{Z}_{j_{2}}\ast\mathcal{Z}_{j_{3}}\ast
\cdots\ast\mathcal{Z}_{j_{m}}=\operatorname*{Ind}\nolimits_{S_{j_{1}}\times
S_{j_{2}}\times\cdots\times S_{j_{m}}}^{S_{n}}\left(  \mathcal{H}_{j_{1}%
}\otimes\mathcal{Z}_{j_{2}}\otimes\mathcal{Z}_{j_{3}}\otimes\cdots
\otimes\mathcal{Z}_{j_{m}}\right)
\]
is also free of rank\footnote{Here, we are denoting the rank of a free
$\mathbf{k}$-module $V$ by $\dim V$, and we are using the fact that an induced
representation $\operatorname*{Ind}\nolimits_{H}^{G}V$ is free of rank
$\dfrac{\left\vert G\right\vert }{\left\vert H\right\vert }\cdot\dim V$ (as a
$\mathbf{k}$-module) whenever $V$ is free (as a $\mathbf{k}$-module). (The
latter fact is an easy consequence of the fact that $\mathbf{k}\left[
G\right]  $ is a free right $\mathbf{k}\left[  H\right]  $-module of rank
$\dfrac{\left\vert G\right\vert }{\left\vert H\right\vert }$.)}%
\begin{align*}
&  \underbrace{\dfrac{\left\vert S_{n}\right\vert }{\left\vert S_{j_{1}}\times
S_{j_{2}}\times\cdots\times S_{j_{m}}\right\vert }}_{=\dfrac{n!}{j_{1}%
!j_{2}!\cdots j_{m}!}}\cdot\underbrace{\dim\left(  \mathcal{H}_{j_{1}}%
\otimes\mathcal{Z}_{j_{2}}\otimes\mathcal{Z}_{j_{3}}\otimes\cdots
\otimes\mathcal{Z}_{j_{m}}\right)  }_{\substack{=\dim\left(  \mathcal{H}%
_{j_{1}}\right)  \cdot\prod_{k=2}^{m}\dim\left(  \mathcal{Z}_{j_{k}}\right)
\\=1\cdot\prod_{k=2}^{m}\left(  j_{k}-1\right)  \\\text{(since }%
\mathcal{H}_{j_{1}}\text{ is free of rank }1\text{,}\\\text{whereas each
}\mathcal{Z}_{j_{k}}\text{ is free of rank }j_{k}-1\text{)}}}\\
&  =\dfrac{n!}{j_{1}!j_{2}!\cdots j_{m}!}\cdot\prod_{k=2}^{m}\left(
j_{k}-1\right)  .
\end{align*}
Thus, Lemma \ref{lem.surj=bij} (applied to $M=\mathcal{H}_{j_{1}}%
\ast\mathcal{Z}_{j_{2}}\ast\mathcal{Z}_{j_{3}}\ast\cdots\ast\mathcal{Z}%
_{j_{m}}$ and $N=F_{i}/F_{i-1}$ and $s=\dfrac{n!}{j_{1}!j_{2}!\cdots j_{m}%
!}\cdot\prod_{k=2}^{m}\left(  j_{k}-1\right)  $ and $\rho=\Psi$) shows that
the surjective $\mathbf{k}$-linear map $\Psi:\mathcal{H}_{j_{1}}%
\ast\mathcal{Z}_{j_{2}}\ast\mathcal{Z}_{j_{3}}\ast\cdots\ast\mathcal{Z}%
_{j_{m}}\rightarrow F_{i}/F_{i-1}$ must be an isomorphism. Since $\Psi$ is a
left $\mathcal{A}$-module morphism, we thus conclude that $\Psi$ is a left
$\mathcal{A}$-module isomorphism. Therefore, $F_{i}/F_{i-1}\cong%
\mathcal{H}_{j_{1}}\ast\mathcal{Z}_{j_{2}}\ast\mathcal{Z}_{j_{3}}\ast
\cdots\ast\mathcal{Z}_{j_{m}}$ as left $\mathcal{A}$-modules, i.e., as $S_{n}%
$-representations. Hence, Theorem \ref{thm.Fi/Fi-1.as-ind} is proved.
\end{proof}

\subsection{In terms of Littlewood--Richardson coefficients}

In the characteristic-$0$ case, we can restate the claim of Theorem
\ref{thm.Fi/Fi-1.as-ind} in terms of Littlewood--Richardson coefficients. Let
us first recount the bare minimum of symmetric function theory needed to state this.

We will use standard notations for (integer) partitions; in particular, the
size of a partition $\lambda$ will be denoted by $\left\vert \lambda
\right\vert $. We let $\operatorname*{Par}$ denote the set of all partitions.
We let $\Lambda$ be the ring of symmetric functions over $\mathbb{Z}$ (not
over $\mathbf{k}$); we refer to \cite[\S 2.1]{GriRei} or \cite[\S 4.3]%
{Sagan-SG} for its definition\footnote{Note that \cite[\S 4.3]{Sagan-SG} uses
$\mathbb{C}$ as the base ring, but everything works for any base ring.}. To
each partition $\lambda$ corresponds a special symmetric function $s_{\lambda
}\in\Lambda$ called the \emph{Schur function}; see \cite[(2.2.4)]{GriRei} or
\cite[\S 4.4]{Sagan-SG} or \cite[Definition 5.3]{Egge19} for its definition.
It is well-known (see \cite[(4.26) and Theorem 4.9.4]{Sagan-SG} or
\cite[Definition 2.5.8 and Corollary 2.6.12]{GriRei} or \cite[Theorem
10.40]{Egge19}) that a product $s_{\mu}s_{\nu}$ of two Schur functions (for
$\mu,\nu\in\operatorname*{Par}$) can always be written as an $\mathbb{N}%
$-linear combination of Schur functions -- i.e., there exist coefficients
$c_{\mu,\nu}^{\lambda}\in\mathbb{N}$ for all $\lambda,\mu,\nu\in
\operatorname*{Par}$ such that every two partitions $\mu$ and $\nu$ satisfy%
\begin{equation}
s_{\mu}s_{\nu}=\sum_{\lambda\in\operatorname*{Par}}c_{\mu,\nu}^{\lambda
}s_{\lambda}. \label{eq.LR.2}%
\end{equation}
These coefficients $c_{\mu,\nu}^{\lambda}$ are known as the
\emph{Littlewood--Richardson coefficients}. More generally, if $\mu_{1}%
,\mu_{2},\ldots,\mu_{k}$ are any $k$ partitions, then we can write the product
$s_{\mu_{1}}s_{\mu_{2}}\cdots s_{\mu_{k}}$ in the form%
\begin{equation}
s_{\mu_{1}}s_{\mu_{2}}\cdots s_{\mu_{k}}=\sum_{\lambda\in\operatorname*{Par}%
}c_{\mu_{1},\mu_{2},\ldots,\mu_{k}}^{\lambda}s_{\lambda} \label{eq.LR.k}%
\end{equation}
with coefficients $c_{\mu_{1},\mu_{2},\ldots,\mu_{k}}^{\lambda}\in\mathbb{N}$.
These \textquotedblleft$k$\emph{-Littlewood--Richardson coefficients}%
\textquotedblright\ $c_{\mu_{1},\mu_{2},\ldots,\mu_{k}}^{\lambda}$ are, in
fact, easily computed by recursion using the standard Littlewood--Richardson
coefficients $c_{\mu,\nu}^{\lambda}$: Namely, for $k=0$, we have%
\[
c^{\lambda}=\delta_{\lambda,\varnothing}\ \ \ \ \ \ \ \ \ \ \left(
\text{Kronecker delta}\right)  ;
\]
for $k=1$, we have%
\[
c_{\mu}^{\lambda}=\delta_{\lambda,\mu}\ \ \ \ \ \ \ \ \ \ \left(
\text{Kronecker delta}\right)  ;
\]
and for any higher $k$, we have%
\[
c_{\mu_{1},\mu_{2},\ldots,\mu_{k}}^{\lambda}=\sum_{\nu\in\operatorname*{Par}%
}c_{\mu_{1},\mu_{2},\ldots,\mu_{k-1}}^{\nu}c_{\nu,\mu_{k}}^{\lambda}%
\]
(since the product $s_{\mu_{1}}s_{\mu_{2}}\cdots s_{\mu_{k}}$ can be computed
as $\left(  s_{\mu_{1}}s_{\mu_{2}}\cdots s_{\mu_{k-1}}\right)  s_{\mu_{k}}$).

Note that any Schur function $s_{\lambda}$ is homogeneous of degree
$\left\vert \lambda\right\vert $. Hence, a Littlewood--Richardson coefficient
$c_{\mu,\nu}^{\lambda}$ is always $0$ unless $\left\vert \lambda\right\vert
=\left\vert \mu\right\vert +\left\vert \nu\right\vert $. Thus, we can rewrite
the equality (\ref{eq.LR.2}) as%
\begin{equation}
s_{\mu}s_{\nu}=\sum_{\substack{\lambda\in\operatorname*{Par};\\\left\vert
\lambda\right\vert =\left\vert \mu\right\vert +\left\vert \nu\right\vert
}}c_{\mu,\nu}^{\lambda}s_{\lambda}. \label{eq.LR.2n}%
\end{equation}
Likewise, we can rewrite (\ref{eq.LR.k}) as%
\begin{equation}
s_{\mu_{1}}s_{\mu_{2}}\cdots s_{\mu_{k}}=\sum_{\substack{\lambda
\in\operatorname*{Par};\\\left\vert \lambda\right\vert =\left\vert \mu
_{1}\right\vert +\left\vert \mu_{2}\right\vert +\cdots+\left\vert \mu
_{k}\right\vert }}c_{\mu_{1},\mu_{2},\ldots,\mu_{k}}^{\lambda}s_{\lambda}.
\label{eq.LR.kn}%
\end{equation}

We note that there is a second Littlewood--Richardson rule (\cite[Theorem
4.9.2]{Sagan-SG}, \cite[(2.6.4)]{GriRei}, \cite[Theorem 10.40]{Egge19}) that
decomposes a skew Schur function $s_{\lambda/\mu}$ into an $\mathbb{N}$-linear
combination of (straight) Schur functions $s_{\nu}$ as follows:%
\begin{equation}
s_{\lambda/\mu}=\sum_{\nu\in\operatorname*{Par}}c_{\mu,\nu}^{\lambda}s_{\nu}.
\label{eq.LR.lam/mu}%
\end{equation}
The formula (\ref{eq.LR.k}) can also be viewed as a particular case of that
second rule, since the product $s_{\mu_{1}}s_{\mu_{2}}\cdots s_{\mu_{k}}$ can
be written as the skew Schur function $s_{\mu_{1}\ast\mu_{2}\ast\cdots\ast
\mu_{k}}$ corresponding to the skew shape $\mu_{1}\ast\mu_{2}\ast\cdots\ast
\mu_{k}$ obtained by attaching the Young diagrams of $\mu_{1},\mu_{2}%
,\ldots,\mu_{k}$ to each other along their northeastern/southwestern corners
(see \cite[\S 1]{RemmelWhitney} for the precise definition; this claim follows
from \cite[Proposition 5.9]{Egge19}; cf. also \cite[Figure 7.2]{Stanley-EC2}).
Thus, a $k$-Littlewood--Richardson coefficient $c_{\mu_{1},\mu_{2},\ldots
,\mu_{k}}^{\lambda}$ can actually be rewritten as a (regular)
Littlewood--Richardson coefficient using (\ref{eq.LR.lam/mu}): If we write the
skew shape $\mu_{1}\ast\mu_{2}\ast\cdots\ast\mu_{k}$ as $\alpha/\beta$, then
\begin{equation}
c_{\mu_{1},\mu_{2},\ldots,\mu_{k}}^{\lambda}=c_{\beta,\lambda}^{\alpha}.
\label{eq.LR.c=c}%
\end{equation}

The same Littlewood--Richardson coefficients govern the decomposition of
induction products of Specht modules into Specht modules in characteristic
$0$. Namely, if $\mathbf{k}$ is a field of characteristic $0$, and if $\mu$
and $\nu$ are two partitions of respective sizes $i$ and $j$, then%
\begin{equation}
\mathcal{S}^{\mu}\ast\mathcal{S}^{\nu}\cong\bigoplus_{\substack{\lambda
\in\operatorname*{Par};\\\left\vert \lambda\right\vert =i+j}}\left(
\mathcal{S}^{\lambda}\right)  ^{\oplus c_{\mu,\nu}^{\lambda}}
\label{eq.LR.2Specht}%
\end{equation}
as $S_{i+j}$-modules\footnote{The notation $V^{\oplus k}$ means the direct sum
$V\oplus V\oplus\cdots\oplus V$ of $k$ copies of $V$.}. Indeed, this follows
from the Schur function equality (\ref{eq.LR.2n}) using the Frobenius
characteristic map \cite[Theorem 4.7.4]{Sagan-SG} (in fact, this map -- or,
rather, its inverse -- sends Schur functions $s_{\lambda}$ to Specht modules
$\mathcal{S}^{\lambda}$, while sending products of symmetric functions to
induction products of representations\footnote{For products with two factors,
this is proved in \cite[Theorem 4.7.4]{Sagan-SG} (using characters and
Frobenius reciprocity). For products with $k$ factors, it follows from the
two-factor case using Proposition \ref{prop.indprod.ass}.}). Likewise, if
$\mathbf{k}$ is a field of characteristic $0$, and if $\mu_{1},\mu_{2}%
,\ldots,\mu_{k}$ are $k$ partitions of respective sizes $i_{1},i_{2}%
,\ldots,i_{k}$, then%
\begin{equation}
\mathcal{S}^{\mu_{1}}\ast\mathcal{S}^{\mu_{2}}\ast\cdots\ast\mathcal{S}%
^{\mu_{k}}\cong\bigoplus_{\substack{\lambda\in\operatorname*{Par};\\\left\vert
\lambda\right\vert =i_{1}+i_{2}+\cdots+i_{k}}}\left(  \mathcal{S}^{\lambda
}\right)  ^{\oplus c_{\mu_{1},\mu_{2},\ldots,\mu_{k}}^{\lambda}}
\label{eq.LR.kSpecht}%
\end{equation}
as $S_{i_{1}+i_{2}+\cdots+i_{k}}$-modules. When $\mathbf{k}$ is a field of
characteristic $\neq0$, or, more generally, just any commutative ring, then
the isomorphisms (\ref{eq.LR.2Specht}) and (\ref{eq.LR.kSpecht}) still hold in
a weaker form, where the direct sums are replaced by filtrations whose
subquotients are Specht modules of the form $\mathcal{S}^{\lambda}$ (see
\cite[Theorem 9.7]{Clausen91} for the more general problem of finding such
filtrations of skew Specht modules\footnote{Indeed, the induction product
$\mathcal{S}^{\mu_{1}}\ast\mathcal{S}^{\mu_{2}}\ast\cdots\ast\mathcal{S}%
^{\mu_{k}}$ is easily seen to be isomorphic to the skew Specht module
$\mathcal{S}^{\alpha/\beta}$, where $\alpha/\beta=\mu_{1}\ast\mu_{2}\ast
\cdots\ast\mu_{k}$. Thus, \cite[Theorem 9.7]{Clausen91} (applied to this skew
diagram $\alpha/\beta$) yields a Specht series (i.e., a filtration whose
subquotients are Specht modules of the form $\mathcal{S}^{\lambda}$) for
$\mathcal{S}^{\mu_{1}}\ast\mathcal{S}^{\mu_{2}}\ast\cdots\ast\mathcal{S}%
^{\mu_{k}}$.}).

Now, Theorem \ref{thm.Fi/Fi-1.as-ind} gives rise to the following
decomposition of $F_{i}/F_{i-1}$:

\begin{corollary}
\label{cor.Fi/Fi-1.lr}Assume that $\mathbf{k}$ is a field of characteristic
$0$.

Let $i\in\left[  f_{n+1}\right]  $. Consider the lacunar subset $Q_{i}$ of
$\left[  n-1\right]  $. Write the set $Q_{i}\cup\left\{  n+1\right\}  $ as
$\left\{  i_{1}<i_{2}<\cdots<i_{m}\right\}  $, so that $i_{m}=n+1$.
Furthermore, set $i_{0}:=1$. Set $j_{k}:=i_{k}-i_{k-1}$ for each $k\in\left[
m\right]  $. (Note that $j_{1}\geq0$ and $j_{2},j_{3},\ldots,j_{m}>1$ and
$j_{1}+j_{2}+\cdots+j_{m}=n$; this follows from Lemma \ref{lem.Fi/Fi-1.sum-js}
(applied to $I=Q_{i}$).)

Let $\operatorname*{Par}\nolimits_{n}$ be the set of all partitions of $n$.
Then,%
\[
F_{i}/F_{i-1}\cong\bigoplus_{\lambda\in\operatorname*{Par}\nolimits_{n}%
}\left(  \mathcal{S}^{\lambda}\right)  ^{\oplus c_{\left(  j_{1}\right)
,\ \left(  j_{2}-1,1\right)  ,\ \left(  j_{3}-1,1\right)  ,\ \ldots,\ \left(
j_{m}-1,1\right)  }^{\lambda}}%
\]
(where the subscripts in $c_{\left(  j_{1}\right)  ,\ \left(  j_{2}%
-1,1\right)  ,\ \left(  j_{3}-1,1\right)  ,\ \ldots,\ \left(  j_{m}%
-1,1\right)  }^{\lambda}$ are the partition $\left(  j_{1}\right)  $ followed
by the partitions $\left(  j_{k}-1,1\right)  $ for all $k\in\left[
2,m\right]  $).
\end{corollary}

\begin{proof}
The set $Q_{i}\cup\left\{  n+1\right\}  $ is lacunar (since $Q_{i}$ is
lacunar, and since $Q_{i}\subseteq\left[  n-1\right]  $ ensures that $n+1$ is
larger than any element of $Q_{i}$ by at least $2$). In other words, every
$k>1$ satisfies $i_{k}>i_{k-1}+1$. Hence, every $k>1$ satisfies $j_{k}>1$
(since $j_{k}=i_{k}-i_{k-1}$) and therefore%
\begin{equation}
\mathcal{Z}_{j_{k}}\cong\mathcal{S}^{\left(  j_{k}-1,1\right)  }
\label{pf.cor.Fi/Fi-1.lr.Z}%
\end{equation}
(by Proposition \ref{prop.Z-iso} \textbf{(b)}, since $j_{k}$ is invertible in
$\mathbf{k}$).

Also, we have $\mathcal{H}_{p}\cong\mathcal{S}^{\left(  p\right)  }$ for each
$p\in\mathbb{N}$ (since both $\mathcal{H}_{p}$ and $\mathcal{S}^{\left(
p\right)  }$ are trivial $1$-dimensional representations of $S_{p}$). Thus,
$\mathcal{H}_{j_{1}}\cong\mathcal{S}^{\left(  j_{1}\right)  }$.

Now, Theorem \ref{thm.Fi/Fi-1.as-ind} yields%
\begin{align*}
F_{i}/F_{i-1}  &  \cong\underbrace{\mathcal{H}_{j_{1}}\ast\mathcal{Z}_{j_{2}%
}\ast\mathcal{Z}_{j_{3}}\ast\cdots\ast\mathcal{Z}_{j_{m}}}%
_{\substack{\text{the first factor is an }\mathcal{H}\text{,}\\\text{while all
others are }\mathcal{Z}\text{'s}}}\\
&  \cong\mathcal{S}^{\left(  j_{1}\right)  }\ast\mathcal{S}^{\left(
j_{2}-1,1\right)  }\ast\mathcal{S}^{\left(  j_{3}-1,1\right)  }\ast\cdots
\ast\mathcal{S}^{\left(  j_{m}-1,1\right)  }\\
&  \ \ \ \ \ \ \ \ \ \ \ \ \ \ \ \ \ \ \ \ \left(
\begin{array}
[c]{c}%
\text{since }\mathcal{H}_{j_{1}}\cong\mathcal{S}^{\left(  j_{1}\right)
}\text{, and since }\mathcal{Z}_{j_{k}}\cong\mathcal{S}^{\left(
j_{k}-1,1\right)  }\text{ for each }k>1\\
\text{(by (\ref{pf.cor.Fi/Fi-1.lr.Z}))}%
\end{array}
\right) \\
&  \cong\bigoplus_{\substack{\lambda\in\operatorname*{Par};\\\left\vert
\lambda\right\vert =j_{1}+j_{2}+\cdots+j_{m}}}\left(  \mathcal{S}^{\lambda
}\right)  ^{\oplus c_{\left(  j_{1}\right)  ,\ \left(  j_{2}-1,1\right)
,\ \left(  j_{3}-1,1\right)  ,\ \ldots,\ \left(  j_{m}-1,1\right)  }^{\lambda
}}\ \ \ \ \ \ \ \ \ \ \left(  \text{by (\ref{eq.LR.kSpecht})}\right) \\
&  =\bigoplus_{\substack{\lambda\in\operatorname*{Par};\\\left\vert
\lambda\right\vert =n}}\left(  \mathcal{S}^{\lambda}\right)  ^{\oplus
c_{\left(  j_{1}\right)  ,\ \left(  j_{2}-1,1\right)  ,\ \left(
j_{3}-1,1\right)  ,\ \ldots,\ \left(  j_{m}-1,1\right)  }^{\lambda}%
}\ \ \ \ \ \ \ \ \ \ \left(  \text{since }j_{1}+j_{2}+\cdots+j_{m}=n\right) \\
&  =\bigoplus_{\lambda\in\operatorname*{Par}\nolimits_{n}}\left(
\mathcal{S}^{\lambda}\right)  ^{\oplus c_{\left(  j_{1}\right)  ,\ \left(
j_{2}-1,1\right)  ,\ \left(  j_{3}-1,1\right)  ,\ \ldots,\ \left(
j_{m}-1,1\right)  }^{\lambda}}.
\end{align*}
This proves Corollary \ref{cor.Fi/Fi-1.lr}.
\end{proof}

\section{\label{sec.specht-spec}The Specht module spectrum}

\subsection{\label{subsec.specht-spec.thm}The theorem}

We need a few more notations from \cite{s2b1}. For any subset $I$ of $\left[
n\right]  $, we define the following:

\begin{itemize}
\item We let $\widehat{I}$ be the set $\left\{  0\right\}  \cup I\cup\left\{
n+1\right\}  $. We shall refer to $\widehat{I}$ as the \emph{enclosure} of $I$.

For example, if $n=5$, then $\widehat{\left\{  2,3\right\}  }=\left\{
0,2,3,6\right\}  $.

\item For any $\ell\in\left[  n\right]  $, we let $m_{I,\ell}$ be the number%
\[
\left(  \text{smallest element of }\widehat{I}\text{ that is }\geq\ell\right)
-\ell\in\left[  0,n+1-\ell\right]  \subseteq\left[  0,n\right]  .
\]

For example, if $n=6$ and $I=\left\{  2,5\right\}  $, then%
\[
\left(  m_{I,1},\ m_{I,2},\ m_{I,3},\ m_{I,4},\ m_{I,5},\ m_{I,6}\right)
=\left(  1,\ 0,\ 2,\ 1,\ 0,\ 1\right)  .
\]
We note that an $\ell\in\left[  n\right]  $ satisfies $m_{I,\ell}=0$ if and
only if $\ell\in\widehat{I}$ (or, equivalently, $\ell\in I$).
\end{itemize}

We recall that any partition $\lambda$ of $n$ gives rise to an $S_{n}%
$-representation called the Specht module $\mathcal{S}^{\lambda}$. If
$\lambda$ is a partition of $n$, and if $a\in\mathcal{A}$, then the action of
$a$ on $\mathcal{S}^{\lambda}$ (that is, the $\mathbf{k}$-linear map
$\mathcal{S}^{\lambda}\rightarrow\mathcal{S}^{\lambda},\ w\mapsto aw$) will be
denoted by $L_{\lambda}\left(  a\right)  $.

\begin{definition}
\label{def.clI}Let $\lambda$ be a partition of $n$. Let $I$ be a lacunar
subset of $\left[  n-1\right]  $. Write the set $I\cup\left\{  n+1\right\}  $
as $\left\{  i_{1}<i_{2}<\cdots<i_{m}\right\}  $, so that $i_{m}=n+1$.
Furthermore, set $i_{0}:=1$. Set $j_{k}:=i_{k}-i_{k-1}$ for each $k\in\left[
m\right]  $. (Note that Lemma \ref{lem.Fi/Fi-1.sum-js} shows that $j_{1}\geq0$
and $j_{2},j_{3},\ldots,j_{m}>1$, hence $j_{2},j_{3},\ldots,j_{m}\geq2$.)

The $m$-Littlewood--Richardson coefficient $c_{\left(  j_{1}\right)
,\ \left(  j_{2}-1,1\right)  ,\ \left(  j_{3}-1,1\right)  ,\ \ldots,\ \left(
j_{m}-1,1\right)  }^{\lambda}$ (as defined in (\ref{eq.LR.k}), where the
subscripts are the partition $\left(  j_{1}\right)  $ followed by the
partitions $\left(  j_{k}-1,1\right)  $ for all $k\in\left[  2,m\right]  $)
will then be denoted by $c_{I}^{\lambda}$.
\end{definition}

As we recall from Corollary \ref{cor.Fi/Fi-1.lr}, if $I=Q_{i}$ for some
$i\in\left[  f_{n+1}\right]  $, then this coefficient $c_{I}^{\lambda}$ is the
multiplicity of the Specht module $\mathcal{S}^{\lambda}$ in the left
$\mathcal{A}$-module $F_{i}/F_{i-1}$ when $\mathbf{k}$ is a field of
characteristic $0$. Indeed, we can rewrite Corollary \ref{cor.Fi/Fi-1.lr} as
follows using Definition \ref{def.clI}:

\begin{corollary}
\label{cor.Fi/Fi-1.lr-rewr}Assume that $\mathbf{k}$ is a field of
characteristic $0$.

Let $i\in\left[  f_{n+1}\right]  $. Let $\operatorname*{Par}\nolimits_{n}$ be
the set of all partitions of $n$. Then,%
\[
F_{i}/F_{i-1}\cong\bigoplus_{\lambda\in\operatorname*{Par}\nolimits_{n}%
}\left(  \mathcal{S}^{\lambda}\right)  ^{\oplus c_{Q_{i}}^{\lambda}}.
\]

\end{corollary}

We shall now state our first main theorem:

\begin{theorem}
\label{thm.eigs}Let $\mathbf{k}$ be any field. Let $\lambda$ be a partition of
$n$. Let $\omega_{1},\omega_{2},\ldots,\omega_{n}\in\mathbf{k}$. For each
subset $I$ of $\left[  n\right]  $, we set%
\[
\omega_{I}:=\omega_{1}m_{I,1}+\omega_{2}m_{I,2}+\cdots+\omega_{n}m_{I,n}%
\in\mathbf{k}.
\]
Then:

\begin{enumerate}
\item[\textbf{(a)}] The eigenvalues of the operator $L_{\lambda}\left(
\omega_{1}t_{1}+\omega_{2}t_{2}+\cdots+\omega_{n}t_{n}\right)  $ on the Specht
module $\mathcal{S}^{\lambda}$ are the elements%
\begin{equation}
\omega_{I}\text{ for all lacunar subsets }I\subseteq\left[  n-1\right]  \text{
satisfying }c_{I}^{\lambda}\neq0,\nonumber
\end{equation}
and their respective algebraic multiplicities are the $c_{I}^{\lambda}$ in the
generic case (i.e., if no two $I$'s produce the same $\omega_{I}$; otherwise
the multiplicities of colliding eigenvalues should be added together).

\item[\textbf{(b)}] If all these $\omega_{I}$ (for all lacunar subsets
$I\subseteq\left[  n-1\right]  $ satisfying $c_{I}^{\lambda}\neq0$) are
distinct, then $L_{\lambda}\left(  \omega_{1}t_{1}+\omega_{2}t_{2}%
+\cdots+\omega_{n}t_{n}\right)  $ is diagonalizable.

\item[\textbf{(c)}] We have%
\[
\prod_{\substack{I\subseteq\left[  n-1\right]  \text{ is lacunar;}%
\\c_{I}^{\lambda}\neq0}}\left(  L_{\lambda}\left(  \omega_{1}t_{1}+\omega
_{2}t_{2}+\cdots+\omega_{n}t_{n}\right)  -\omega_{I}\operatorname*{id}%
\nolimits_{\mathcal{S}^{\lambda}}\right)  =0.
\]

\end{enumerate}
\end{theorem}

To prove this theorem, we will need a further theorem, which \textquotedblleft
maps\textquotedblright\ the Fibonacci filtration from $\mathcal{A}$ to a given
Specht module $\mathcal{S}^{\lambda}$:

\begin{theorem}
\label{thm.specht.filtr}Let $\mathbf{k}$ be any field of characteristic $0$.
Let $\lambda$ be a partition of $n$. Then, there exists a filtration
\[
0=F_{f_{n+1}}^{\lambda}\subseteq F_{f_{n+1}-1}^{\lambda}\subseteq
F_{f_{n+1}-2}^{\lambda}\subseteq\cdots\subseteq F_{2}^{\lambda}\subseteq
F_{1}^{\lambda}\subseteq F_{0}^{\lambda}=\mathcal{S}^{\lambda}%
\]
(note the \textquotedblleft backward\textquotedblright\ indexing!) of the
Specht module $\mathcal{S}^{\lambda}$ by left $\mathcal{T}$-submodules with
the following four properties:

\begin{enumerate}
\item Each subquotient $F_{i-1}^{\lambda}/F_{i}^{\lambda}$ (for $i\in\left[
f_{n+1}\right]  $) has dimension $c_{Q_{i}}^{\lambda}$ as a $\mathbf{k}%
$-vector space (see Definition \ref{def.clI} for the meaning of $c_{I}%
^{\lambda}$).

\item In particular, an $i\in\left[  f_{n+1}\right]  $ satisfies
$F_{i-1}^{\lambda}=F_{i}^{\lambda}$ if and only if $c_{Q_{i}}^{\lambda}=0$.

\item On each subquotient $F_{i-1}^{\lambda}/F_{i}^{\lambda}$ (for
$i\in\left[  f_{n+1}\right]  $), each element $t_{\ell}\in\mathcal{T}$ (for
$\ell\in\left[  n\right]  $) acts as multiplication by the scalar
$m_{Q_{i},\ell}$.

\item More generally, on each subquotient $F_{i-1}^{\lambda}/F_{i}^{\lambda}$
(for $i\in\left[  f_{n+1}\right]  $), each element $P\left(  t_{1}%
,t_{2},\ldots,t_{n}\right)  \in\mathcal{T}$ (where $P$ is a polynomial in $n$
noncommuting indeterminates over $\mathbf{k}$) acts as multiplication by the
scalar $P\left(  m_{Q_{i},1},m_{Q_{i},2},\ldots,m_{Q_{i},n}\right)  $.
\end{enumerate}
\end{theorem}

\subsection{Lemmas about $S_{n}$-representations}

In order to prove Theorem \ref{thm.eigs} and Theorem \ref{thm.specht.filtr},
we need a few lemmas about left $\mathcal{A}$-modules (i.e., representations
of $S_{n}$). We begin with something basic and well-known (see, e.g.,
\cite[last paragraph of \S 5.13]{Etingof-et-al}):

\begin{proposition}
\label{prop.specht.tensA}Assume that $\mathbf{k}$ is a field of characteristic
$0$. Let $\lambda$ and $\mu$ be two partitions of $n$. Then, the Specht
modules $\mathcal{S}^{\lambda}$ and $\mathcal{S}^{\mu}$ satisfy%
\[
\operatorname*{Hom}\nolimits_{\mathcal{A}}\left(  \mathcal{S}^{\lambda
},\mathcal{S}^{\mu}\right)  \cong%
\begin{cases}
\mathbf{k}, & \text{if }\lambda=\mu;\\
0, & \text{if }\lambda\neq\mu
\end{cases}
\ \ \ \ \ \ \ \ \ \ \text{as }\mathbf{k}\text{-vector spaces.}%
\]
(Here and in the following, \textquotedblleft$\operatorname*{Hom}%
\nolimits_{\mathcal{A}}$\textquotedblright\ always stands for the set of left
$\mathcal{A}$-module morphisms. This is always a $\mathbf{k}$-vector space,
but usually not an $\mathcal{A}$-module on any side.)
\end{proposition}

For the sake of completeness, we give a proof of this proposition in the
Appendix (Subsection \ref{subsec.apx.proofs.specht.tensA}).

\begin{lemma}
\label{lem.specht.exact}Assume that $\mathbf{k}$ is a field of characteristic
$0$. Let $\lambda$ be a partition of $n$. Define the contravariant functor
$\operatorname*{Hom}\nolimits_{\mathcal{A}}\left(  -,\mathcal{S}^{\lambda
}\right)  $ from the category of left $\mathcal{A}$-modules to the category of
$\mathbf{k}$-vector spaces that is given by%
\[
X\mapsto\operatorname*{Hom}\nolimits_{\mathcal{A}}\left(  X,\mathcal{S}%
^{\lambda}\right)  \ \ \ \ \ \ \ \ \ \ \text{on objects}%
\]
and likewise on morphisms. (This is a contravariant Hom functor.)

This contravariant functor $\operatorname*{Hom}\nolimits_{\mathcal{A}}\left(
-,\mathcal{S}^{\lambda}\right)  $ is exact (i.e., respects exact sequences).
\end{lemma}

\begin{proof}
The $\mathbf{k}$-algebra $\mathcal{A}=\mathbf{k}\left[  S_{n}\right]  $ is
semisimple (by Maschke's theorem, since $\mathbf{k}$ is a field of
characteristic $0$). Hence, every short exact sequence of left $\mathcal{A}%
$-modules is split. Consequently, any Hom functor from the category of left
$\mathcal{A}$-modules is exact\footnote{We remind the reader of the proof of
this fact: Let $\mathcal{F}$ be a contravariant Hom functor from the category
of left $\mathcal{A}$-modules (the case of a covariant Hom functor is
analogous). We must show that $\mathcal{F}$ is exact.
\par
Let $0\rightarrow X\rightarrow Y\rightarrow Z\rightarrow0$ be a short exact
sequence of left $\mathcal{A}$-modules. Then, this sequence is split (since
every short exact sequence of left $\mathcal{A}$-modules is split), and thus
is isomorphic to the obvious exact sequence $0\rightarrow X\rightarrow X\oplus
Z\rightarrow Z\rightarrow0$. The Hom functor $\mathcal{F}$ sends the latter
sequence to another short exact sequence $0\rightarrow\mathcal{F}\left(
X\right)  \rightarrow\mathcal{F}\left(  X\right)  \oplus\mathcal{F}\left(
Z\right)  \rightarrow\mathcal{F}\left(  Z\right)  \rightarrow0$ (since Hom
functors respect finite direct sums). But $\mathcal{F}$ is a functor; thus,
$\mathcal{F}$ sends any two isomorphic complexes to two isomorphic complexes.
Hence, because the two short exact sequences $0\rightarrow X\rightarrow
Y\rightarrow Z\rightarrow0$ and $0\rightarrow X\rightarrow X\oplus
Z\rightarrow Z\rightarrow0$ are isomorphic, their images under $\mathcal{F}$
must also be isomorphic. In other words, the image of the short exact sequence
$0\rightarrow X\rightarrow Y\rightarrow Z\rightarrow0$ under $\mathcal{F}$ is
isomorphic to the image of $0\rightarrow X\rightarrow X\oplus Z\rightarrow
Z\rightarrow0$ under $\mathcal{F}$. But the latter image is exact, as we have
shown. Hence, the former image is also exact.
\par
Thus, $\mathcal{F}$ sends the original short exact sequence $0\rightarrow
X\rightarrow Y\rightarrow Z\rightarrow0$ to a short exact sequence. Since this
is true for any short exact sequence $0\rightarrow X\rightarrow Y\rightarrow
Z\rightarrow0$, we thus have shown that the functor $\mathcal{F}$ respects
short exact sequences. In other words, $\mathcal{F}$ is exact, qed.}. Thus,
the contravariant Hom functor $\operatorname*{Hom}\nolimits_{\mathcal{A}%
}\left(  -,\mathcal{S}^{\lambda}\right)  $ is exact. This proves Lemma
\ref{lem.specht.exact}.
\end{proof}

Note that the Specht module $\mathcal{S}^{\lambda}$ in Lemma
\ref{lem.specht.exact} could be replaced by any left $\mathcal{A}$-module, but
we will use $\mathcal{S}^{\lambda}$ only. The same applies to the following lemma:

\begin{lemma}
\label{lem.specht.f1A}Let $\lambda$ be a partition of $n$. Let $J$ be a left
$\mathcal{A}$-submodule of $\mathcal{A}$ (that is, a left ideal of
$\mathcal{A}$). Then, there is a canonical $\mathbf{k}$-vector space
isomorphism%
\begin{align*}
\operatorname*{Hom}\nolimits_{\mathcal{A}}\left(  \mathcal{A}/J,\ \mathcal{S}%
^{\lambda}\right)   &  \rightarrow\left\{  v\in\mathcal{S}^{\lambda}%
\ \mid\ Jv=0\right\}  ,\\
f  &  \mapsto f\left(  \overline{1_{\mathcal{A}}}\right)  .
\end{align*}

\end{lemma}

\begin{proof}
The left $\mathcal{A}$-module morphisms from $\mathcal{A}/J$ to $\mathcal{S}%
^{\lambda}$ can be identified with the left $\mathcal{A}$-module morphisms
from $\mathcal{A}$ to $\mathcal{S}^{\lambda}$ that vanish on $J$. Thus, we
obtain a $\mathbf{k}$-vector space isomorphism%
\begin{align*}
\Phi:\operatorname*{Hom}\nolimits_{\mathcal{A}}\left(  \mathcal{A}%
/J,\ \mathcal{S}^{\lambda}\right)   &  \rightarrow\left\{  g\in
\operatorname*{Hom}\nolimits_{\mathcal{A}}\left(  \mathcal{A},\ \mathcal{S}%
^{\lambda}\right)  \ \mid\ g\left(  J\right)  =0\right\}  ,\\
f  &  \mapsto\left(  \mathcal{A}\rightarrow\mathcal{S}^{\lambda},\ a\mapsto
f\left(  \overline{a}\right)  \right)  .
\end{align*}

However, recall the well-known $\mathbf{k}$-vector space isomorphism
$\operatorname*{Hom}\nolimits_{A}\left(  A,M\right)  \cong M$ that holds for
any $\mathbf{k}$-algebra $A$ and any left $A$-module $M$. Thus, in particular,
the left $\mathcal{A}$-module morphisms from $\mathcal{A}$ to $\mathcal{S}%
^{\lambda}$ can be identified with the elements of $\mathcal{S}^{\lambda}$ via
the $\mathbf{k}$-vector space isomorphism
\begin{align*}
\Psi:\operatorname*{Hom}\nolimits_{\mathcal{A}}\left(  \mathcal{A}%
,\ \mathcal{S}^{\lambda}\right)   &  \rightarrow\mathcal{S}^{\lambda},\\
g  &  \mapsto g\left(  1_{\mathcal{A}}\right)  .
\end{align*}
This latter isomorphism $\Psi$ has the property that an arbitrary
left $\mathcal{A}$-module
morphism $g\in\operatorname*{Hom}\nolimits_{\mathcal{A}}\left(  \mathcal{A}%
,\ \mathcal{S}^{\lambda}\right)  $ satisfies $g\left(  J\right)  =0$ if and
only if its image $\Psi\left(  g\right)  $ satisfies $J\cdot\Psi\left(
g\right)  =0$ (since $g\left(  J\right)  =g\left(  J\cdot1_{\mathcal{A}%
}\right)  =J\cdot\underbrace{g\left(  1_{\mathcal{A}}\right)  }_{=\Psi\left(
g\right)  }=J\cdot\Psi\left(  g\right)  $). Thus, $\Psi$ can be restricted to
a $\mathbf{k}$-vector space isomorphism
\begin{align*}
\Psi^{\prime}:\left\{  g\in\operatorname*{Hom}\nolimits_{\mathcal{A}}\left(
\mathcal{A},\ \mathcal{S}^{\lambda}\right)  \ \mid\ g\left(  J\right)
=0\right\}   &  \rightarrow\left\{  v\in\mathcal{S}^{\lambda}\ \mid
\ Jv=0\right\}  ,\\
g  &  \mapsto g\left(  1_{\mathcal{A}}\right)  .
\end{align*}
The composition $\Psi^{\prime}\circ\Phi$ is thus a $\mathbf{k}$-vector space
isomorphism%
\begin{align*}
\operatorname*{Hom}\nolimits_{\mathcal{A}}\left(  \mathcal{A}/J,\ \mathcal{S}%
^{\lambda}\right)   &  \rightarrow\left\{  v\in\mathcal{S}^{\lambda}%
\ \mid\ Jv=0\right\}  ,\\
f  &  \mapsto f\left(  \overline{1_{\mathcal{A}}}\right)  .
\end{align*}
This is clearly canonical in $J$, so that Lemma \ref{lem.specht.f1A} is proved.
\end{proof}

\subsection{The proofs}

We are now ready to prove Theorem \ref{thm.specht.filtr} and Theorem
\ref{thm.eigs}, in this order.

\begin{proof}
[Proof of Theorem \ref{thm.specht.filtr}.]For each $i\in\left[  0,f_{n+1}%
\right]  $, we define a subset $F_{i}^{\lambda}$ of $\mathcal{S}^{\lambda}$ by%
\[
F_{i}^{\lambda}:=\left\{  v\in\mathcal{S}^{\lambda}\ \mid\ F_{i}v=0\right\}
.
\]
This subset $F_{i}^{\lambda}$ is actually a left $\mathcal{T}$-submodule of
$\mathcal{S}^{\lambda}$ (since Proposition \ref{prop.fibfilt.A} shows that
$F_{i}$ is a right $\mathcal{T}$-submodule of $\mathcal{A}$%
)\ \ \ \ \footnote{\textit{Proof.} To show that $F_{i}^{\lambda}$ is closed
under addition, we observe that any $v,w\in F_{i}^{\lambda}$ satisfy $v+w\in
F_{i}^{\lambda}$ (because $v,w\in F_{i}^{\lambda}$ entails $F_{i}v=0$ and
$F_{i}w=0$ and thus $F_{i}\left(  v+w\right)  \subseteq\underbrace{F_{i}%
v}_{=0}+\underbrace{F_{i}w}_{=0}=0$, so that $F_{i}\left(  v+w\right)  =0$ and
therefore $v+w\in F_{i}^{\lambda}$).
\par
To show that $F_{i}^{\lambda}$ is closed under the left $\mathcal{T}$-action,
we observe that any $t\in\mathcal{T}$ and $v\in F_{i}^{\lambda}$ satisfy
$tv\in F_{i}^{\lambda}$ (because $v\in F_{i}^{\lambda}$ entails $F_{i}v=0$;
but we have $F_{i}t\subseteq F_{i}$ since $F_{i}$ is a right $\mathcal{T}%
$-submodule of $\mathcal{A}$; thus $F_{i}\left(  tv\right)  =\underbrace{F_{i}%
t}_{\subseteq F_{i}}v\subseteq F_{i}v=0$ and therefore $F_{i}\left(
tv\right)  =0$, so that $tv\in F_{i}^{\lambda}$).
\par
An even simpler argument shows that $0\in F_{i}^{\lambda}$.}. Moreover,
because of%
\[
0=F_{0}\subseteq F_{1}\subseteq F_{2}\subseteq\cdots\subseteq F_{f_{n+1}%
}=\mathbf{k}\left[  S_{n}\right]  =\mathcal{A},
\]
we have%
\[
\mathcal{S}^{\lambda}=F_{0}^{\lambda}\supseteq F_{1}^{\lambda}\supseteq
F_{2}^{\lambda}\supseteq\cdots\supseteq F_{f_{n+1}}^{\lambda}=0.
\]
This is a left $\mathcal{T}$-module filtration of $\mathcal{S}^{\lambda}$,
albeit written backwards. We can rewrite it as%
\[
0=F_{f_{n+1}}^{\lambda}\subseteq F_{f_{n+1}-1}^{\lambda}\subseteq
F_{f_{n+1}-2}^{\lambda}\subseteq\cdots\subseteq F_{2}^{\lambda}\subseteq
F_{1}^{\lambda}\subseteq F_{0}^{\lambda}=\mathcal{S}^{\lambda}.
\]
Our goal is to show that this filtration satisfies the four properties 1, 2, 3
and 4 claimed in Theorem \ref{thm.specht.filtr}.

For this purpose, we fix $i\in\left[  f_{n+1}\right]  $. First, we shall show
property 3. We must show that each element $t_{\ell}\in\mathcal{T}$ acts on
$F_{i-1}^{\lambda}/F_{i}^{\lambda}$ as multiplication by the scalar
$m_{Q_{i},\ell}$. So we let $\ell\in\left[  n\right]  $ and $\overline{v}\in
F_{i-1}^{\lambda}/F_{i}^{\lambda}$ (with $v\in F_{i-1}^{\lambda}$) be
arbitrary. We must show that $t_{\ell}\overline{v}=m_{Q_{i},\ell}\overline{v}$
in $F_{i-1}^{\lambda}/F_{i}^{\lambda}$.

We have $v\in F_{i-1}^{\lambda}$. In other words, $v\in\mathcal{S}^{\lambda}$
and $F_{i-1}v=0$ (by the definition of $F_{i-1}^{\lambda}$).

Theorem \ref{thm.t-simultri} \textbf{(c)} yields $F_{i}\cdot\left(  t_{\ell
}-m_{Q_{i},\ell}\right)  \subseteq F_{i-1}$. Hence,%
\[
\underbrace{F_{i}\cdot\left(  t_{\ell}-m_{Q_{i},\ell}\right)  }_{\subseteq
F_{i-1}}v\subseteq F_{i-1}v=0,
\]
so that $F_{i}\cdot\left(  t_{\ell}-m_{Q_{i},\ell}\right)  v=0$. In other
words, $\left(  t_{\ell}-m_{Q_{i},\ell}\right)  v\in F_{i}^{\lambda}$ (by the
definition of $F_{i}^{\lambda}$). In other words, $t_{\ell}v-m_{Q_{i},\ell
}v\in F_{i}^{\lambda}$. In other words, $\overline{t_{\ell}v}=\overline
{m_{Q_{i},\ell}v}$ in $F_{i-1}^{\lambda}/F_{i}^{\lambda}$. In other words,
$t_{\ell}\overline{v}=m_{Q_{i},\ell}\overline{v}$ in $F_{i-1}^{\lambda}%
/F_{i}^{\lambda}$. Thus, the proof of property 3 is complete.

Property 4 follows immediately from property 3.

Let us now prove property 1.

Note that $F_{i}$ is a left $\mathcal{A}$-submodule of $\mathcal{A}$ (by
Proposition \ref{prop.fibfilt.A}). Lemma \ref{lem.specht.f1A} shows that
whenever $J$ is a left $\mathcal{A}$-submodule of $\mathcal{A}$ (that is, a
left ideal of $\mathcal{A}$), there is a canonical $\mathbf{k}$-vector space
isomorphism%
\begin{align*}
\operatorname*{Hom}\nolimits_{\mathcal{A}}\left(  \mathcal{A}/J,\ \mathcal{S}%
^{\lambda}\right)   &  \rightarrow\left\{  v\in\mathcal{S}^{\lambda}%
\ \mid\ Jv=0\right\}  ,\\
f  &  \mapsto f\left(  \overline{1_{\mathcal{A}}}\right)  .
\end{align*}
Thus, we have%
\begin{equation}
\left\{  v\in\mathcal{S}^{\lambda}\ \mid\ Jv=0\right\}  \cong%
\operatorname*{Hom}\nolimits_{\mathcal{A}}\left(  \mathcal{A}/J,\ \mathcal{S}%
^{\lambda}\right)  \label{pf.thm.specht.filtr.caniso}%
\end{equation}
canonically for each left $\mathcal{A}$-submodule $J$ of $\mathcal{A}$. Now,
the definition of $F_{i}^{\lambda}$ yields%
\begin{equation}
F_{i}^{\lambda}=\left\{  v\in\mathcal{S}^{\lambda}\ \mid\ F_{i}v=0\right\}
\cong\operatorname*{Hom}\nolimits_{\mathcal{A}}\left(  \mathcal{A}%
/F_{i},\ \mathcal{S}^{\lambda}\right)  \label{pf.thm.specht.filtr.canisoi}%
\end{equation}
canonically (by (\ref{pf.thm.specht.filtr.caniso})) and similarly%
\begin{equation}
F_{i-1}^{\lambda}\cong\operatorname*{Hom}\nolimits_{\mathcal{A}}\left(
\mathcal{A}/F_{i-1},\ \mathcal{S}^{\lambda}\right)  .
\label{pf.thm.specht.filtr.canisoi-1}%
\end{equation}

However, Lemma \ref{lem.specht.exact} shows that the contravariant functor
$\operatorname*{Hom}\nolimits_{\mathcal{A}}\left(  -,\mathcal{S}^{\lambda
}\right)  $ from the category of left $\mathcal{A}$-modules to the category of
$\mathbf{k}$-vector spaces is exact. Hence, applying this contravariant
functor to the exact sequence%
\[
0\rightarrow F_{i}/F_{i-1}\rightarrow\mathcal{A}/F_{i-1}\rightarrow
\mathcal{A}/F_{i}\rightarrow0
\]
of left $\mathcal{A}$-modules, we obtain an exact sequence%
\[
0\rightarrow\operatorname*{Hom}\nolimits_{\mathcal{A}}\left(  \mathcal{A}%
/F_{i},\ \mathcal{S}^{\lambda}\right)  \rightarrow\operatorname*{Hom}%
\nolimits_{\mathcal{A}}\left(  \mathcal{A}/F_{i-1},\ \mathcal{S}^{\lambda
}\right)  \rightarrow\operatorname*{Hom}\nolimits_{\mathcal{A}}\left(
F_{i}/F_{i-1},\ \mathcal{S}^{\lambda}\right)  \rightarrow0
\]
of $\mathbf{k}$-vector spaces. In view of (\ref{pf.thm.specht.filtr.canisoi-1}%
) and (\ref{pf.thm.specht.filtr.canisoi}), we can rewrite this latter exact
sequence as%
\[
0\rightarrow F_{i}^{\lambda}\rightarrow F_{i-1}^{\lambda}\rightarrow
\operatorname*{Hom}\nolimits_{\mathcal{A}}\left(  F_{i}/F_{i-1},\ \mathcal{S}%
^{\lambda}\right)  \rightarrow0.
\]
The arrow $F_{i}^{\lambda}\rightarrow F_{i-1}^{\lambda}$ here is the canonical
inclusion (since the isomorphisms in (\ref{pf.thm.specht.filtr.canisoi-1}) and
(\ref{pf.thm.specht.filtr.canisoi}) are the canonical ones), and thus we
obtain%
\begin{equation}
\operatorname*{Hom}\nolimits_{\mathcal{A}}\left(  F_{i}/F_{i-1},\ \mathcal{S}%
^{\lambda}\right)  \cong F_{i-1}^{\lambda}/F_{i}^{\lambda}
\label{pf.thm.specht.filtr.iso5}%
\end{equation}
from the exactness of our sequence.

However, Corollary \ref{cor.Fi/Fi-1.lr-rewr} says that%
\begin{equation}
F_{i}/F_{i-1}\cong\bigoplus_{\nu\in\operatorname*{Par}\nolimits_{n}}\left(
\mathcal{S}^{\nu}\right)  ^{\oplus c_{Q_{i}}^{\nu}}.
\label{pf.thm.specht.filtr.iso6}%
\end{equation}
Hence,%
\begin{align}
\operatorname*{Hom}\nolimits_{\mathcal{A}}\left(  F_{i}/F_{i-1},\ \mathcal{S}%
^{\lambda}\right)   &  \cong\operatorname*{Hom}\nolimits_{\mathcal{A}}\left(
\bigoplus_{\nu\in\operatorname*{Par}\nolimits_{n}}\left(  \mathcal{S}^{\nu
}\right)  ^{\oplus c_{Q_{i}}^{\nu}},\ \mathcal{S}^{\lambda}\right) \nonumber\\
&  \cong\bigoplus_{\nu\in\operatorname*{Par}\nolimits_{n}}\left(
\operatorname*{Hom}\nolimits_{\mathcal{A}}\left(  \mathcal{S}^{\nu
},\mathcal{S}^{\lambda}\right)  \right)  ^{\oplus c_{Q_{i}}^{\nu}}
\label{pf.thm.specht.filtr.iso6Hom}%
\end{align}
(since Hom functors respect finite direct sums).

But each $\nu\in\operatorname*{Par}\nolimits_{n}$ satisfies%
\[
\operatorname*{Hom}\nolimits_{\mathcal{A}}\left(  \mathcal{S}^{\nu
},\mathcal{S}^{\lambda}\right)  \cong%
\begin{cases}
\mathbf{k}, & \text{if }\nu=\lambda;\\
0, & \text{if }\nu\neq\lambda
\end{cases}
\ \ \ \ \ \ \ \ \ \ \left(  \text{by Proposition \ref{prop.specht.tensA}%
}\right)
\]
and thus
\begin{equation}
\left(  \operatorname*{Hom}\nolimits_{\mathcal{A}}\left(  \mathcal{S}^{\nu
},\mathcal{S}^{\lambda}\right)  \right)  ^{\oplus c_{Q_{i}}^{\nu}}\cong\left(
\begin{cases}
\mathbf{k}, & \text{if }\nu=\lambda;\\
0, & \text{if }\nu\neq\lambda
\end{cases}
\right)  ^{\oplus c_{Q_{i}}^{\nu}}\cong%
\begin{cases}
\mathbf{k}^{\oplus c_{Q_{i}}^{\nu}}, & \text{if }\nu=\lambda;\\
0, & \text{if }\nu\neq\lambda.
\end{cases}
\nonumber
\end{equation}
Thus, we can rewrite (\ref{pf.thm.specht.filtr.iso6Hom}) as%
\[
\operatorname*{Hom}\nolimits_{\mathcal{A}}\left(  F_{i}/F_{i-1},\ \mathcal{S}%
^{\lambda}\right)  \cong\bigoplus_{\nu\in\operatorname*{Par}\nolimits_{n}}%
\begin{cases}
\mathbf{k}^{\oplus c_{Q_{i}}^{\nu}}, & \text{if }\nu=\lambda;\\
0, & \text{if }\nu\neq\lambda
\end{cases}
\ \ \ =\mathbf{k}^{\oplus c_{Q_{i}}^{\lambda}}.
\]
Comparing this with (\ref{pf.thm.specht.filtr.iso5}), we see that%
\[
F_{i-1}^{\lambda}/F_{i}^{\lambda}\cong\mathbf{k}^{\oplus c_{Q_{i}}^{\lambda}%
}.
\]

Thus, the $\mathbf{k}$-vector space $F_{i-1}^{\lambda}/F_{i}^{\lambda}$ has
dimension $c_{Q_{i}}^{\lambda}$. This proves property 1.

Property 2 follows immediately from property 1 (since $F_{i-1}^{\lambda}%
=F_{i}^{\lambda}$ is equivalent to $\dim\left(  F_{i-1}^{\lambda}%
/F_{i}^{\lambda}\right)  =0$). Hence, our proof of Theorem
\ref{thm.specht.filtr} is complete.
\end{proof}

\begin{proof}
[Proof of Theorem \ref{thm.eigs}.]In the following, the symbol $\dim$ will
always refer to the dimension of a $\mathbf{k}$-vector space, even if some
other module structures are present. Thus, in particular, if $X$ is a
$\mathcal{T}$-module, then $\dim X$ will mean the dimension of $X$ as a
$\mathbf{k}$-vector space. \medskip

\textbf{(a)} Let us first assume that $\mathbf{k}$ is a field of
characteristic $0$. We shall later extend this to the general case.

Theorem \ref{thm.specht.filtr} shows that there exists a filtration
\begin{equation}
0=F_{f_{n+1}}^{\lambda}\subseteq F_{f_{n+1}-1}^{\lambda}\subseteq
F_{f_{n+1}-2}^{\lambda}\subseteq\cdots\subseteq F_{2}^{\lambda}\subseteq
F_{1}^{\lambda}\subseteq F_{0}^{\lambda}=\mathcal{S}^{\lambda}
\label{pf.thm.eigs.filt}%
\end{equation}
of the Specht module $\mathcal{S}^{\lambda}$ by left $\mathcal{T}$-submodules
with the four properties 1, 2, 3 and 4 stated in Theorem
\ref{thm.specht.filtr}. Consider this filtration. Fix any basis $\left(
v_{1},v_{2},\ldots,v_{s}\right)  $ of $\mathcal{S}^{\lambda}$ that conforms
with this filtration (i.e., a basis that begins with a basis of $F_{f_{n+1}%
-1}^{\lambda}$, then extends it to a basis of $F_{f_{n+1}-2}^{\lambda}$, then
extends it to a basis of $F_{f_{n+1}-3}^{\lambda}$, and so on), so that each
$F_{i}^{\lambda}$ is spanned by $v_{1},v_{2},\ldots,v_{j\left(  i\right)  }$
for some $j\left(  i\right)  \in\left[  0,s\right]  $. Note that the
inclusions in (\ref{pf.thm.eigs.filt}) yield%
\[
0=j\left(  f_{n+1}\right)  \leq j\left(  f_{n+1}-1\right)  \leq j\left(
f_{n+1}-2\right)  \leq\cdots\leq j\left(  2\right)  \leq j\left(  1\right)
\leq j\left(  0\right)  =s.
\]
Note that each $i\in\left[  f_{n+1}\right]  $ satisfies $j\left(  i\right)
=\dim\left(  F_{i}^{\lambda}\right)  $ and $j\left(  i-1\right)  =\dim\left(
F_{i-1}^{\lambda}\right)  $ and thus
\begin{align}
j\left(  i-1\right)  -j\left(  i\right)   &  =\dim\left(  F_{i-1}^{\lambda
}\right)  -\dim\left(  F_{i}^{\lambda}\right)  =\dim\left(  F_{i-1}^{\lambda
}/F_{i}^{\lambda}\right) \nonumber\\
&  =c_{Q_{i}}^{\lambda} \label{pf.thm.eigs.j-j}%
\end{align}
(by property 1 of our filtration).

The operator $L_{\lambda}\left(  \omega_{1}t_{1}+\omega_{2}t_{2}+\cdots
+\omega_{n}t_{n}\right)  $ preserves the filtration (\ref{pf.thm.eigs.filt})
(since this filtration is a filtration by left $\mathcal{T}$-submodules, but
$\omega_{1}t_{1}+\omega_{2}t_{2}+\cdots+\omega_{n}t_{n}\in\mathcal{T}$).
Hence, the matrix $M$ that represents this operator with respect to the basis
$\left(  v_{1},v_{2},\ldots,v_{s}\right)  $ is block-upper-triangular with
blocks of sizes
\[
j\left(  i-1\right)  -j\left(  i\right)  \ \ \ \ \ \ \ \ \ \ \text{for all
}i\in\left[  f_{n+1}\right]
\]
(because, e.g., the fact that the operator preserves $F_{i}^{\lambda}$ means
that the first $j\left(  i\right)  $ columns of the matrix $M$ have zeroes
everywhere below the $j\left(  i\right)  $-th row). Moreover, property 4 of
our filtration shows that the element $\omega_{1}t_{1}+\omega_{2}t_{2}%
+\cdots+\omega_{n}t_{n}\in\mathcal{T}$ acts as multiplication by the scalar
\[
\omega_{1}m_{Q_{i},1}+\omega_{2}m_{Q_{i},2}+\cdots+\omega_{n}m_{Q_{i}%
,n}=\omega_{Q_{i}}\ \ \ \ \ \ \ \ \ \ \left(  \text{by the definition of
}\omega_{Q_{i}}\right)
\]
on each subquotient $F_{i-1}^{\lambda}/F_{i}^{\lambda}$. In other words, for
each $v\in F_{i-1}^{\lambda}$, the vector $\overline{v}\in F_{i-1}^{\lambda
}/F_{i}^{\lambda}$ satisfies%
\begin{equation}
\left(  \omega_{1}t_{1}+\omega_{2}t_{2}+\cdots+\omega_{n}t_{n}\right)
\cdot\overline{v}=\omega_{Q_{i}}\overline{v}, \label{pf.thm.eigs.3}%
\end{equation}
and therefore%
\begin{align*}
&  \left(  L_{\lambda}\left(  \omega_{1}t_{1}+\omega_{2}t_{2}+\cdots
+\omega_{n}t_{n}\right)  \right)  \left(  v\right) \\
&  =\left(  \omega_{1}t_{1}+\omega_{2}t_{2}+\cdots+\omega_{n}t_{n}\right)
\cdot v\\
&  =\omega_{Q_{i}}v+\left(  \text{some element of }F_{i}^{\lambda}\right)
\ \ \ \ \ \ \ \ \ \ \left(  \text{by (\ref{pf.thm.eigs.3})}\right) \\
&  =\omega_{Q_{i}}v+\left(  \text{some linear combination of }v_{1}%
,v_{2},\ldots,v_{j\left(  i\right)  }\right)
\end{align*}
(since $F_{i}^{\lambda}$ is spanned by $v_{1},v_{2},\ldots,v_{j\left(
i\right)  }$). We can apply this in particular to $v=v_{k}$ for each
$k\in\left[  j\left(  i-1\right)  \right]  $ (since $F_{i-1}^{\lambda}$ is
spanned by $v_{1},v_{2},\ldots,v_{j\left(  i-1\right)  }$), and conclude that%
\begin{align*}
&  \left(  L_{\lambda}\left(  \omega_{1}t_{1}+\omega_{2}t_{2}+\cdots
+\omega_{n}t_{n}\right)  \right)  \left(  v_{k}\right) \\
&  =\omega_{Q_{i}}v_{k}+\left(  \text{some linear combination of }v_{1}%
,v_{2},\ldots,v_{j\left(  i\right)  }\right)
\end{align*}
for each $k\in\left[  j\left(  i-1\right)  \right]  $.

Thus, the matrix $M$ that represents the operator $L_{\lambda}\left(
\omega_{1}t_{1}+\omega_{2}t_{2}+\cdots+\omega_{n}t_{n}\right)  $ with respect
to the basis $\left(  v_{1},v_{2},\ldots,v_{s}\right)  $ is not only
block-upper-triangular, but also has the property that its $i$-th diagonal
block (for each $i\in\left[  f_{n+1}\right]  $, counted from the end) is the
scalar matrix $\omega_{Q_{i}}\cdot I_{j\left(  i-1\right)  -j\left(  i\right)
}=\omega_{Q_{i}}\cdot I_{c_{Q_{i}}^{\lambda}}$ (by (\ref{pf.thm.eigs.j-j})).
Consequently, the matrix $M$ is upper-triangular, and its diagonal entries are
the elements $\omega_{Q_{i}}$ for all $i\in\left[  f_{n+1}\right]  $, with
each $\omega_{Q_{i}}$ appearing $c_{Q_{i}}^{\lambda}$ times (this means that
if $c_{Q_{i}}^{\lambda}=0$, then $\omega_{Q_{i}}$ does not appear at all).

Of course, this allows us to read off the eigenvalues of this matrix $M$, and
thus of the operator $L_{\lambda}\left(  \omega_{1}t_{1}+\omega_{2}%
t_{2}+\cdots+\omega_{n}t_{n}\right)  $ (since the eigenvalues of a triangular
matrix are just its diagonal entries). We conclude that the eigenvalues of
$L_{\lambda}\left(  \omega_{1}t_{1}+\omega_{2}t_{2}+\cdots+\omega_{n}%
t_{n}\right)  $ are the elements $\omega_{Q_{i}}$ for all $i\in\left[
f_{n+1}\right]  $, with each $\omega_{Q_{i}}$ appearing with algebraic
multiplicity $c_{Q_{i}}^{\lambda}$. Since $Q_{1},Q_{2},\ldots,Q_{f_{n+1}}$ are
just the lacunar subsets of $\left[  n-1\right]  $, we can rewrite this as
follows: The eigenvalues of $L_{\lambda}\left(  \omega_{1}t_{1}+\omega
_{2}t_{2}+\cdots+\omega_{n}t_{n}\right)  $ are the elements $\omega_{I}$ for
all lacunar subsets $I\subseteq\left[  n-1\right]  $, with each $\omega_{I}$
appearing with algebraic multiplicity $c_{I}^{\lambda}$. We can restrict this
list to those lacunar subsets $I\subseteq\left[  n-1\right]  $ that satisfy
$c_{I}^{\lambda}\neq0$ (since an eigenvalue $\omega_{I}$ that appears with
algebraic multiplicity $c_{I}^{\lambda}=0$ simply does not appear at all).

This proves Theorem \ref{thm.eigs} \textbf{(a)} in the case when $\mathbf{k}$
is a field of characteristic $0$. It remains to extend the proof to the case
when $\mathbf{k}$ is an arbitrary field. But there is a standard trick for
this: We recast our result as a polynomial identity. Namely, Theorem
\ref{thm.eigs} \textbf{(a)} is saying that
\[
\underbrace{\det\left(  x\operatorname*{id}\nolimits_{\mathcal{S}^{\lambda}%
}-\,L_{\lambda}\left(  \omega_{1}t_{1}+\omega_{2}t_{2}+\cdots+\omega_{n}%
t_{n}\right)  \right)  }_{\substack{\text{This is the characteristic
polynomial of the}\\\text{endomorphism }L_{\lambda}\left(  \omega_{1}%
t_{1}+\omega_{2}t_{2}+\cdots+\omega_{n}t_{n}\right)  \text{ of }%
\mathcal{S}^{\lambda}}}=\prod_{I\subseteq\left[  n-1\right]  \text{ lacunar}%
}\left(  x-\omega_{I}\right)  ^{c_{I}^{\lambda}}%
\]
in the polynomial ring $\mathbf{k}\left[  x\right]  $. This is a polynomial
identity in the indeterminates $x,\omega_{1},\omega_{2},\ldots,\omega_{n}$
(since the Specht module $\mathcal{S}^{\lambda}$ has a basis consisting of the
standard polytabloids, and the action of $S_{n}$ on this basis is independent
of the base field $\mathbf{k}$). Thus, knowing that this identity holds
whenever $\mathbf{k}$ is a field of characteristic $0$, we can immediately
conclude that it holds for all fields $\mathbf{k}$ (and even all commutative
rings $\mathbf{k}$). This proves Theorem \ref{thm.eigs} \textbf{(a)} in the
general case. \medskip

\textbf{(c)} Again, let us first assume that $\mathbf{k}$ is a field of
characteristic $0$. Recall the filtration (\ref{pf.thm.eigs.filt}) constructed
in the proof of part \textbf{(a)}. As we saw in that proof, the element
$\omega_{1}t_{1}+\omega_{2}t_{2}+\cdots+\omega_{n}t_{n}\in\mathcal{T}$ acts as
multiplication by the scalar $\omega_{Q_{i}}$ on each subquotient
$F_{i-1}^{\lambda}/F_{i}^{\lambda}$ of that filtration. In other words, for
each $i\in\left[  f_{n+1}\right]  $, we have%
\[
\left(  \omega_{1}t_{1}+\omega_{2}t_{2}+\cdots+\omega_{n}t_{n}\right)
\overline{v}=\omega_{Q_{i}}\overline{v}\ \ \ \ \ \ \ \ \ \ \text{for each
}\overline{v}\in F_{i-1}^{\lambda}/F_{i}^{\lambda},
\]
that is,%
\[
\left(  \omega_{1}t_{1}+\omega_{2}t_{2}+\cdots+\omega_{n}t_{n}\right)
v-\omega_{Q_{i}}v\in F_{i}^{\lambda}\ \ \ \ \ \ \ \ \ \ \text{for each }v\in
F_{i-1}^{\lambda}.
\]
In other words, for each $i\in\left[  f_{n+1}\right]  $, we have%
\[
\left(  L_{\lambda}\left(  \omega_{1}t_{1}+\omega_{2}t_{2}+\cdots+\omega
_{n}t_{n}\right)  -\omega_{Q_{i}}\operatorname*{id}\nolimits_{\mathcal{S}%
^{\lambda}}\right)  F_{i-1}^{\lambda}\subseteq F_{i}^{\lambda}.
\]
Hence, the operator%
\[
\prod_{i\in\left[  f_{n+1}\right]  }\left(  L_{\lambda}\left(  \omega_{1}%
t_{1}+\omega_{2}t_{2}+\cdots+\omega_{n}t_{n}\right)  -\omega_{Q_{i}%
}\operatorname*{id}\nolimits_{\mathcal{S}^{\lambda}}\right)  \in
\operatorname*{End}\nolimits_{\mathbf{k}}\left(  \mathcal{S}^{\lambda}\right)
\]
\footnote{This product is well-defined (and does not depend on the order of
its factors), since all its factors (being polynomials in $L_{\lambda}\left(
\omega_{1}t_{1}+\omega_{2}t_{2}+\cdots+\omega_{n}t_{n}\right)  $) commute.}
sends the whole $\mathcal{S}^{\lambda}$ to $0$ (because its first factor sends
$\mathcal{S}^{\lambda}=F_{0}^{\lambda}$ down to $F_{1}^{\lambda}$, then its
second factor sends $F_{1}^{\lambda}$ further down to $F_{2}^{\lambda}$, then
its third factor sends $F_{2}^{\lambda}$ onward to $F_{3}^{\lambda}$, and so
on, until the last factor sends $F_{f_{n+1}-1}^{\lambda}$ down to $F_{f_{n+1}%
}^{\lambda}=0$). Moreover, for this to hold, we do not actually need all the
$f_{n+1}$ factors of this product, but rather only those factors that
correspond to the numbers $i\in\left[  f_{n+1}\right]  $ satisfying $c_{Q_{i}%
}^{\lambda}\neq0$ (because if $c_{Q_{i}}^{\lambda}=0$, then property 2 of our
filtration shows that $F_{i-1}^{\lambda}=F_{i}^{\lambda}$, and thus we don't
need to apply the $L_{\lambda}\left(  \omega_{1}t_{1}+\omega_{2}t_{2}%
+\cdots+\omega_{n}t_{n}\right)  -\omega_{Q_{i}}\operatorname*{id}%
\nolimits_{\mathcal{S}^{\lambda}}$ factor to send us from $F_{i-1}^{\lambda}$
down into $F_{i}^{\lambda}$). Hence, the operator
\[
\prod_{\substack{i\in\left[  f_{n+1}\right]  ;\\c_{Q_{i}}^{\lambda}\neq
0}}\left(  L_{\lambda}\left(  \omega_{1}t_{1}+\omega_{2}t_{2}+\cdots
+\omega_{n}t_{n}\right)  -\omega_{Q_{i}}\operatorname*{id}%
\nolimits_{\mathcal{S}^{\lambda}}\right)  \in\operatorname*{End}%
\nolimits_{\mathbf{k}}\left(  \mathcal{S}^{\lambda}\right)
\]
sends the whole $\mathcal{S}^{\lambda}$ to $0$ as well. In other words,
\[
\prod_{\substack{i\in\left[  f_{n+1}\right]  ;\\c_{Q_{i}}^{\lambda}\neq
0}}\left(  L_{\lambda}\left(  \omega_{1}t_{1}+\omega_{2}t_{2}+\cdots
+\omega_{n}t_{n}\right)  -\omega_{Q_{i}}\operatorname*{id}%
\nolimits_{\mathcal{S}^{\lambda}}\right)  =0.
\]
Since $Q_{1},Q_{2},\ldots,Q_{f_{n+1}}$ are just the lacunar subsets of
$\left[  n-1\right]  $, we can rewrite this as%
\[
\prod_{\substack{I\subseteq\left[  n-1\right]  \text{ is lacunar;}%
\\c_{I}^{\lambda}\neq0}}\left(  L_{\lambda}\left(  \omega_{1}t_{1}+\omega
_{2}t_{2}+\cdots+\omega_{n}t_{n}\right)  -\omega_{I}\operatorname*{id}%
\nolimits_{\mathcal{S}^{\lambda}}\right)  =0.
\]
This proves Theorem \ref{thm.eigs} \textbf{(c)} in the case when $\mathbf{k}$
is a field of characteristic $0$. Just as in our proof of part \textbf{(a)},
we can derive the general case from this case by a polynomial identity
argument (treating $\omega_{1},\omega_{2},\ldots,\omega_{n}$ as
indeterminates, and now considering polynomials with values in
$\operatorname*{End}\nolimits_{\mathbf{k}}\left(  \mathcal{S}^{\lambda
}\right)  $, which can be encoded as tuples of usual polynomials). \medskip

\textbf{(b)} Theorem \ref{thm.eigs} \textbf{(c)} shows that the endomorphism
$L_{\lambda}\left(  \omega_{1}t_{1}+\omega_{2}t_{2}+\cdots+\omega_{n}%
t_{n}\right)  \in\operatorname*{End}\nolimits_{\mathbf{k}}\left(
\mathcal{S}^{\lambda}\right)  $ is annihilated by the polynomial
$\prod_{\substack{I\subseteq\left[  n-1\right]  \text{ is lacunar;}%
\\c_{I}^{\lambda}\neq0}}\left(  x-\omega_{I}\right)  \in\mathbf{k}\left[
x\right]  $ (meaning that the polynomial vanishes when we substitute the
endomorphism for $x$). But it is well-known that a linear endomorphism (of a
finite-dimensional $\mathbf{k}$-vector space) that is annihilated by a
polynomial of the form $\prod\left(  x-r\right)  $ with pairwise distinct
scalars $r$ is always diagonalizable. Hence, if the $\omega_{I}$ in the above
polynomial are pairwise distinct, then the endomorphism $L_{\lambda}\left(
\omega_{1}t_{1}+\omega_{2}t_{2}+\cdots+\omega_{n}t_{n}\right)  $ is
diagonalizable. This proves Theorem \ref{thm.eigs} \textbf{(b)}.
\end{proof}

\section{Final remarks}

Thus we have computed the eigenvalues -- and their algebraic multiplicities --
for the action of any one-sided cycle shuffle $\omega_{1}t_{1}+\omega_{2}%
t_{2}+\cdots+\omega_{n}t_{n}$ on any Specht module $\mathcal{S}^{\lambda}$.
With a trivial amount of work, we could extend this analysis to the action of
any element of $\mathcal{T}$ (that is, of any noncommutative polynomial in
$t_{1},t_{2},\ldots,t_{n}$). This automatically allows us to identify the
eigenvalues of such elements on any $S_{n}$-representation $V$, as long as the
decomposition of $V$ into Specht modules is known.

The proof of our result was achieved in a rather roundabout way: We did no
work in the Specht modules $\mathcal{S}^{\lambda}$ themselves. Instead, we
used a filtration of $\mathcal{A}$ (the Fibonacci filtration) whose
subquotients $F_{i}/F_{i-1}$ we were able to decompose into Specht modules
(Theorem \ref{thm.Fi/Fi-1.as-ind}). Then, we \textquotedblleft
projected\textquotedblright\ this filtration onto each Specht module
$\mathcal{S}^{\lambda}$ (Theorem \ref{thm.specht.filtr}) and used the
semisimplicity of $\mathcal{A}$ (actually, the complete reducibility of
$\mathcal{S}^{\lambda}$ would have sufficed) to triangularize the action of
$\mathcal{T}$ on $\mathcal{S}^{\lambda}$. This is in contrast to other
instances of similar questions, such as the recent \cite{2407.08644}, where
the solution requires significant exploration of the inner life of
$\mathcal{S}^{\lambda}$.

Our above method for proving Theorems \ref{thm.specht.filtr} and
\ref{thm.eigs} -- in which we used the Fibonacci filtration to triangularize
$L_{\lambda}\left(  \omega_{1}t_{1}+\omega_{2}t_{2}+\cdots+\omega_{n}%
t_{n}\right)  $ -- is partly generalizable:

\begin{proposition}
\label{prop.reflT-gen}Let $A$ be a $\mathbf{k}$-algebra\footnotemark, and let
$T$ be a $\mathbf{k}$-subalgebra of $A$. Let%
\begin{equation}
0=F_{0}\subseteq F_{1}\subseteq F_{2}\subseteq\cdots\subseteq F_{m}=A
\label{eq.prop.reflT-gen.filA}%
\end{equation}
be a filtration of $A$ by $\left(  A,T\right)  $-subbimodules. Let $V$ be any
left $A$-module. If $B$ is any $\left(  A,T\right)  $-subbimodule of $A$, then
we can define a left $T$-submodule
\[
V^{B}:=\left\{  v\in V\ \mid\ Bv=0\right\}
\]
of $V$. Then, we have a filtration%
\begin{equation}
0=V^{F_{m}}\subseteq V^{F_{m-1}}\subseteq V^{F_{m-2}}\subseteq\cdots\subseteq
V^{F_{0}}=V \label{eq.prop.reflT-gen.filV}%
\end{equation}
of $V$ by left $T$-submodules.

\begin{enumerate}
\item[\textbf{(a)}] Its subquotients $V^{F_{i-1}}/V^{F_{i}}$ can be
canonically embedded into $\operatorname*{Hom}\nolimits_{A}\left(
F_{i}/F_{i-1},\ V\right)  $ as left $T$-modules. Thus, if some element $t\in
T$ acts triangularly from the right on the filtration
(\ref{eq.prop.reflT-gen.filA}) (meaning that it acts as a scalar on each
subquotient $F_{i}/F_{i-1}$), then it also acts triangularly from the left on
the filtration (\ref{eq.prop.reflT-gen.filV}).

\item[\textbf{(b)}] If $\mathbf{k}$ is a field and the algebra $A$ is
semisimple, then these embeddings $V^{F_{i-1}}/V^{F_{i}}\rightarrow
\operatorname*{Hom}\nolimits_{A}\left(  F_{i}/F_{i-1},\ V\right)  $ are
isomorphisms. Thus, in this case, knowing the dimensions of the Hom-spaces
$\operatorname*{Hom}\nolimits_{A}\left(  F_{i},\ V\right)  $ allows us to
compute the multiplicities of eigenvalues for a triangular $t\in T$ acting on
$V$.
\end{enumerate}
\end{proposition}

\footnotetext{Recall that $\mathbf{k}$ is an arbitrary commutative ring.}

\begin{proof}
This is implicit in our above proofs of Theorems \ref{thm.specht.filtr} and
\ref{thm.eigs} (where $A$, $T$ and $V$ were taken to be $\mathcal{A}$,
$\mathcal{T}$ and $\mathcal{S}^{\lambda}$, and where the submodules $V^{F_{i}%
}$ were called $F_{i}^{\lambda}$).
\end{proof}

\appendix

\section{\label{sec.apx.proofs}Omitted proofs}

In this appendix, we give proofs to some folklore facts about representations
of symmetric groups.

\subsection{\label{subsec.apx.pf.prop.Z-iso}Proof of Proposition
\ref{prop.Z-iso}}

\begin{proof}
[Proof of Proposition \ref{prop.Z-iso} (sketched).]We will use the notations
from \cite{sga}.

Let $\left(  e_{1},e_{2},\ldots,e_{n}\right)  $ be the standard basis of the
natural representation $\mathcal{N}_{n}=\mathbf{k}^{n}$. The symmetric group
$S_{n}$ acts on it by the rule $\sigma\cdot e_{i}=e_{\sigma\left(  i\right)
}$ for all $\sigma\in S_{n}$ and $i\in\left[  n\right]  $. The
subrepresentation $\mathcal{D}_{n}$ is then spanned by the single vector
$e_{1}+e_{2}+\cdots+e_{n}=\left(  1,1,\ldots,1\right)  $.

Consider the Young diagram $Y\left(  \left(  n-1,1\right)  \right)  $ of the
partition $\left(  n-1,1\right)  $. This Young diagram gives rise to a Young
module $\mathcal{M}^{\left(  n-1,1\right)  }:=\mathcal{M}^{Y\left(  \left(
n-1,1\right)  \right)  }$ and a Specht module $\mathcal{S}^{\left(
n-1,1\right)  }:=\mathcal{S}^{Y\left(  \left(  n-1,1\right)  \right)  }$.

The Young module $\mathcal{M}^{\left(  n-1,1\right)  }$ has a basis formed by
the $n$-tabloids of shape $Y\left(  \left(  n-1,1\right)  \right)  $. We will
use the short notation $\underline{\overline{k}}$ for the $n$-tabloid of shape
$Y\left(  \left(  n-1,1\right)  \right)  $ that has the entry $k$ in cell
$\left(  2,1\right)  $, as in \cite[Example 5.3.16]{sga}. For instance,%
\[
\text{if }n=5\text{, then }\underline{\overline{3}}%
=\ytableausetup{tabloids}\ytableaushort{1245,3}\ytableausetup{notabloids}\ \ .
\]

The Specht module $\mathcal{S}^{\left(  n-1,1\right)  }$ is the submodule of
$\mathcal{M}^{\left(  n-1,1\right)  }$ spanned by the polytabloids
$\mathbf{e}_{T}=\underline{\overline{T\left(  2,1\right)  }}%
-\underline{\overline{T\left(  1,1\right)  }}$, or, equivalently, by the
differences $\underline{\overline{k}}-\underline{\overline{\ell}}$ for
$k\neq\ell$ in $\left[  n\right]  $. (See \cite[Example 5.4.2]{sga}.)

The $S_{n}$-representation $\mathcal{M}^{\left(  n-1,1\right)  }$ is
isomorphic to the natural representation $\mathcal{N}_{n}$ via the isomorphism%
\begin{align*}
\mathcal{M}^{\left(  n-1,1\right)  }  &  \rightarrow\mathcal{N}_{n},\\
\underline{\overline{k}}  &  \mapsto e_{k}.
\end{align*}
Under this isomorphism, the Specht module $\mathcal{S}^{\left(  n-1,1\right)
}$ becomes the span of the differences $e_{k}-e_{\ell}$ for $k\neq\ell$ in
$\left[  n\right]  $. Hence, $\mathcal{S}^{\left(  n-1,1\right)  }$ is
isomorphic to the zero-sum subrepresentation of $\mathcal{N}_{n}$ (called
$R\left(  \mathbf{k}^{n}\right)  $ in \cite[\S 4.2.5]{sga}). \medskip

\begin{noncompile}
\textit{Old proof of Proposition \ref{prop.Z-iso} \textbf{(a)}:} We define the
$\mathbf{k}$-bilinear form%
\begin{align*}
f:\mathcal{N}_{n}\times\mathcal{N}_{n}  &  \rightarrow\mathbf{k},\\
\left(  e_{i},e_{j}\right)   &  \mapsto\delta_{i,j}\ \ \ \ \ \ \ \ \ \ \left(
\text{Kronecker delta}\right)  .
\end{align*}
This form $f$ is $S_{n}$-invariant and nondegenerate in the strongest sense
(i.e., it gives rise to an isomorphism $\mathcal{N}_{n}\rightarrow
\mathcal{N}_{n}^{\ast}$). Restricting it to the subset $\mathcal{S}^{\left(
n-1,1\right)  }\times\mathcal{N}_{n}$, we obtain a $\mathbf{k}$-bilinear form%
\begin{align*}
f^{\prime}:\mathcal{S}^{\left(  n-1,1\right)  }\times\mathcal{N}_{n}  &
\rightarrow\mathbf{k},\\
\left(  \underline{\overline{k}}-\underline{\overline{\ell}},e_{j}\right)   &
\mapsto\delta_{k,j}-\delta_{\ell,j},
\end{align*}
which gives rise to a $\mathbf{k}$-module morphism
\begin{align*}
f_{R}^{\prime}:\mathcal{N}_{n}  &  \rightarrow\left(  \mathcal{S}^{\left(
n-1,1\right)  }\right)  ^{\ast},\\
v  &  \mapsto\left(  u\mapsto f^{\prime}\left(  u,v\right)  \right)  .
\end{align*}
This morphism $f_{R}^{\prime}$ is a morphism of $S_{n}$-representations (since
the form $f$ and thus the form $f^{\prime}$ is $S_{n}$-invariant), i.e., a
left $\mathcal{A}$-module morphism. Moreover, its kernel is%
\begin{align*}
\operatorname*{Ker}\left(  f_{R}^{\prime}\right)   &  =\left\{  v\in
\mathcal{N}_{n}\ \mid\ f^{\prime}\left(  u,v\right)  =0\text{ for all }%
u\in\mathcal{S}^{\left(  n-1,1\right)  }\right\} \\
&  =\left\{  v\in\mathcal{N}_{n}\ \mid\ f^{\prime}\left(  \underline{\overline
{k}}-\underline{\overline{\ell}},v\right)  =0\text{ for all }k\neq\ell\right\}
\\
&  =\left\{  v\in\mathcal{N}_{n}\ \mid\ v_{k}-v_{\ell}=0\text{ for all }%
k\neq\ell\right\} \\
&  \ \ \ \ \ \ \ \ \ \ \ \ \ \ \ \ \ \ \ \ \left(
\begin{array}
[c]{c}%
\text{where }v_{i}\text{ denotes the }i\text{-th entry of}\\
\text{the vector }v\in\mathcal{N}_{n}=\mathbf{k}^{n}%
\end{array}
\right) \\
&  =\mathcal{D}_{n}.
\end{align*}
Hence, it factors through the quotient $\mathcal{N}_{n}/\mathcal{D}%
_{n}=\mathcal{Z}_{n}$. Thus, we obtain a left $\mathcal{A}$-module morphism%
\begin{align*}
f_{R}^{\prime\prime}:\mathcal{Z}_{n}  &  \rightarrow\left(  \mathcal{S}%
^{\left(  n-1,1\right)  }\right)  ^{\ast},\\
\overline{v}  &  \mapsto\left(  u\mapsto f^{\prime}\left(  u,v\right)
\right)  .
\end{align*}
This morphism $f_{R}^{\prime\prime}$ is injective (since it is obtained from
$f_{R}^{\prime}$ by factoring out the whole kernel $\operatorname*{Ker}\left(
f_{R}^{\prime}\right)  =\mathcal{D}_{n}$) and surjective (since it sends each
basis vector $\overline{e_{i}}$ to the linear form on $\mathcal{S}^{\left(
n-1,1\right)  }$ that sends each vector to its $i$-th coordinate; but the
latter forms span $\left(  \mathcal{S}^{\left(  n-1,1\right)  }\right)
^{\ast}$), hence bijective. Thus, it is a left $\mathcal{A}$-module
isomorphism, i.e., an isomorphism of $S_{n}$-representations. This proves
Proposition \ref{prop.Z-iso} \textbf{(a)}. \medskip
\end{noncompile}

\textbf{(a)} Let $\left(  e_{1}^{\ast},e_{2}^{\ast},\ldots,e_{n}^{\ast
}\right)  $ be the dual basis of the standard basis $\left(  e_{1}%
,e_{2},\ldots,e_{n}\right)  $ of $\mathcal{N}_{n}$. Thus, each $i\in\left[
n\right]  $ satisfies $e_{i}^{\ast}\in\mathcal{N}_{n}^{\ast}$, and each
$i,j\in\left[  n\right]  $ satisfy $e_{i}^{\ast}\left(  e_{j}\right)
=\delta_{i,j}$ (Kronecker delta). Furthermore, it is straightforward to see
that $\sigma\left(  e_{i}^{\ast}\right)  =e_{\sigma\left(  i\right)  }^{\ast}$
for each $\sigma\in S_{n}$ and $i\in\left[  n\right]  $.

Now, consider the $\mathbf{k}$-linear map%
\begin{align*}
\Phi:\mathcal{M}^{\left(  n-1,1\right)  }  &  \rightarrow\mathcal{N}_{n}%
^{\ast},\\
\underline{\overline{k}}  &  \mapsto e_{k}^{\ast}\ \ \ \ \ \ \ \ \ \ \left(
\text{for all }k\in\left[  n\right]  \right)  .
\end{align*}
This map $\Phi$ is invertible (since $\left(  \underline{\overline{k}}\right)
_{k\in\left[  n\right]  }$ and $\left(  e_{k}^{\ast}\right)  _{k\in\left[
n\right]  }$ are bases of the $\mathbf{k}$-modules $\mathcal{M}^{\left(
n-1,1\right)  }$ and $\mathcal{N}_{n}^{\ast}$, respectively) and $S_{n}%
$-equivariant (since each $k\in\left[  n\right]  $ and $\sigma\in S_{n}$
satisfy $\sigma\left(  \underline{\overline{k}}\right)  =\underline{\overline
{\sigma\left(  k\right)  }}$ and $\sigma\left(  e_{k}^{\ast}\right)
=e_{\sigma\left(  k\right)  }^{\ast}$). Hence, this map $\Phi$ is an
isomorphism of $S_{n}$-representations.

Now, consider the subrepresentation $\mathcal{S}^{\left(  n-1,1\right)
}=\operatorname*{span}\nolimits_{\mathbf{k}}\left\{  \underline{\overline{k}%
}-\underline{\overline{\ell}}\ \mid\ k\neq\ell\right\}  $ of $\mathcal{M}%
^{\left(  n-1,1\right)  }$. The image of this subrepresentation under $\Phi$
is
\begin{align}
&  \Phi\left(  \mathcal{S}^{\left(  n-1,1\right)  }\right) \nonumber\\
&  =\Phi\left(  \operatorname*{span}\nolimits_{\mathbf{k}}\left\{
\underline{\overline{k}}-\underline{\overline{\ell}}\ \mid\ k\neq\ell\right\}
\right)  \ \ \ \ \ \ \ \ \ \ \left(  \text{since }\mathcal{S}^{\left(
n-1,1\right)  }=\operatorname*{span}\nolimits_{\mathbf{k}}\left\{
\underline{\overline{k}}-\underline{\overline{\ell}}\ \mid\ k\neq\ell\right\}
\right) \nonumber\\
&  =\operatorname*{span}\nolimits_{\mathbf{k}}\left\{  \Phi\left(
\underline{\overline{k}}-\underline{\overline{\ell}}\right)  \ \mid\ k\neq
\ell\right\}  \ \ \ \ \ \ \ \ \ \ \left(  \text{since }\Phi\text{ is
}\mathbf{k}\text{-linear}\right) \nonumber\\
&  =\operatorname*{span}\nolimits_{\mathbf{k}}\left\{  e_{k}^{\ast}-e_{\ell
}^{\ast}\ \mid\ k\neq\ell\right\}  \label{pf.prop.Z-iso.a.2}%
\end{align}
(since every $k,\ell\in\left[  n\right]  $ satisfy $\Phi\left(
\underline{\overline{k}}-\underline{\overline{\ell}}\right)  =\Phi\left(
\underline{\overline{k}}\right)  -\Phi\left(  \underline{\overline{\ell}%
}\right)  =e_{k}^{\ast}-e_{\ell}^{\ast}$ by the definition of $\Phi$).

Now, let us recall some general properties of dual modules. If $V$ is any
$\mathbf{k}$-module, and if $W$ is a $\mathbf{k}$-submodule of $V$, then we
let $W^{\perp}$ denote the $\mathbf{k}$-submodule $\left\{  f\in V^{\ast
}\ \mid\ f\left(  W\right)  =0\right\}  $ of $V^{\ast}$. It is well-known that
there is a canonical isomorphism $W^{\perp}\cong\left(  V/W\right)  ^{\ast}$
in this situation (since any linear map $f\in W^{\perp}$ annihilates $W$ and
thus can be factored through $V/W$ by the first isomorphism theorem).
Moreover, if $V$ is a representation of a group $G$, and if $W$ is a
subrepresentation of $V$, then $W^{\perp}$ is a subrepresentation of $V^{\ast
}$, and the above-mentioned isomorphism $W^{\perp}\cong\left(  V/W\right)
^{\ast}$ is an isomorphism of $G$-representations. Applying this fact to
$G=S_{n}$ and $V=\mathcal{N}_{n}$ and $W=\mathcal{D}_{n}$, we obtain the
isomorphism%
\begin{equation}
\mathcal{D}_{n}^{\perp}\cong\left(  \mathcal{N}_{n}/\mathcal{D}_{n}\right)
^{\ast} \label{pf.prop.Z-iso.a.3}%
\end{equation}
of $S_{n}$-representations.

Now we shall show that $\operatorname*{span}\nolimits_{\mathbf{k}}\left\{
e_{k}^{\ast}-e_{\ell}^{\ast}\ \mid\ k\neq\ell\right\}  =\mathcal{D}_{n}%
^{\perp}$. Indeed,
\begin{align}
\mathcal{D}_{n}^{\perp}  &  =\left\{  f\in\mathcal{N}_{n}^{\ast}%
\ \mid\ f\left(  \mathcal{D}_{n}\right)  =0\right\}
\ \ \ \ \ \ \ \ \ \ \left(  \text{by the definition of }\mathcal{D}_{n}%
^{\perp}\right) \nonumber\\
&  =\left\{  f\in\mathcal{N}_{n}^{\ast}\ \mid\ f\left(  \left(  1,1,\ldots
,1\right)  \right)  =0\right\}  \label{pf.prop.Z-iso.a.4}%
\end{align}
(since the $\mathbf{k}$-module $\mathcal{D}_{n}$ is spanned by the single
vector $\left(  1,1,\ldots,1\right)  $, and thus a $\mathbf{k}$-linear map
$f\in\mathcal{N}_{n}^{\ast}$ satisfies $f\left(  \mathcal{D}_{n}\right)  =0$
if and only if it satisfies $f\left(  \left(  1,1,\ldots,1\right)  \right)
=0$). For any $k\neq\ell$, we have $\left(  e_{k}^{\ast}-e_{\ell}^{\ast
}\right)  \left(  \left(  1,1,\ldots,1\right)  \right)  =1-1=0$ and thus
$e_{k}^{\ast}-e_{\ell}^{\ast}\in\mathcal{D}_{n}^{\perp}$ (by
(\ref{pf.prop.Z-iso.a.4})). Hence, by linearity, it follows that
\begin{equation}
\operatorname*{span}\nolimits_{\mathbf{k}}\left\{  e_{k}^{\ast}-e_{\ell}%
^{\ast}\ \mid\ k\neq\ell\right\}  \subseteq\mathcal{D}_{n}^{\perp}.
\label{pf.prop.Z-iso.a.6}%
\end{equation}

Let us now show the converse inclusion. Indeed, let $f\in\mathcal{D}%
_{n}^{\perp}$. Then, $f\in\mathcal{N}_{n}^{\ast}$ and $f\left(  \left(
1,1,\ldots,1\right)  \right)  =0$ (by (\ref{pf.prop.Z-iso.a.4})). But $\left(
e_{1}^{\ast},e_{2}^{\ast},\ldots,e_{n}^{\ast}\right)  $ is a basis of
$\mathcal{N}_{n}^{\ast}$; hence, $f$ can be written as a $\mathbf{k}$-linear
combination $f=\alpha_{1}e_{1}^{\ast}+\alpha_{2}e_{2}^{\ast}+\cdots+\alpha
_{n}e_{n}^{\ast}$ with $\alpha_{1},\alpha_{2},\ldots,\alpha_{n}\in\mathbf{k}$
(since $f\in\mathcal{N}_{n}^{\ast}$). Consider these $\alpha_{1},\alpha
_{2},\ldots,\alpha_{n}$. From $f=\alpha_{1}e_{1}^{\ast}+\alpha_{2}e_{2}^{\ast
}+\cdots+\alpha_{n}e_{n}^{\ast}$, we obtain
\begin{align*}
f\left(  \left(  1,1,\ldots,1\right)  \right)   &  =\left(  \alpha_{1}%
e_{1}^{\ast}+\alpha_{2}e_{2}^{\ast}+\cdots+\alpha_{n}e_{n}^{\ast}\right)
\left(  \left(  1,1,\ldots,1\right)  \right) \\
&  =\alpha_{1}1+\alpha_{2}1+\cdots+\alpha_{n}1=\alpha_{1}+\alpha_{2}%
+\cdots+\alpha_{n}.
\end{align*}
Thus, $\alpha_{1}+\alpha_{2}+\cdots+\alpha_{n}=f\left(  \left(  1,1,\ldots
,1\right)  \right)  =0$, so that $\alpha_{1}=-\alpha_{2}-\alpha_{3}%
-\cdots-\alpha_{n}$. Now,%
\begin{align*}
f  &  =\underbrace{\alpha_{1}}_{=-\alpha_{2}-\alpha_{3}-\cdots-\alpha_{n}%
}e_{1}^{\ast}+\alpha_{2}e_{2}^{\ast}+\cdots+\alpha_{n}e_{n}^{\ast}\\
&  =\left(  -\alpha_{2}-\alpha_{3}-\cdots-\alpha_{n}\right)  e_{1}^{\ast
}+\alpha_{2}e_{2}^{\ast}+\alpha_{3}e_{3}^{\ast}+\cdots+\alpha_{n}e_{n}^{\ast
}\\
&  =\alpha_{2}\underbrace{\left(  e_{2}^{\ast}-e_{1}^{\ast}\right)  }%
_{\in\left\{  e_{k}^{\ast}-e_{\ell}^{\ast}\ \mid\ k\neq\ell\right\}
}+\,\alpha_{3}\underbrace{\left(  e_{3}^{\ast}-e_{1}^{\ast}\right)  }%
_{\in\left\{  e_{k}^{\ast}-e_{\ell}^{\ast}\ \mid\ k\neq\ell\right\}  }%
+\cdots+\alpha_{n}\underbrace{\left(  e_{n}^{\ast}-e_{1}^{\ast}\right)  }%
_{\in\left\{  e_{k}^{\ast}-e_{\ell}^{\ast}\ \mid\ k\neq\ell\right\}  }\\
&  \in\operatorname*{span}\nolimits_{\mathbf{k}}\left\{  e_{k}^{\ast}-e_{\ell
}^{\ast}\ \mid\ k\neq\ell\right\}  .
\end{align*}

Forget that we fixed $f$. We thus have shown that $f\in\operatorname*{span}%
\nolimits_{\mathbf{k}}\left\{  e_{k}^{\ast}-e_{\ell}^{\ast}\ \mid\ k\neq
\ell\right\}  $ for each $f\in\mathcal{D}_{n}^{\perp}$. In other words,
$\mathcal{D}_{n}^{\perp}\subseteq\operatorname*{span}\nolimits_{\mathbf{k}%
}\left\{  e_{k}^{\ast}-e_{\ell}^{\ast}\ \mid\ k\neq\ell\right\}  $. Combining
this with (\ref{pf.prop.Z-iso.a.6}), we obtain%
\[
\mathcal{D}_{n}^{\perp}=\operatorname*{span}\nolimits_{\mathbf{k}}\left\{
e_{k}^{\ast}-e_{\ell}^{\ast}\ \mid\ k\neq\ell\right\}  .
\]
Comparing this with (\ref{pf.prop.Z-iso.a.2}), we are led to%
\begin{align*}
\Phi\left(  \mathcal{S}^{\left(  n-1,1\right)  }\right)   &  =\mathcal{D}%
_{n}^{\perp}\cong\left(  \mathcal{N}_{n}/\mathcal{D}_{n}\right)  ^{\ast
}\ \ \ \ \ \ \ \ \ \ \left(  \text{by (\ref{pf.prop.Z-iso.a.3})}\right) \\
&  \cong\mathcal{Z}_{n}^{\ast}\ \ \ \ \ \ \ \ \ \ \left(  \text{since
}\mathcal{N}_{n}/\mathcal{D}_{n}=\mathcal{Z}_{n}\right)  ,
\end{align*}
and this is an isomorphism of $S_{n}$-representations (as we saw above). On
the other hand, however, we also have $\Phi\left(  \mathcal{S}^{\left(
n-1,1\right)  }\right)  \cong\mathcal{S}^{\left(  n-1,1\right)  }$ as $S_{n}%
$-representations (since $\Phi$ is an isomorphism of $S_{n}$-representations).
Thus,%
\[
\mathcal{S}^{\left(  n-1,1\right)  }\cong\Phi\left(  \mathcal{S}^{\left(
n-1,1\right)  }\right)  \cong\mathcal{Z}_{n}^{\ast}%
\ \ \ \ \ \ \ \ \ \ \text{as }S_{n}\text{-representations.}%
\]
Taking duals, we thus obtain%
\begin{equation}
\left(  \mathcal{S}^{\left(  n-1,1\right)  }\right)  ^{\ast}\cong\left(
\mathcal{Z}_{n}^{\ast}\right)  ^{\ast}\ \ \ \ \ \ \ \ \ \ \text{as }%
S_{n}\text{-representations.} \label{pf.prop.Z-iso.a.dual}%
\end{equation}

But the $\mathbf{k}$-module $\mathcal{Z}_{n}=\mathcal{N}_{n}/\mathcal{D}%
_{n}=\mathbf{k}^{n}/\operatorname*{span}\nolimits_{\mathbf{k}}\left\{  \left(
1,1,\ldots,1\right)  \right\}  $ has a finite basis (namely, $\left(
\overline{e_{1}},\overline{e_{2}},\ldots,\overline{e_{n-1}}\right)  $, as can
be easily seen from basic linear algebra). Hence, \cite[Proposition 5.19.22
\textbf{(b)}]{sga} shows that $\left(  \mathcal{Z}_{n}^{\ast}\right)  ^{\ast
}\cong\mathcal{Z}_{n}$ as $S_{n}$-representations. In view of this, we can
rewrite (\ref{pf.prop.Z-iso.a.dual}) as follows:%
\[
\left(  \mathcal{S}^{\left(  n-1,1\right)  }\right)  ^{\ast}\cong%
\mathcal{Z}_{n}\ \ \ \ \ \ \ \ \ \ \text{as }S_{n}\text{-representations.}%
\]
This proves Proposition \ref{prop.Z-iso} \textbf{(a)}. \medskip

\textbf{(b)} Assume that $n$ is invertible in $\mathbf{k}$. Then,
\cite[Proposition 4.2.28 \textbf{(a)}]{sga} says that $\mathbf{k}^{n}=R\left(
\mathbf{k}^{n}\right)  \oplus D\left(  \mathbf{k}^{n}\right)  $, where we are
using the notations of \cite{sga}. In our notations, this is saying that
$\mathcal{N}_{n}=\mathcal{S}^{\left(  n-1,1\right)  }\oplus\mathcal{D}_{n}$
(since the submodule $\mathcal{S}^{\left(  n-1,1\right)  }%
=\operatorname*{span}\nolimits_{\mathbf{k}}\left\{  \underline{\overline{k}%
}-\underline{\overline{\ell}}\ \mid\ k\neq\ell\right\}  $ of $\mathcal{M}%
^{\left(  n-1,1\right)  }$ corresponds to the zero-sum subrepresentation
$R\left(  \mathbf{k}^{n}\right)  =\operatorname*{span}\nolimits_{\mathbf{k}%
}\left\{  e_{k}-e_{\ell}\ \mid\ k\neq\ell\right\}  $ of $\mathbf{k}^{n}$).
Hence, $\mathcal{S}^{\left(  n-1,1\right)  }\cong\mathcal{N}_{n}%
/\mathcal{D}_{n}=\mathcal{Z}_{n}$ as $S_{n}$-representations. This proves
Proposition \ref{prop.Z-iso} \textbf{(b)}.

Alternatively, we can prove Proposition \ref{prop.Z-iso} \textbf{(b)}
directly: Let
\[
\underline{\overline{\operatorname*{avg}}}:=\dfrac{1}{n}\left(
\underline{\overline{1}}+\underline{\overline{2}}+\cdots+\underline{\overline
{n}}\right)  =\dfrac{1}{n}\sum_{j=1}^{n}\underline{\overline{j}}\in
\mathcal{M}^{\left(  n-1,1\right)  }.
\]
Thus, $n\cdot\underline{\overline{\operatorname*{avg}}}=\sum_{j=1}%
^{n}\underline{\overline{j}}=\sum_{k=1}^{n}\underline{\overline{k}}$. Hence,%
\begin{align}
&  \left(  \underline{\overline{1}}-\underline{\overline{\operatorname*{avg}}%
}\right)  +\left(  \underline{\overline{2}}-\underline{\overline
{\operatorname*{avg}}}\right)  +\cdots+\left(  \underline{\overline{n}%
}-\underline{\overline{\operatorname*{avg}}}\right) \nonumber\\
&  =\sum_{k=1}^{n}\left(  \underline{\overline{k}}-\underline{\overline
{\operatorname*{avg}}}\right)  =\sum_{k=1}^{n}\underline{\overline{k}}%
-n\cdot\underline{\overline{\operatorname*{avg}}}=0 \label{pf.prop.Z-iso.b.2}%
\end{align}
(since $n\cdot\underline{\overline{\operatorname*{avg}}}=\sum_{k=1}%
^{n}\underline{\overline{k}}$).

It is furthermore easy to see that each $i\in\left[  n\right]  $ satisfies
$\underline{\overline{i}}-\underline{\overline{\operatorname*{avg}}}%
\in\mathcal{S}^{\left(  n-1,1\right)  }$ (this follows from an easy
computation\footnote{\textit{Proof.} Let $i\in\left[  n\right]  $. Then, from
$\underline{\overline{\operatorname*{avg}}}=\dfrac{1}{n}\sum_{j=1}%
^{n}\underline{\overline{j}}$, we obtain%
\[
\underline{\overline{i}}-\underline{\overline{\operatorname*{avg}}%
}=\underline{\overline{i}}-\dfrac{1}{n}\sum_{j=1}^{n}\underline{\overline{j}%
}=\dfrac{1}{n}\sum_{j=1}^{n}\underbrace{\left(  \underline{\overline{i}%
}-\underline{\overline{j}}\right)  }_{\substack{\in\mathcal{S}^{\left(
n-1,1\right)  }\\\text{(since }\mathcal{S}^{\left(  n-1,1\right)
}=\operatorname*{span}\nolimits_{\mathbf{k}}\left\{  \underline{\overline{k}%
}-\underline{\overline{\ell}}\ \mid\ k\neq\ell\right\}  \text{)}}%
}\in\mathcal{S}^{\left(  n-1,1\right)  }%
\]
(since $\mathcal{S}^{\left(  n-1,1\right)  }$ is a $\mathbf{k}$-module),
qed.}). Moreover, the list $\left(  \underline{\overline{1}}%
-\underline{\overline{\operatorname*{avg}}},\ \underline{\overline{2}%
}-\underline{\overline{\operatorname*{avg}}},\ \ldots,\ \underline{\overline
{n-1}}-\underline{\overline{\operatorname*{avg}}}\right)  $ is a basis of the
$\mathbf{k}$-module $\mathcal{S}^{\left(  n-1,1\right)  }$ (this is an easy
exercise in linear algebra\footnote{For the sake of completeness, here is a
\textit{proof:}
\par
\begin{itemize}
\item The $n-1$ vectors $\underline{\overline{1}}-\underline{\overline
{\operatorname*{avg}}},\ \underline{\overline{2}}-\underline{\overline
{\operatorname*{avg}}},\ \ldots,\ \underline{\overline{n-1}}%
-\underline{\overline{\operatorname*{avg}}}$ are $\mathbf{k}$-linearly
independent.
\par
\textit{Proof:} Let $\alpha_{1},\alpha_{2},\ldots,\alpha_{n-1}\in\mathbf{k}$
be scalars satisfying $\sum_{k=1}^{n-1}\alpha_{k}\left(  \underline{\overline
{k}}-\underline{\overline{\operatorname*{avg}}}\right)  =0$. We must show that
all coefficients $\alpha_{1},\alpha_{2},\ldots,\alpha_{n-1}$ are $0$.
\par
Recall that $\mathcal{M}^{\left(  n-1,1\right)  }$ is a free $\mathbf{k}%
$-module with basis $\left(  \underline{\overline{1}},\underline{\overline{2}%
},\ldots,\underline{\overline{n}}\right)  $. Comparing the coefficients of
$\underline{\overline{n}}$ on both sides of the equality $\sum_{k=1}%
^{n-1}\alpha_{k}\left(  \underline{\overline{k}}-\underline{\overline
{\operatorname*{avg}}}\right)  =0$, we obtain $-\sum_{k=1}^{n-1}\alpha
_{k}\cdot\dfrac{1}{n}=0$ (since the coefficient of $\underline{\overline{n}}$
in $\underline{\overline{\operatorname*{avg}}}$ is $\dfrac{1}{n}$, whereas the
coefficient of $\underline{\overline{n}}$ in $\underline{\overline{k}}$ for
any $k\in\left[  n-1\right]  $ is $0$). Multiplying the latter equality by
$-n$, we find $\sum_{k=1}^{n-1}\alpha_{k}=0$. Hence,%
\[
\sum_{k=1}^{n-1}\alpha_{k}\left(  \underline{\overline{k}}%
-\underline{\overline{\operatorname*{avg}}}\right)  =\sum_{k=1}^{n-1}%
\alpha_{k}\underline{\overline{k}}-\underbrace{\sum_{k=1}^{n-1}\alpha_{k}%
}_{=0}\underline{\overline{\operatorname*{avg}}}=\sum_{k=1}^{n-1}\alpha
_{k}\underline{\overline{k}},
\]
so that $\sum_{k=1}^{n-1}\alpha_{k}\underline{\overline{k}}=\sum_{k=1}%
^{n-1}\alpha_{k}\left(  \underline{\overline{k}}-\underline{\overline
{\operatorname*{avg}}}\right)  =0$. Since the basis vectors
$\underline{\overline{1}},\underline{\overline{2}},\ldots,\underline{\overline
{n-1}}$ are $\mathbf{k}$-linearly independent, this entails that all
coefficients $\alpha_{1},\alpha_{2},\ldots,\alpha_{n-1}$ are $0$. This
completes the proof of the linear independence of the vectors
$\underline{\overline{1}}-\underline{\overline{\operatorname*{avg}}%
},\ \underline{\overline{2}}-\underline{\overline{\operatorname*{avg}}%
},\ \ldots,\ \underline{\overline{n-1}}-\underline{\overline
{\operatorname*{avg}}}$.
\par
\item The $n-1$ vectors $\underline{\overline{1}}-\underline{\overline
{\operatorname*{avg}}},\ \underline{\overline{2}}-\underline{\overline
{\operatorname*{avg}}},\ \ldots,\ \underline{\overline{n-1}}%
-\underline{\overline{\operatorname*{avg}}}$ span the $\mathbf{k}$-module
$\mathcal{S}^{\left(  n-1,1\right)  }$.
\par
\textit{Proof:} Consider the two subsets%
\begin{align*}
W  &  :=\left\{  \underline{\overline{i}}-\underline{\overline
{\operatorname*{avg}}}\ \mid\ i\in\left[  n\right]  \right\}  =\left\{
\underline{\overline{1}}-\underline{\overline{\operatorname*{avg}}%
},\ \underline{\overline{2}}-\underline{\overline{\operatorname*{avg}}%
},\ \ldots,\ \underline{\overline{n}}-\underline{\overline{\operatorname*{avg}%
}}\right\}  \ \ \ \ \ \ \ \ \ \ \text{and}\\
W^{\prime}  &  :=\left\{  \underline{\overline{i}}-\underline{\overline
{\operatorname*{avg}}}\ \mid\ i\in\left[  n-1\right]  \right\}  =\left\{
\underline{\overline{1}}-\underline{\overline{\operatorname*{avg}}%
},\ \underline{\overline{2}}-\underline{\overline{\operatorname*{avg}}%
},\ \ldots,\ \underline{\overline{n-1}}-\underline{\overline
{\operatorname*{avg}}}\right\}
\end{align*}
of $\mathcal{M}^{\left(  n-1,1\right)  }$. Then, $W\subseteq\mathcal{S}%
^{\left(  n-1,1\right)  }$ (since each $i\in\left[  n\right]  $ satisfies
$\underline{\overline{i}}-\underline{\overline{\operatorname*{avg}}}%
\in\mathcal{S}^{\left(  n-1,1\right)  }$), so that $\operatorname*{span}%
W\subseteq\mathcal{S}^{\left(  n-1,1\right)  }$ (by linearity). Similarly,
$\operatorname*{span}W^{\prime}\subseteq\mathcal{S}^{\left(  n-1,1\right)  }$.
Also, the set $W$ differs from $W^{\prime}$ only in the extra element
$\underline{\overline{n}}-\underline{\overline{\operatorname*{avg}}}$. Thus,
$W=W^{\prime}\cup\left\{  \underline{\overline{n}}-\underline{\overline
{\operatorname*{avg}}}\right\}  $. Therefore,%
\[
\operatorname*{span}W=\operatorname*{span}\left(  W^{\prime}\cup\left\{
\underline{\overline{n}}-\underline{\overline{\operatorname*{avg}}}\right\}
\right)  =\operatorname*{span}W^{\prime}+\operatorname*{span}\left\{
\underline{\overline{n}}-\underline{\overline{\operatorname*{avg}}}\right\}
.
\]
However, solving the equation (\ref{pf.prop.Z-iso.b.2}) for
$\underline{\overline{n}}-\underline{\overline{\operatorname*{avg}}}$, we find%
\begin{align*}
\underline{\overline{n}}-\underline{\overline{\operatorname*{avg}}}  &
=-\left(  \underline{\overline{1}}-\underline{\overline{\operatorname*{avg}}%
}\right)  -\left(  \underline{\overline{2}}-\underline{\overline
{\operatorname*{avg}}}\right)  -\cdots-\left(  \underline{\overline{n-1}%
}-\underline{\overline{\operatorname*{avg}}}\right) \\
&  \in\operatorname*{span}\underbrace{\left\{  \underline{\overline{1}%
}-\underline{\overline{\operatorname*{avg}}},\ \underline{\overline{2}%
}-\underline{\overline{\operatorname*{avg}}},\ \ldots,\ \underline{\overline
{n-1}}-\underline{\overline{\operatorname*{avg}}}\right\}  }_{=W^{\prime}%
}=\operatorname*{span}W^{\prime}.
\end{align*}
Therefore, $\operatorname*{span}\left\{  \underline{\overline{n}%
}-\underline{\overline{\operatorname*{avg}}}\right\}  \subseteq
\operatorname*{span}W^{\prime}$, so that%
\[
\operatorname*{span}W=\operatorname*{span}W^{\prime}%
+\underbrace{\operatorname*{span}\left\{  \underline{\overline{n}%
}-\underline{\overline{\operatorname*{avg}}}\right\}  }_{\subseteq
\operatorname*{span}W^{\prime}}\subseteq\operatorname*{span}W^{\prime
}+\operatorname*{span}W^{\prime}\subseteq\operatorname*{span}W^{\prime}.
\]
\par
Recall that the $\mathbf{k}$-module $\mathcal{S}^{\left(  n-1,1\right)  }$ is
spanned by the differences $\underline{\overline{k}}-\underline{\overline
{\ell}}$ with $k\neq\ell$. But each of the latter differences
$\underline{\overline{k}}-\underline{\overline{\ell}}$ belongs to
$\operatorname*{span}W$ (indeed, $\underline{\overline{k}}%
-\underline{\overline{\ell}}=\underbrace{\left(  \underline{\overline{k}%
}-\underline{\overline{\operatorname*{avg}}}\right)  }_{\in W}%
-\underbrace{\left(  \underline{\overline{\ell}}-\underline{\overline
{\operatorname*{avg}}}\right)  }_{\in W}\in W-W\subseteq\operatorname*{span}%
W$). Combining the previous two sentences, we conclude (by linearity) that
$\mathcal{S}^{\left(  n-1,1\right)  }\subseteq\operatorname*{span}%
W\subseteq\operatorname*{span}W^{\prime}$, so that $\mathcal{S}^{\left(
n-1,1\right)  }=\operatorname*{span}W^{\prime}$ (since $\operatorname*{span}%
W^{\prime}\subseteq\mathcal{S}^{\left(  n-1,1\right)  }$). In other words, the
$n-1$ vectors $\underline{\overline{1}}-\underline{\overline
{\operatorname*{avg}}},\ \underline{\overline{2}}-\underline{\overline
{\operatorname*{avg}}},\ \ldots,\ \underline{\overline{n-1}}%
-\underline{\overline{\operatorname*{avg}}}$ span the $\mathbf{k}$-module
$\mathcal{S}^{\left(  n-1,1\right)  }$ (since these $n-1$ vectors are the
elements of $W^{\prime}$).
\end{itemize}
\par
So we know that the $n-1$ vectors $\underline{\overline{1}}%
-\underline{\overline{\operatorname*{avg}}},\ \underline{\overline{2}%
}-\underline{\overline{\operatorname*{avg}}},\ \ldots,\ \underline{\overline
{n-1}}-\underline{\overline{\operatorname*{avg}}}$ are $\mathbf{k}$-linearly
independent and span the $\mathbf{k}$-module $\mathcal{S}^{\left(
n-1,1\right)  }$. Hence, they are a basis of $\mathcal{S}^{\left(
n-1,1\right)  }$.}). Now, the $\mathbf{k}$-linear map%
\begin{align*}
\varrho:\mathcal{N}_{n}  &  \rightarrow\mathcal{S}^{\left(  n-1,1\right)  },\\
e_{i}  &  \mapsto\underline{\overline{i}}-\underline{\overline
{\operatorname*{avg}}}\ \ \ \ \ \ \ \ \ \ \left(  \text{for all }i\in\left[
n\right]  \right)
\end{align*}
is well-defined (since each $i\in\left[  n\right]  $ satisfies
$\underline{\overline{i}}-\underline{\overline{\operatorname*{avg}}}%
\in\mathcal{S}^{\left(  n-1,1\right)  }$) and $S_{n}$-equivariant (since we
can easily see that $\sigma\cdot\underline{\overline{\operatorname*{avg}}%
}=\underline{\overline{\operatorname*{avg}}}$ for each $\sigma\in S_{n}$). It
furthermore sends $e_{1}+e_{2}+\cdots+e_{n}$ to $\left(  \underline{\overline
{1}}-\underline{\overline{\operatorname*{avg}}}\right)  +\left(
\underline{\overline{2}}-\underline{\overline{\operatorname*{avg}}}\right)
+\cdots+\left(  \underline{\overline{n}}-\underline{\overline
{\operatorname*{avg}}}\right)  =0$ (by (\ref{pf.prop.Z-iso.b.2})), and thus
vanishes on the submodule $\mathcal{D}_{n}$ (by linearity, since
$\mathcal{D}_{n}$ is spanned by $e_{1}+e_{2}+\cdots+e_{n}$). Hence, it factors
through an $S_{n}$-equivariant $\mathbf{k}$-linear map $\varrho^{\prime
}:\mathcal{N}_{n}/\mathcal{D}_{n}\rightarrow\mathcal{S}^{\left(  n-1,1\right)
}$. This latter map $\varrho^{\prime}$ sends each residue class $\overline
{e_{i}}$ to $\underline{\overline{i}}-\underline{\overline{\operatorname*{avg}%
}}$, and thus is invertible (since it sends the basis $\left(  \overline
{e_{1}},\overline{e_{2}},\ldots,\overline{e_{n-1}}\right)  $ of $\mathcal{N}%
_{n}/\mathcal{D}_{n}$ to the basis $\left(  \underline{\overline{1}%
}-\underline{\overline{\operatorname*{avg}}},\ \underline{\overline{2}%
}-\underline{\overline{\operatorname*{avg}}},\ \ldots,\ \underline{\overline
{n-1}}-\underline{\overline{\operatorname*{avg}}}\right)  $ of $\mathcal{S}%
^{\left(  n-1,1\right)  }$). Hence, $\varrho^{\prime}$ is an isomorphism of
$S_{n}$-representations (since it is $S_{n}$-equivariant). Therefore,
$\mathcal{S}^{\left(  n-1,1\right)  }\cong\mathcal{N}_{n}/\mathcal{D}%
_{n}=\mathcal{Z}_{n}$ as $S_{n}$-representations. Proposition \ref{prop.Z-iso}
\textbf{(b)} is now proved again.
\end{proof}

\subsection{\label{subsec.apx.pf.prop.indprod.ass}Proof of Proposition
\ref{prop.indprod.ass}}

\begin{proof}
[Proof of Proposition \ref{prop.indprod.ass}.]We shall use the following
general facts about induced representations (\cite[Exercise 4.1.2]{GriRei} and
\cite[Exercise 4.1.3]{GriRei}, respectively\footnote{The facts are stated in
\cite{GriRei} only for $\mathbf{k}=\mathbb{C}$, but the proofs work equally
well for any $\mathbf{k}$.}):

\begin{itemize}
\item \textit{Transitivity of induction:} Let $G$ be a group. Let $H$ be a
subgroup of $G$. Let $I$ be a subgroup of $H$. Let $U$ be a representation of
$I$. Then,
\begin{equation}
\operatorname*{Ind}\nolimits_{H}^{G}\operatorname*{Ind}\nolimits_{I}^{H}%
U\cong\operatorname*{Ind}\nolimits_{I}^{G}U. \label{pf.prop.indprod.ass.trind}%
\end{equation}

\item \textit{Monoidality of induction:} Let $G_{1}$ and $G_{2}$ be two
groups. Let $H_{1}$ be a subgroup of $G_{1}$, and let $H_{2}$ be a subgroup of
$G_{2}$. Let $W_{1}$ be a representation of $H_{1}$, and let $W_{2}$ be a
representation of $H_{2}$. Then,%
\begin{equation}
\operatorname*{Ind}\nolimits_{H_{1}\times H_{2}}^{G_{1}\times G_{2}}\left(
W_{1}\otimes W_{2}\right)  \cong\left(  \operatorname*{Ind}\nolimits_{H_{1}%
}^{G_{1}}W_{1}\right)  \otimes\left(  \operatorname*{Ind}\nolimits_{H_{2}%
}^{G_{2}}W_{2}\right)  . \label{pf.prop.indprod.ass.mind}%
\end{equation}

\end{itemize}

Now, set $G_{1}:=S_{n_{1}+n_{2}+\cdots+n_{i}}$ and $H_{1}:=S_{n_{1}}\times
S_{n_{2}}\times\cdots\times S_{n_{i}}$ and $W_{1}:=U_{1}\otimes U_{2}%
\otimes\cdots\otimes U_{i}$ (a representation of $H_{1}$) and $G_{2}%
:=S_{n_{i+1}+n_{i+2}+\cdots+n_{k}}$ and $H_{2}:=S_{n_{i+1}}\times S_{n_{i+2}%
}\times\cdots\times S_{n_{k}}$ and $W_{2}:=U_{i+1}\otimes U_{i+2}\otimes
\cdots\otimes U_{k}$ (a representation of $H_{2}$). Then, by the definition of
induction products, we have%
\[
U_{1}\ast U_{2}\ast\cdots\ast U_{i}=\operatorname*{Ind}\nolimits_{S_{n_{1}%
}\times S_{n_{2}}\times\cdots\times S_{n_{i}}}^{S_{n_{1}+n_{2}+\cdots+n_{i}}%
}\left(  U_{1}\otimes U_{2}\otimes\cdots\otimes U_{i}\right)
=\operatorname*{Ind}\nolimits_{H_{1}}^{G_{1}}W_{1}%
\]
(by the definitions of $G_{1}$ and $H_{1}$ and $W_{1}$) and likewise%
\[
U_{i+1}\ast U_{i+2}\ast\cdots\ast U_{k}=\operatorname*{Ind}\nolimits_{H_{2}%
}^{G_{2}}W_{2}.
\]
Thus,
\begin{align*}
&  \underbrace{\left(  U_{1}\ast U_{2}\ast\cdots\ast U_{i}\right)
}_{=\operatorname*{Ind}\nolimits_{H_{1}}^{G_{1}}W_{1}}\ast\underbrace{\left(
U_{i+1}\ast U_{i+2}\ast\cdots\ast U_{k}\right)  }_{=\operatorname*{Ind}%
\nolimits_{H_{2}}^{G_{2}}W_{2}}\\
&  =\left(  \operatorname*{Ind}\nolimits_{H_{1}}^{G_{1}}W_{1}\right)
\ast\left(  \operatorname*{Ind}\nolimits_{H_{2}}^{G_{2}}W_{2}\right) \\
&  =\operatorname*{Ind}\nolimits_{S_{n_{1}+n_{2}+\cdots+n_{i}}\times
S_{n_{i+1}+n_{i+2}+\cdots+n_{k}}}^{S_{n_{1}+n_{2}+\cdots+n_{k}}}\left(
\left(  \operatorname*{Ind}\nolimits_{H_{1}}^{G_{1}}W_{1}\right)
\otimes\left(  \operatorname*{Ind}\nolimits_{H_{2}}^{G_{2}}W_{2}\right)
\right) \\
&  \ \ \ \ \ \ \ \ \ \ \ \ \ \ \ \ \ \ \ \ \left(  \text{by the definition of
an induction product}\right) \\
&  =\operatorname*{Ind}\nolimits_{G_{1}\times G_{2}}^{S_{n_{1}+n_{2}%
+\cdots+n_{k}}}\underbrace{\left(  \left(  \operatorname*{Ind}\nolimits_{H_{1}%
}^{G_{1}}W_{1}\right)  \otimes\left(  \operatorname*{Ind}\nolimits_{H_{2}%
}^{G_{2}}W_{2}\right)  \right)  }_{\substack{\cong\operatorname*{Ind}%
\nolimits_{H_{1}\times H_{2}}^{G_{1}\times G_{2}}\left(  W_{1}\otimes
W_{2}\right)  \\\text{(by (\ref{pf.prop.indprod.ass.mind}))}}%
}\ \ \ \ \ \ \ \ \ \ \left(
\begin{array}
[c]{c}%
\text{by the definitions}\\
\text{of }G_{1}\text{ and }G_{2}%
\end{array}
\right) \\
&  \cong\operatorname*{Ind}\nolimits_{G_{1}\times G_{2}}^{S_{n_{1}%
+n_{2}+\cdots+n_{k}}}\left(  \operatorname*{Ind}\nolimits_{H_{1}\times H_{2}%
}^{G_{1}\times G_{2}}\left(  W_{1}\otimes W_{2}\right)  \right) \\
&  \cong\operatorname*{Ind}\nolimits_{H_{1}\times H_{2}}^{S_{n_{1}%
+n_{2}+\cdots+n_{k}}}\left(  W_{1}\otimes W_{2}\right)
\ \ \ \ \ \ \ \ \ \ \left(  \text{by an application of
(\ref{pf.prop.indprod.ass.trind})}\right) \\
&  \cong\operatorname*{Ind}\nolimits_{S_{n_{1}}\times S_{n_{2}}\times
\cdots\times S_{n_{k}}}^{S_{n_{1}+n_{2}+\cdots+n_{k}}}\left(  U_{1}\otimes
U_{2}\otimes\cdots\otimes U_{k}\right)
\end{align*}
(since the definitions of $H_{1}$ and $H_{2}$ yield%
\begin{align*}
H_{1}\times H_{2}  &  =\left(  S_{n_{1}}\times S_{n_{2}}\times\cdots\times
S_{n_{i}}\right)  \times\left(  S_{n_{i+1}}\times S_{n_{i+2}}\times
\cdots\times S_{n_{k}}\right) \\
&  \cong S_{n_{1}}\times S_{n_{2}}\times\cdots\times S_{n_{k}}%
\end{align*}
and%
\begin{align*}
W_{1}\otimes W_{2}  &  =\left(  U_{1}\otimes U_{2}\otimes\cdots\otimes
U_{i}\right)  \otimes\left(  U_{i+1}\otimes U_{i+2}\otimes\cdots\otimes
U_{k}\right) \\
&  \cong U_{1}\otimes U_{2}\otimes\cdots\otimes U_{k},
\end{align*}
and both of these isomorphisms are canonical and \textquotedblleft fit
together\textquotedblright\ in that the latter isomorphism respects the action
of the former groups\footnote{Alternatively, you can argue directly: Both
$\operatorname*{Ind}\nolimits_{H_{1}\times H_{2}}^{S_{n_{1}+n_{2}+\cdots
+n_{k}}}\left(  W_{1}\otimes W_{2}\right)  $ and $\operatorname*{Ind}%
\nolimits_{S_{n_{1}}\times S_{n_{2}}\times\cdots\times S_{n_{k}}}%
^{S_{n_{1}+n_{2}+\cdots+n_{k}}}\left(  U_{1}\otimes U_{2}\otimes\cdots\otimes
U_{k}\right)  $ can be written as%
\[
\mathbf{k}\left[  S_{n_{1}+n_{2}+\cdots+n_{k}}\right]  \otimes\left(
U_{1}\otimes U_{2}\otimes\cdots\otimes U_{k}\right)
\]
modulo the relation
\[
\sigma\left(  \sigma_{1}\ast\sigma_{2}\ast\cdots\ast\sigma_{k}\right)
\otimes\left(  u_{1}\otimes u_{2}\otimes\cdots\otimes u_{k}\right)
-\sigma\otimes\left(  \sigma_{1}u_{1}\otimes\sigma_{2}u_{2}\otimes
\cdots\otimes\sigma_{k}u_{k}\right)
\]
for all $\sigma\in S_{n_{1}+n_{2}+\cdots+n_{k}}$ and $\sigma_{i}\in S_{n_{i}}$
and $u_{i}\in U_{i}$.}). In view of%
\[
U_{1}\ast U_{2}\ast\cdots\ast U_{k}=\operatorname*{Ind}\nolimits_{S_{n_{1}%
}\times S_{n_{2}}\times\cdots\times S_{n_{k}}}^{S_{n_{1}+n_{2}+\cdots+n_{k}}%
}\left(  U_{1}\otimes U_{2}\otimes\cdots\otimes U_{k}\right)
\]
(by the definition of the induction product), we can rewrite this as%
\[
\left(  U_{1}\ast U_{2}\ast\cdots\ast U_{i}\right)  \ast\left(  U_{i+1}\ast
U_{i+2}\ast\cdots\ast U_{k}\right)  \cong U_{1}\ast U_{2}\ast\cdots\ast
U_{k}.
\]
This proves Proposition \ref{prop.indprod.ass}.
\end{proof}

\subsection{\label{subsec.apx.proofs.specht.tensA}Proof of Proposition
\ref{prop.specht.tensA}}

\begin{proof}
[Proof of Proposition \ref{prop.specht.tensA}.]The following proof works more
generally when $\mathbf{k}$ is a commutative ring in which $n!$ is invertible.

We shall use the notations of \cite[Chapter 5]{sga}. Pick any $n$-tableau $P$
of shape $\lambda$ and any $n$-tableau $Q$ of shape $\mu$. Consider the
corresponding Young symmetrizers $\mathbf{E}_{P}$ and $\mathbf{E}_{Q}$ (as
defined in \cite[Definition 5.11.1]{sga}). From \cite[(228)]{sga}, we have
$\mathcal{S}^{\lambda}\cong\mathcal{A}\mathbf{E}_{P}$ and $\mathcal{S}^{\mu
}\cong\mathcal{A}\mathbf{E}_{Q}$. Hence, $\operatorname*{Hom}%
\nolimits_{\mathcal{A}}\left(  \mathcal{S}^{\lambda},\mathcal{S}^{\mu}\right)
\cong\operatorname*{Hom}\nolimits_{\mathcal{A}}\left(  \mathcal{A}%
\mathbf{E}_{P},\mathcal{A}\mathbf{E}_{Q}\right)  $ (since $\operatorname*{Hom}%
\nolimits_{\mathcal{A}}$ is functorial).

From \cite[Lemma 5.11.13]{sga}, we know that the coefficient of the
permutation $\operatorname*{id}\in S_{n}$ in $\mathbf{E}_{P}$ is $1$. Thus,
the single element $\mathbf{E}_{P}$ is $\mathbf{k}$-linearly independent.
Hence, $\mathbf{kE}_{P}\cong\mathbf{k}$ as $\mathbf{k}$-module.

However, \cite[Theorem 5.11.3]{sga} shows that $\mathbf{E}_{P}^{2}=\dfrac
{n!}{f^{\lambda}}\mathbf{E}_{P}$. This shows that the element
$\widetilde{\mathbf{E}}_{P}:=\dfrac{f^{\lambda}}{n!}\mathbf{E}_{P}%
\in\mathcal{A}$ is idempotent\footnote{\textit{Proof.} From
$\widetilde{\mathbf{E}}_{P}=\dfrac{f^{\lambda}}{n!}\mathbf{E}_{P}$, we obtain
$\widetilde{\mathbf{E}}_{P}^{2}=\left(  \dfrac{f^{\lambda}}{n!}\mathbf{E}%
_{P}\right)  ^{2}=\left(  \dfrac{f^{\lambda}}{n!}\right)  ^{2}%
\underbrace{\mathbf{E}_{P}^{2}}_{=\dfrac{n!}{f^{\lambda}}\mathbf{E}_{P}%
}=\left(  \dfrac{f^{\lambda}}{n!}\right)  ^{2}\dfrac{n!}{f^{\lambda}%
}\mathbf{E}_{P}=\dfrac{f^{\lambda}}{n!}\mathbf{E}_{P}=\widetilde{\mathbf{E}%
}_{P}$. In other words, $\widetilde{\mathbf{E}}_{P}$ is idempotent.}.
Moreover, $\widetilde{\mathbf{E}}_{P}$ is a nonzero scalar multiple of
$\mathbf{E}_{P}$ (since the scalar $\dfrac{f^{\lambda}}{n!}$ is clearly
nonzero). Thus, $\mathcal{A}\mathbf{E}_{P}=\mathcal{A}\widetilde{\mathbf{E}%
}_{P}$.

But \cite[Lemma 5.13.4]{Etingof-et-al} says that if $A$ is a $\mathbf{k}%
$-algebra and $e\in A$ is an idempotent, then $\operatorname*{Hom}%
\nolimits_{A}\left(  Ae,M\right)  \cong eM$ for any left $A$-module $M$.
Applying this to $A=\mathcal{A}$ and $e=\widetilde{\mathbf{E}}_{P}$ and
$M=\mathcal{A}\mathbf{E}_{Q}$, we find%
\[
\operatorname*{Hom}\nolimits_{\mathcal{A}}\left(  \mathcal{A}%
\widetilde{\mathbf{E}}_{P},\mathcal{A}\mathbf{E}_{Q}\right)  \cong%
\widetilde{\mathbf{E}}_{P}\mathcal{A}\mathbf{E}_{Q}=\mathbf{E}_{P}%
\mathcal{A}\mathbf{E}_{Q}%
\]
(since $\widetilde{\mathbf{E}}_{P}$ is a nonzero scalar multiple of
$\mathbf{E}_{P}$). Altogether,%
\begin{align}
\operatorname*{Hom}\nolimits_{\mathcal{A}}\left(  \mathcal{S}^{\lambda
},\mathcal{S}^{\mu}\right)   &  \cong\operatorname*{Hom}\nolimits_{\mathcal{A}%
}\left(  \underbrace{\mathcal{A}\mathbf{E}_{P}}_{=\mathcal{A}%
\widetilde{\mathbf{E}}_{P}},\mathcal{A}\mathbf{E}_{Q}\right)
=\operatorname*{Hom}\nolimits_{\mathcal{A}}\left(  \mathcal{A}%
\widetilde{\mathbf{E}}_{P},\mathcal{A}\mathbf{E}_{Q}\right) \nonumber\\
&  \cong\mathbf{E}_{P}\mathcal{A}\mathbf{E}_{Q}.
\label{pf.prop.specht.tensA.all}%
\end{align}

Now, if $\lambda\neq\mu$, then \cite[Proposition 5.11.15]{sga} (applied to
$S=P$ and $T=Q$) shows that $\mathbf{E}_{P}\mathbf{aE}_{Q}=0$ for all
$\mathbf{a}\in\mathcal{A}$, and thus $\mathbf{E}_{P}\mathcal{A}\mathbf{E}%
_{Q}=0$. Hence, if $\lambda\neq\mu$, then (\ref{pf.prop.specht.tensA.all})
becomes%
\begin{equation}
\operatorname*{Hom}\nolimits_{\mathcal{A}}\left(  \mathcal{S}^{\lambda
},\mathcal{S}^{\mu}\right)  \cong\mathbf{E}_{P}\mathcal{A}\mathbf{E}_{Q}=0.
\label{pf.prop.specht.tensA.distinct}%
\end{equation}

This proves Proposition \ref{prop.specht.tensA} in the case when $\lambda
\neq\mu$.

Thus, we now WLOG assume that $\lambda=\mu$. Hence, $\mu=\lambda$ and thus
$\mathcal{S}^{\mu}=\mathcal{S}^{\lambda}\cong\mathcal{A}\mathbf{E}_{P}$.
Hence, the same argument that we used to prove (\ref{pf.prop.specht.tensA.all}%
) can be applied to $\lambda$ and $P$ instead of $\mu$ and $Q$. This results
in%
\begin{equation}
\operatorname*{Hom}\nolimits_{\mathcal{A}}\left(  \mathcal{S}^{\lambda
},\mathcal{S}^{\lambda}\right)  \cong\mathbf{E}_{P}\mathcal{A}\mathbf{E}_{P}.
\label{pf.prop.specht.tensA.7}%
\end{equation}
However, \cite[Proposition 5.11.5]{sga} (applied to $T=P$) shows that each
$\mathbf{a}\in\mathcal{A}$ satisfies $\mathbf{E}_{P}\mathbf{aE}_{P}%
=\kappa\mathbf{E}_{P}$ for some $\kappa\in\mathbf{k}$. In other words, each
$\mathbf{a}\in\mathcal{A}$ satisfies $\mathbf{E}_{P}\mathbf{aE}_{P}%
\in\mathbf{kE}_{P}$. In other words, $\mathbf{E}_{P}\mathcal{A}\mathbf{E}%
_{P}\subseteq\mathbf{kE}_{P}$. On the other hand, each $\kappa\in\mathbf{k}$
satisfies
\[
\kappa\mathbf{E}_{P}=\dfrac{\kappa f^{\lambda}}{n!}\cdot\underbrace{\dfrac
{n!}{f^{\lambda}}\mathbf{E}_{P}}_{\substack{=\mathbf{E}_{P}^{2}=\mathbf{E}%
_{P}1\mathbf{E}_{P}\\\in\mathbf{E}_{P}\mathcal{A}\mathbf{E}_{P}}}\in
\dfrac{\kappa f^{\lambda}}{n!}\mathbf{E}_{P}\mathcal{A}\mathbf{E}_{P}%
\subseteq\mathbf{E}_{P}\mathcal{A}\mathbf{E}_{P}.
\]
Thus, $\mathbf{kE}_{P}\subseteq\mathbf{E}_{P}\mathcal{A}\mathbf{E}_{P}$.
Combining this with $\mathbf{E}_{P}\mathcal{A}\mathbf{E}_{P}\subseteq
\mathbf{kE}_{P}$, we obtain%
\[
\mathbf{E}_{P}\mathcal{A}\mathbf{E}_{P}=\mathbf{kE}_{P}.
\]
Hence, (\ref{pf.prop.specht.tensA.7}) becomes%
\[
\operatorname*{Hom}\nolimits_{\mathcal{A}}\left(  \mathcal{S}^{\lambda
},\mathcal{S}^{\lambda}\right)  \cong\mathbf{E}_{P}\mathcal{A}\mathbf{E}%
_{P}=\mathbf{kE}_{P}\cong\mathbf{k}.
\]
In other words, $\operatorname*{Hom}\nolimits_{\mathcal{A}}\left(
\mathcal{S}^{\lambda},\mathcal{S}^{\mu}\right)  \cong\mathbf{k}$ if
$\lambda=\mu$. Thus, Proposition \ref{prop.specht.tensA} is proved in the case
$\lambda=\mu$ as well.
\end{proof}

\begin{noncompile}
\begin{proof}
[Old proof of Proposition \ref{prop.specht.tensA}.]Recall (e.g., from
\cite[Theorem 5.12.2 and Corollary 5.12.4]{Etingof-et-al}\footnote{We are
using the fact that the Specht module $\mathcal{S}^{\lambda}$ is isomorphic to
the $V_{\lambda}$ in \cite{Etingof-et-al}. This holds by \cite[Theorem 5.5.13
\textbf{(b)}]{sga}.}) that the Specht modules $\mathcal{S}^{\nu}$ for the
partitions $\nu$ of $n$ are pairwise non-isomorphic split irreducible $S_{n}%
$-representations. Hence, the hom-space $\operatorname*{Hom}%
\nolimits_{\mathcal{A}}\left(  \mathcal{S}^{\lambda},\mathcal{S}^{\mu}\right)
$ is given by%
\[
\operatorname*{Hom}\nolimits_{\mathcal{A}}\left(  \mathcal{S}^{\lambda
},\mathcal{S}^{\mu}\right)  \cong%
\begin{cases}
\mathbf{k}, & \text{if }\lambda=\mu;\\
0, & \text{if }\lambda\neq\mu
\end{cases}
\ \ \ \ \ \ \ \ \ \ \text{as }\mathbf{k}\text{-vector spaces.}%
\]
This proves Proposition \ref{prop.specht.tensA}.
\end{proof}
\end{noncompile}

\bibliographystyle{halpha-abbrv}
\bibliography{biblio.bib}

\end{document}